\RequirePackage{fix-cm}
\documentclass[smallextended]{svjour3}       
\smartqed  
\usepackage{graphicx}

\usepackage{bbm}
\usepackage{amsmath}
\pdfoutput=1
\usepackage{amsfonts}
\usepackage{amssymb}
\usepackage{epsfig}
\usepackage{slashbox}
\usepackage{latexsym}
\usepackage[margin=1in]{geometry}
\usepackage{algpseudocode,algorithm}

\newcommand{\adots}{\mathinner{\mkern1mu\raise1pt\vbox{\kern1pt\hbox{.}}
\mkern2mu\raise4pt\hbox{.}\mkern2mu\raise7pt\hbox{.}\mkern1mu}}

\makeatletter
\def\@setcopyright{}
\def\serieslogo@{}
\makeatother
\begin{document}

\title{\bf  Inverting Incomplete Fourier Transforms by a Sparse Regularization Model and Applications in Seismic  Wavefield Modeling
\thanks{
T. Wu was supported in part by the Natural Science Foundation of Shandong Province of China under grants ZR2021MA049, ZR2020MA031, and Shandong Province Higher Educational Science and Technology Program of China under grant J18KA221. Y. Xu was supported in part by the US National Science Foundation under grant DMS-1912958 and by the US National Institutes of Health under grant R21CA263876. The datasets generated and analysed during the current study are not publicly available but are available from the corresponding author on reasonable request.}}

\titlerunning{Inverting Incomplete FT by a Sparse Regularization Model and Applications in SW Modeling}        

\author{Tingting Wu\ \ and Yuesheng Xu}


\institute{T. Wu\\
 School of Mathematics and Statistics, Shandong Normal University, Jinan 250358, P.R. China. tingtingwu@sdnu.edu.cn.
           \\
           Y. Xu\\
            Department of Mathematics and Statistics, Old Dominion University, Norfolk, VA 23529, USA.  y1xu@odu.edu. All correspondence should be sent to this author.
}

\date{Received: date / Accepted: date}

\maketitle

\begin{abstract}
We propose a sparse regularization model for inversion of incomplete Fourier transforms and apply it to seismic wavefield modeling. The objective function of the proposed model employs the Moreau envelope of the $\ell_0$ norm under a tight framelet system as a regularization to promote sparsity. This model leads to a non-smooth, non-convex optimization problem for which traditional iteration schemes are inefficient or even divergent. By exploiting special structures of the $\ell_0$ norm,  we identify a local minimizer of the proposed non-convex optimization problem with a global minimizer of a convex optimization problem, which provides us insights for the development of efficient and convergence guaranteed algorithms to solve it. We characterize the solution of the regularization model in terms of a fixed-point of a map defined by the proximity operator of the $\ell_0$ norm  and develop a fixed-point iteration algorithm to solve it. By connecting the map with an $\alpha$-averaged nonexpansive operator, we prove that the sequence generated by the proposed fixed-point proximity algorithm converges to a local minimizer of the proposed model. Our numerical examples confirm that the proposed model outperforms significantly the existing model based on the $\ell_1$-norm. The seismic wavefield modeling in the frequency domain requires solving a series of the Helmholtz equation with large wave numbers, which is a computationally intensive task. Applying the proposed sparse regularization model to the seismic wavefield modeling requires data of only a few low frequencies, avoiding solving the Helmholtz equation with large wave numbers. This makes the proposed model particularly suitable for the seismic wavefield (SW) modeling. Numerical results show that the proposed method performs better than the existing method based on the $\ell_1$ norm in terms of the SNR values and visual quality of the restored synthetic seismograms.
\subclass{MSC 65.}
\end{abstract}

\section{Introduction}

The aim of this study is to develop a sparse regularization model for inverting incomplete Fourier transforms and an efficient, convergence guaranteed iteration algorithm
for solving the resulting non-convex, non-smooth optimization problem. Moreover, we apply the developed method to seismic wavefield modeling.
Incomplete Fourier transforms arise in many engineering problems \cite{Brigham,H3,Riyanti-Kononov-Erlanggga-Vuik,WSX}. They are of special interest as reconstructing a digital signal or image from incomplete Fourier data has important applications in biomedical imaging (MRI and tomography), astrophysics (interferometric imaging), and geophysical exploration. Reconstruction of a digital signal or image from incomplete Fourier transform data is an ill-posed problem, which often produces aliasing artifacts due to  vast undersampling and distortion which means that the reconstructed signal is not like what we expect. Therefore, it is crucial to develop an effective inversion model which alleviates the artifacts and distortion, and design an efficient algorithm to solve the resulting optimization problem.

To overcome the difficulty caused by incomplete data, inverting incomplete Fourier transforms has been investigated in the context of sparse signal/image processing.
Compressed sensing \cite{Candes-Romberg-Tao1,Candes-Romberg-Tao2,Donoho} was used in  \cite{Candes-Romberg-Tao1,Lin_Herrmann,LDP}  to invert incomplete Fourier transforms. Specifically, the paper \cite{Candes-Romberg-Tao1} applied the $\ell_1$ norm as a regularization to reconstruct an object from randomly chosen incomplete
frequency samples. In \cite{Lin_Herrmann,Lin_Lebed_Erlangga}, the problem of inverting incomplete Fourier transforms is also considered in forward wavefield extrapolation. While in \cite{LDP}, the compressed sensing method was applied for rapid MR imaging, which employed an $\ell_1$-norm model to invert incomplete Fourier transforms. We also developed a sparse regularization method in \cite{WSX} for inverting incomplete Fourier transforms.
Both the compressed sensing method and the sparse regularization method employ the $\ell_1$ norm as a regularization to impose sparsity for the reconstructed signal under certain transforms. Because the $\ell_1$-norm based models are convex, they can be solved efficiently by available tools \cite{CSXZ-12,Gold-Osher,Krol-Li-Shen-Xu,Li-Shen-Xu-Zhang,MSX,Micchelli-Shen-Xu-Zengprox2}.  However, according to \cite{Fan-Li2001}, the $\ell_1$-norm based models can lead to outliers and thus, there is a need to develop more effective, robust models.

The main purpose of this research is to propose a model which can reduce both artifacts and outliers in the reconstructed signal and can be efficiently solved.
To this end, we propose to use the Moreau envelope of the $\ell_0$ norm as a sparsity promoting function as a regularization. That is, we will invert incomplete Fourier transforms with a sparsity penalty, under a framelet transform, of the envelope of the $\ell_0$ norm. Note that the sparsity of a vector is originally measured by the number of its nonzero components, namely, the $\ell_0$ norm of the vector. However, the $\ell_0$ norm is discontinuous at the origin, which is not favorable from a computational viewpoint. The envelope of the $\ell_0$ norm is a continuous surrogate of the $\ell_0$ norm. Although the $\ell_0$ norm is non-convex, according to \cite{Xu},  due to the special structure of the $\ell_0$ norm, a local minimizer of a function that is the sum of a convex function and the $\ell_0$ norm can be identified with a global minimizer of the convex function. This fact provides us with great convenience for algorithmic development of optimization problems of this type. The use of the $\ell_0$ norm enables us to formulate a sparsity regularization model, for inverting incomplete Fourier transforms, which can reduce artifacts and outliers in the reconstructed signal, and allow us to design an efficient fixed-point iteration algorithm for the resulting non-convex, non-smooth optimization problem. Moreover, by exploiting the connection of this minimization problem with the related {\it convex} minimization problem, we are able to establish convergence of the proposed fixed-point algorithm.

The second component of this paper is to apply the developed method for inverting incomplete Fourier transforms to analyzing seismic wavefield in the frequency domain.
It is well-known \cite{Lin_Herrmann} that seismic wavefield can be analyzed by inverting Fourier transforms. In this approach, we need to solve the Helmholtz equation with wave numbers that correspond to Fourier frequencies. However, this approach has a major drawback: A high Fourier frequency corresponds to a large wave number and the numerical solution of the Helmholtz equation with a large wave number is a challenging task due to the high oscillation in its solution \cite{B2}. The proposed incomplete Fourier transform inversion method suggests that we can analyze seismic wavefield without solving the Helmholtz equation with large wave numbers. That is, using only low Fourier frequencies, we can obtain satisfactory reconstruction results by employing the developed inversion method. Therefore, the proposed incomplete Fourier transform inversion method makes the seismic wavefield modeling in the frequency domain a feasible approach.

We organize the paper in eight sections. In Section \ref{sec:MAM2}, the envelope of the $\ell_0$ norm is employed to construct a sparse regularization model for inverting
incomplete Fourier transforms. For the proposed regularization model,  an equivalent model is then presented by considering properties of the envelope of the $\ell_0$ norm. Section \ref{sec:LConvexity} is devoted to an investigation of a local convexity of the proposed env-$\ell_0$ regularization model which is by nature a non-convex minimization problem. We propose a fixed-point iterative algorithm for solving the resulting  non-convex minimization problem in Section \ref{sec:MAM_alro}, and establish its convergence theorem in Section \ref{sec:MAM_conv}. Section \ref{sec:MAM4} considers applications of the proposed inversion method of incomplete Fourier transforms in seismic wavefield modeling in the frequency domain. Numerical examples are presented in Section \ref{sec:MAM5} to validate the effectiveness, robustness and efficiency of the proposed methods. Finally, Section \ref{sec:Conclusion} concludes this paper.

\section{A Sparse Regularization Model}\label{sec:MAM2}\setcounter{equation}{0}

In this section, we propose a sparse regularization model using the envelope of the $\ell_0$ norm for inverting incomplete Fourier transforms. By employing properties of the envelope of the $\ell_0$ norm, we derive an equivalent model for the purpose of algorithmic development.



We first describe incomplete Fourier transforms under consideration. By $\mathbf{F}$ we denote an $M \times M$ discrete Fourier transform (DFT) matrix with the $(m,n)$-th entry given by
$$
\mathbf{F}_{mn}:=\frac{1}{\sqrt{M}}\exp\left(-i\frac{2\pi
(m-1)(n-1)}{M}\right).
$$
Suppose that $\mathbf{r}\in\mathbb{R}^d$ is a given vector in the Fourier frequency domain, where $d$ is a positive integer such that $d<M$. As a convention through out the paper, we assume that all vectors are column vectors unless stated otherwise.  Let $\mathbf{R}$ denote  a $d \times M$ ``row selector'' matrix, a row submatrix of the identity matrix $\mathbf{I}$ by selecting certain rows of $\mathbf{I}$. In fact, for a positive integer $m$ with $2\leq m\leq M$,  if the $m$-th row of $\mathbf{I}$ is included in $\mathbf{R}$, then  $\mathbf{R}$ also includes the $\left(M-m+2\right)$-th row of $\mathbf{I}$. This is a reasonable choice, as   $\mathbf{F}_{m,j}$ and $\mathbf{F}_{M-m+2,j}$ are mutually conjugate for $2\leq m\leq M$ and $1\leq j\leq M$. For a row selector $\mathbf{R}$, $\mathbf{RF}$ is an incomplete Fourier transform. Inverting the incomplete Fourier transform is to find a vector $\mathbf{v}\in\mathbb{R}^M$ such that
\begin{equation}
\label{seismogram-model-origin-app}
 \mathbf{RF}\mathbf{v}=
\mathbf{r}.
\end{equation}

It is known that a solution $\mathbf{v}\in \mathbb{R}^M$ of equation \eqref{seismogram-model-origin-app} may be sparse under a transform \cite{Lebed_Herrmann}.  We wish to reconstruct a solution $\mathbf{v}\in \mathbb{R}^M$ of equation \eqref{seismogram-model-origin-app} which is sparse under a tight framelet transform.
%
To this end, for a proper tight framelet matrix $\mathbf{W}$ of size $N \times M$, we define
\begin{eqnarray}
\mathbf{y}\ := \mathbf{W}\mathbf{v} \quad \mbox{and} \quad  \mathbf{K}\ :=\mathbf{RFW}^{*},
\label{def:y_K}
\end{eqnarray}
where $\mathbf{A}^{*}$ denotes the conjugate transpose of a matrix $\mathbf{A}$. Then model~\eqref{seismogram-model-origin-app} becomes
\begin{equation}\label{Inverse_fourier_problem-new}
\mathbf{K} \mathbf{y} = \mathbf{r}.
\end{equation}
Note that the vector $\mathbf{y}$ is the transform of $\mathbf{v}$ under the framelet matrix $ \mathbf{W}$. Inverting equation \eqref{Inverse_fourier_problem-new} is an ill-posed problem, which requires proper regularization.

Our next task is to describe the sparse regularization model for ``inverting" equation \eqref{Inverse_fourier_problem-new} to obtain a sparse vector $\mathbf{y}$. The sparsity of a vector is naturally measured by the $\ell_0$ ``norm'' which counts the number of nonzero components of the vector. Specifically, for an $a\in\mathbb{R}$, we let $|a|_0:=1$ if $a\neq 0$, and  $|a|_0:=0$ if $a=0$. The $\ell_0$ norm of $\mathbf{x}\in\mathbb{R}^N$ is defined by
$$
\|\mathbf{x}\|_0~:=\sum_{i=1}^N|x_i|_0.
$$
Even though $\|\cdot\|_0$ is not a norm, traditionally it is called the $\ell_0$ norm in the signal processing community. We will follow the tradition to call it the $\ell_0$ norm through out this paper.
The $\ell_0$ norm is non-convex and discontinuous at the origin, which causes computational difficulties. To overcome the difficulties, we adopt a continuous approximation of the  $\ell_0$ norm by its Moreau envelope. According to  \cite{moreau:RASPS:62,SXZ}, for a positive number $\beta$, the Moreau envelope of $\|\cdot\|_0$ with index $\beta$ at $\mathbf{x}\in \mathbb{R}^{N}$ is defined by
\begin{equation}
\label{Eq_envelope}
\mathrm{env}_{\beta\|\cdot\|_0}(\mathbf{x})~:=\min\left\{\frac{1}{2\beta}\|\mathbf{x}-\mathbf{z}\|_2^2+\|\mathbf{z}\|_0: \mathbf{z}\in\mathbb{R}^N\right\}.
\end{equation}
A direct computation leads to
$$
 \mathrm{env}_{\beta\|\cdot\|_0}(\mathbf{x})=\sum_{i=1}^{N}\varphi(x_i),
$$
where
$$
\varphi(x_i)~:= \left\{
\begin{array}{ll}
     1, & \hbox{$|x_i|\geq \sqrt{2\beta}$;} \\
   \frac{1}{2\beta}x_i^2, & \hbox{otherwise.}
\end{array}
\right.
$$
Clearly, $\mathrm{env}_{\beta\|\cdot\|_0}$ is continuous and locally convex near the origin. Moreover, as $\beta\rightarrow 0^{+}$, $\mathrm{env}_{\beta\|\cdot\|_0}\rightarrow\|\cdot\|_0$. Therefore, when $\beta$ is small enough,  $\mathrm{env}_{\beta\|\cdot\|_0}$ is a good approximation of $\|\cdot\|_0$. With an appropriate choice of the parameter $\beta$, $\mathrm{env}_{\beta\|\cdot\|_0}$ can be used as a measure of sparsity and at the same time we avoid drawbacks of $\|\cdot\|_0$.
 For $(\mathbf{r},\mathbf{y})\in\mathbb{R}^{d}\times\mathbb{R}^{N}$, we let
\begin{eqnarray}\label{Def-Q}
Q(\mathbf{r},\mathbf{y})=\frac{1}{2}\|\mathbf{K}\mathbf{y}-\mathbf{r}\|_2^2
+\gamma {\rm env}_{\beta\|\cdot\|_0}\left(\mathbf{y}\right),
\end{eqnarray}
where $\gamma$ is a positive parameter.
As $\mathbf{r}\in\mathbb{R}^{d}$ in inverting incomplete Fourier transform  \eqref{seismogram-model-origin-app} is fixed,  we write $Q(\mathbf{r},\mathbf{y})$ as $Q(\mathbf{y})$ for convenient presentation.
We now propose the sparse regularization model using the Moreau envelope  of the $\ell_0$ norm to recover a sparse vector $\mathbf{y}$ from (\ref{Inverse_fourier_problem-new})
\begin{equation}\label{model:general}
\mathbf{y}^\star ={\rm argmin}\left\{Q\left(\mathbf{y}\right): \mathbf{y}\in \mathbb{R}^{N}\right\}.
\end{equation}
Since $\mathrm{env}_{\beta\|\cdot\|_0}$ is an approximation of $\|\cdot\|_0$, we expect that the proposed model enjoys most advantages of the $\ell_0$ norm-based model, while it may be solved by efficient algorithms.


Model \eqref{model:general} may be reinterpreted from a {\it Bayesian} viewpoint \cite{Stuart2010}. In Bayesian statistics, a maximum a posteriori probability (MAP) estimate is an estimate of an unknown quantity, which equals to the mode of the posterior distribution.
To derive model \eqref{model:general} from the Bayesian approach, we treat variables $\mathbf{y}$ and $\mathbf{r}$ in equation \eqref{Inverse_fourier_problem-new} as random variables. Let $\mathbf{1}$ denote the vector of the same size as $\mathbf{r}$, with its all components being 1. We assume that the known data $\mathbf{r}$ related to the unknown distribution  $\mathbf{y}$ can be approximated by the following model
\begin{eqnarray}
\label{Bay_normal}
\mathbf{r}
={\rm Normal}\left(\mathbf{K} \mathbf{y},\sigma^2\mathbf{1}\right),
\end{eqnarray}
where ${\rm Normal}\left(\mathbf{K} \mathbf{y},\sigma^2\mathbf{1}\right)$ denotes a Normal distributed random  vector with mean $\mathbf{K} \mathbf{y}$ and variation $\sigma^2\mathbf{1}$. In model \eqref{Bay_normal} we use the Normal distribution since the Gaussian noise is the most common noise that we encounter in applications.
The MAP estimate $\mathbf{y}^{\star}$ is obtained by maximizing the conditional a
  $posteriori$ probability $p(\mathbf{y}|\mathbf{r})$, the probability that $\mathbf{y}$ occurs when $\mathbf{r}$ is observed. This probability
may be computed using the Bayes law:
\begin{eqnarray}
\label{Bay_formula}
p(\mathbf{y}|\mathbf{r})\propto
p(\mathbf{r}|\mathbf{y})p(\mathbf{y}),
\end{eqnarray}
where the notation ``${x}\propto{z}$'' means that the scalar ${x}$ is proportional to the scalar ${z}$,
and $p(\mathbf{r}|\mathbf{y})$ denotes the conditional probability that $\mathbf{r}$ occurs when $\mathbf{y}$ is known. By taking the logarithm
of both sides of equation \eqref{Bay_formula}, the MAP estimate can then be calculated using the formula
\begin{eqnarray}
\label{Bay_model}
\mathbf{y}^\star ={\rm argmax}\left\{\ln p(\mathbf{r}|\mathbf{y})+\ln p(\mathbf{y}): \mathbf{y}\in \mathbb{R}^{N}\right\}.
\end{eqnarray}
The first term can
be considered as a fidelity term, a measure of the discrepancy between the estimated and the
observed data. The second term is a regularization function, which penalizes solutions that
have low probability. We then compute the two terms in model \eqref{Bay_model}. According
to equation \eqref{Bay_normal}, $\mathbf{r}$ follows the Normal distribution with $\mathbf{K} \mathbf{y}$ as its mean and $\sigma^2\mathbf{1}$ as its variation. As a result,
the probability density function $p(\mathbf{r}|\mathbf{y})$ of $\mathbf{r}$ conditioned on $\mathbf{y}$ can be computed by using the
formula
\begin{eqnarray}
\label{P_r_y}
p(\mathbf{r}|\mathbf{y})=\prod_{i=1}^{N}
\frac{1}{\sqrt{2\pi}\sigma}
\exp\left(-\frac{\left(\left(\mathbf{K} \mathbf{y}\right)_i-r_i\right)^2}{2\sigma^2}\right).
\end{eqnarray}
Taking the logarithm of both sides of equation \eqref{P_r_y} yields
\begin{eqnarray}
\label{ln_P_r_y}
\ln p(\mathbf{r}|\mathbf{y})={\rm const}_1-\frac{1}{2\sigma^2}
\|\mathbf{K}\mathbf{y}-\mathbf{r}\|_2^2,
\end{eqnarray}
where ${\rm const}_1$ is a constant independent of $\mathbf{y}$.
To compute the second term in model  \eqref{Bay_model}, the Gibbs prior \cite{Lalush_Tsui} is used
\begin{eqnarray}
\label{p_y_formula}
p(\mathbf{y})\propto
\exp\left(-\widetilde{\gamma}\vartheta(\mathbf{y})\right),
\end{eqnarray}
with the Gibbs real-valued energy function $\vartheta(\mathbf{y})$ defined on $\mathbb{R}^{N}$ and a positive regularization
parameter $\widetilde{\gamma}$ called hyperparameter. For the purpose of promoting sparsity of the estimated solution, we may choose the energy function $\vartheta(\mathbf{y})$ in \eqref{p_y_formula} as a sparsity promoting norm of $\mathbf{y}$, such as $\|\mathbf{y}\|_1$ \cite{Donoho,WSX}. As we have discussed earlier,  $\mathrm{env}_{\beta\|\cdot\|_0}$ is an excellent sparsity promoting function. Hence, in this paper we adopt
 \begin{eqnarray}
 \label{vartheta}
 \vartheta(\mathbf{y})=\mathrm{env}_{\beta\|\cdot\|_0}(\mathbf{y}).
 \end{eqnarray}
From \eqref{p_y_formula} and \eqref{vartheta}, we have that
\begin{eqnarray}
\label{ln_p_y}
\ln p(\mathbf{y})=-\widetilde{\gamma}
\mathrm{env}_{\beta\|\cdot\|_0}(\mathbf{y})+{\rm const}_2,
\end{eqnarray}
where ${\rm const}_2$ is again a constant independent of $\mathbf{y}$.
Letting $\gamma:~=\sigma^2\widetilde{\gamma}$ and substituting  equations \eqref{ln_P_r_y}, \eqref{ln_p_y} into \eqref{Bay_model} lead to model \eqref{model:general}.


We reformulate the proposed model \eqref{model:general} so that the resulting model is convenient for computation. Motivated by the definition \eqref{Eq_envelope} of $\mathrm{env}_{\beta\|\cdot\|_0}$ in the non-convex model \eqref{model:general}, we introduce the following function
\begin{eqnarray}
F(\mathbf{r},\mathbf{x},\mathbf{y})~:=
\frac{1}{2}\|\mathbf{K}\mathbf{y}-\mathbf{r}\|_2^2
+\frac{\gamma}{2\beta}\|\mathbf{x}-\mathbf{y}\|_2^2
+\gamma\|\mathbf{x}\|_0,
\qquad (\mathbf{r},\mathbf{x},\mathbf{y})\in \mathbb{R}^{d}\times\mathbb{R}^{N}\times \mathbb{R}^{N}.
\label{def:F}
\end{eqnarray}
Again since $\mathbf{r}$ is fixed,  we write $F(\mathbf{r},\mathbf{x},\mathbf{y})$ as $F(\mathbf{x},\mathbf{y})$. The non-convex function $F(\mathbf{x},\mathbf{y})$ is a special case of those considered in \cite{Xu}.
We then consider the model
\begin{equation}\label{model:two_variable}
(\mathbf{x}^{\star},\mathbf{y}^{\star})={\rm argmin}\left\{
F(\mathbf{x},\mathbf{y}),\quad (\mathbf{x},\mathbf{y})\in \mathbb{R}^{N}\times \mathbb{R}^{N}\right\}.
\end{equation}

We next show that models \eqref{model:general} and \eqref{model:two_variable} are essentially equivalent. A global minimizer of any of these models will also be called a solution of the model. We first present a relation between $Q(\mathbf{y})$ and $F(\mathbf{x}, \mathbf{y})$. We recall the proximity operator of the $\ell_0$ norm. For $\beta>0$, the proximity operator of  $\|\cdot\|_0$ at $\mathbf{z}\in \mathbb{R}^{N}$ is defined by \begin{eqnarray}
\mathrm{prox}_{\beta\|\cdot\|_0} \left(\mathbf{z}\right)
~:={\rm argmin}\left\{\frac{1}{2\beta}\|\mathbf{x}-\mathbf{z}\|_2^2+\|\mathbf{x}\|_0: \mathbf{x}\in\mathbb{R}^N\right\},
\label{def:prox_beta_ell0}
\end{eqnarray}
see, for example, \cite{MSX}. Clearly, if $\mathbf{x}\in \mathrm{prox}_{\beta\|\cdot\|_0} \left(\mathbf{z}\right)$, then we have that
\begin{equation}\label{env-prox}
    \mathrm{env}_{\beta\|\cdot\|_0} \left(\mathbf{z}\right)=\frac{1}{2\beta}\|\mathbf{x}-\mathbf{z}\|^2_2+\|\mathbf{x}\|_0.
\end{equation}
By equation \eqref{env-prox}, we observe that
\begin{eqnarray}
Q(\mathbf{y})=F(\mathbf{x},\mathbf{y}),\ \ \ \ \ {\rm for\ all}
\ \mathbf{x}\in\mathrm{prox}_{\beta\|\cdot\|_0} \left(\mathbf{y}\right) \ {\rm and\ for\ all}\ \mathbf{y}\in\mathbb{R}^{N}.
\label{Q_F_relation}
\end{eqnarray}
The next proposition modifies Proposition 1 of \cite{SXZ} for the current setting.

\begin{proposition}\label{pro:model_relation}
Let  $\beta>0$ and $\gamma>0$. A pair $(\mathbf{x}^{\star},\mathbf{y}^{\star})$ is a
solution of model~\eqref{model:two_variable} if and only if $\mathbf{y}^{\star}$
is a solution of model \eqref{model:general} with $\mathbf{x}^{\star}$ satisfying the inclusion relation
\begin{eqnarray}
\mathbf{x}^{\star}\in\mathrm{prox}_{\beta\|\cdot\|_0} \left(\mathbf{y}^{\star}\right).
\label{xstar_ystar_re1}
\end{eqnarray}
\end{proposition}
\begin{proof}
Suppose that a pair $(\mathbf{x}^{\star},\mathbf{y}^{\star})$ is a solution of model \eqref{model:two_variable}.
We first establish the inclusion relation
\eqref{xstar_ystar_re1}.
Since $(\mathbf{x}^{\star},\mathbf{y}^{\star})$ is a solution of the minimization problem  \eqref{model:two_variable}, we have for all  $(\mathbf{x},\mathbf{y})\in \mathbb{R}^{N}\times \mathbb{R}^{N}$ that
\begin{eqnarray}
\frac{1}{2}\|\mathbf{K}\mathbf{y}^{\star}-\mathbf{r}\|_2^2
+\frac{\gamma}{2\beta}\|\mathbf{x}^{\star}-\mathbf{y}^{\star}\|_2^2
+\gamma\|\mathbf{x}^{\star}\|_0
\leq
\frac{1}{2}\|\mathbf{K}\mathbf{y}-\mathbf{r}\|_2^2
+\frac{\gamma}{2\beta}\|\mathbf{x}-\mathbf{y}\|_2^2
+\gamma\|\mathbf{x}\|_0.
\label{equ:F_global_solution_nec}
\end{eqnarray}
In particular, inequality \eqref{equ:F_global_solution_nec} holds for all $\mathbf{x}\in \mathbb{R}^{N}$ with $\mathbf{y}:=\mathbf{y}^{\star}$.
The resulting inequality together with $\gamma>0$ yields
\begin{eqnarray*}
\frac{1}{2\beta}\|\mathbf{x}^{\star}-\mathbf{y}^{\star}\|_2^2
+\|\mathbf{x}^{\star}\|_0
\leq
\frac{1}{2\beta}\|\mathbf{x}-\mathbf{y}^{\star}\|_2^2
+\|\mathbf{x}\|_0,
\ \ {\rm for \ all}\  \mathbf{x} \in \mathbb{R}^{N}.
\end{eqnarray*}
By the definition \eqref{def:prox_beta_ell0} of $\mathrm{prox}_{\beta\|\cdot\|_0} \left(\mathbf{y}^{\star}\right)$, we obtain the desired inclusion relation \eqref{xstar_ystar_re1}.
We next show by contradiction that $\mathbf{y}^{\star}$ is a solution of model \eqref{model:general}.
Assume to the contrary that there exists a vector $\widetilde{\mathbf{y}}\in\mathbb{R}^{N}$ such that
$Q(\widetilde{\mathbf{y}})<Q(\mathbf{y}^{\star})$.
This inequality together with \eqref{Q_F_relation} leads to
\begin{eqnarray}
F(\widetilde{\mathbf{x}},\widetilde{\mathbf{y}})=Q(\widetilde{\mathbf{y}})
<Q(\mathbf{y}^{\star}),
\ \ {\rm for\ all}\  \widetilde{\mathbf{x}}\in\mathrm{prox}_{\beta\|\cdot\|_0} \left(\widetilde{\mathbf{y}}\right).\label{Q_F_re2}
\end{eqnarray}
By inclusion \eqref{xstar_ystar_re1} and equation \eqref{Q_F_relation}, we obtain that $Q(\mathbf{y}^{\star})=F(\mathbf{x}^{\star},\mathbf{y}^{\star})$.
Thus, \eqref{Q_F_re2} yields that $F(\widetilde{\mathbf{x}},\widetilde{\mathbf{y}})<F(\mathbf{x}^{\star},\mathbf{y}^{\star})$,
which contradicts the assumption and proves that $\mathbf{y}^{\star}$ is a solution of model \eqref{model:general}.

Now, we suppose that $\mathbf{y}^{\star}$ is a solution of model
\eqref{model:general} with $\mathbf{x}^{\star}$ satisfying inclusion  \eqref{xstar_ystar_re1} and show that the pair $(\mathbf{x}^{\star},\mathbf{y}^{\star})$ is a solution of model \eqref{model:two_variable}. Assume to the contrary that $(\mathbf{x}^{\star},\mathbf{y}^{\star})$ is not a solution of model \eqref{model:two_variable}. Then, there exists a pair $\left(\widetilde{\mathbf{x}},\widetilde{\mathbf{y}}\right)\in\mathbb{R}^{N}\times\mathbb{R}^{N}$ satisfying
\begin{eqnarray*}
F(\widetilde{\mathbf{x}},\widetilde{\mathbf{y}})
<F(\mathbf{x}^{\star},\mathbf{y}^{\star})=Q(\mathbf{y}^{\star}).
\end{eqnarray*}
If $\widetilde{\mathbf{x}}\in\mathrm{prox}_{\beta\|\cdot\|_0} \left(\widetilde{\mathbf{y}}\right)$, we have that
$
Q\left(\widetilde{\mathbf{y}}\right)=F(\widetilde{\mathbf{x}},\widetilde{\mathbf{y}}).
$
If $\widetilde{\mathbf{x}}\notin\mathrm{prox}_{\beta\|\cdot\|_0} \left(\widetilde{\mathbf{y}}\right)$, by the definition of $\mathrm{prox}_{\beta\|\cdot\|_0}$ and $\mathrm{env}_{\beta\|\cdot\|_0}$, we have that
$$
\mathrm{env}_{\beta\|\cdot\|_0}\left(\widetilde{\mathbf{y}}\right)<\frac{1}{2\beta}\|\widetilde{\mathbf{x}}-\widetilde{\mathbf{y}}\|_2^2
+\|\widetilde{\mathbf{x}}\|_0.
$$
Combining this with the definition of $Q$ in \eqref{Def-Q} and $F$ in \eqref{def:F}   yields
$
Q(\widetilde{\mathbf{y}})<F\left(\widetilde{\mathbf{x}},\widetilde{\mathbf{y}}\right).
$
Therefore, in either case, we have that
$
Q(\widetilde{\mathbf{y}})<Q(\mathbf{y}^{\star}),
$
which contradicts the assumption of $\mathbf{y}^{\star}$ being a solution of model \eqref{model:general}. Thus, $(\mathbf{x}^{\star},\mathbf{y}^{\star})$ must be a solution of model \eqref{model:two_variable}.
\end{proof}

Proposition \ref{pro:model_relation} leads us to solving minimization problem \eqref{model:two_variable}.

\section{Local Convexity of the env-$\ell_0$ Regularization Model}\label{sec:LConvexity}\setcounter{equation}{0}

In this section, we show that the proposed env-$\ell_0$ regularization model \eqref{model:general}, a non-convex minimization problem, is reduced to a convex minimization problem on a subdomain. For this purpose, as Proposition \ref{pro:model_relation} ensures that  the proposed model \eqref{model:general} and model \eqref{model:two_variable} are essentially equivalent, it suffices to prove that a local minimizer of the non-convex model \eqref{model:two_variable} is  a minimizer of a convex problem on a subdomain. To this end, we first construct a convex optimization problem on a proper subdomain, and discuss the relation between a minimizer of the convex optimization problem on a proper subdomain and that of model \eqref{model:two_variable}.   Then, we establish the relation between a local minimizer of model \eqref{model:general} and that of model  \eqref{model:two_variable}.

We first present a convex model on a subdomain of $\mathbb{R}^{N}\times \mathbb{R}^{N}$ related to model \eqref{model:two_variable}. We need the notion of the support of a vector. By $N(\mathbf{x})$ we denote the support of $\mathbf{x}\in\mathbb{R}^{N}$, the index set on which the components of $\mathbf{x}$ is nonzero, that is $N(\mathbf{x})~:=\{i~:x_i\neq 0\}$. Note that when the support of $\mathbf{x}$ in model \eqref{model:two_variable} is specified, the non-convex model \eqref{model:two_variable} reduces to a convex one. Based on this observation, we introduce a convex function by
\begin{eqnarray}
G(\mathbf{r},\mathbf{x},\mathbf{y})~:=\frac{1}{2}\|\mathbf{K}\mathbf{y}-\mathbf{r}\|_2^2
+\frac{\gamma}{2\beta}\|\mathbf{x}-\mathbf{y}\|_2^2,
\qquad\quad (\mathbf{r},\mathbf{x},\mathbf{y})\in \mathbb{R}^{d}\times\mathbb{R}^{N}\times \mathbb{R}^{N}.
\label{def:G}
\end{eqnarray}
When there is no ambiguity,  we shall write $G(\mathbf{x},\mathbf{y})=G(\mathbf{r},\mathbf{x},\mathbf{y})$ for simplicity since $\mathbf{r}$ is fixed.
Clearly, $F(\mathbf{x},\mathbf{y})= G(\mathbf{x},\mathbf{y})+\gamma \|\mathbf{x}\|_0$ and $G(\mathbf{x},\mathbf{y})$ is a differentiable and convex component of $F(\mathbf{x},\mathbf{y})$ on $\mathbb{R}^{N}\times \mathbb{R}^{N}$.  For a given index set $\mathcal{N}$,
we define a subspace of $\mathbb{R}^{N}$ by letting
\begin{eqnarray}
\mathcal{B}_{\mathcal{N}}~:=\left\{\mathbf{x}\in\mathbb{R}^{N}~:
N(\mathbf{x})\subseteq\mathcal{N}\right\}.
\label{def:B_N}
\end{eqnarray}
Clearly, for $\mathbf{x}\in\mathcal{B}_{\mathcal{N}}$, $x_j=0$ for all $j\notin \mathcal{N}$, and $\mathcal{B}_{\mathcal{N}}$ is convex.
Moreover,  $\mathcal{B}_{\mathcal{N}}$ is a closed set in $\mathbb{R}^{N}$.
To show this, we state  Item (i) of Lemma 4.7 in \cite{Xu} as the next lemma.

\begin{lemma}\label{corollary:support_xk_any}
If the sequence $\{\mathbf{x}^{k}\}\subset \mathbb{R}^{N}$ converges to $\mathbf{x}^{\star}$, then there exists an integer $V>0$ such that $N(\mathbf{x}^{\star})\subseteq N(\mathbf{x}^{k})$ for all $k\geq V$.
\end{lemma}


We remark that the reverse inclusion relation of that in Lemma \ref{corollary:support_xk_any} does not hold in general, see \cite{Xu}. We now return to set  $\mathcal{B}_{\mathcal{N}}$.

\begin{lemma}\label{lem:BN_closed}
The set $\mathcal{B}_{\mathcal{N}}$ defined by \eqref{def:B_N}
is closed in the  norm $\|\cdot\|_2$.
\end{lemma}
\begin{proof}
 Assume that $\left\{\mathbf{x}^k\right\}$
is a convergent sequence in $\mathcal{B}_{\mathcal{N}}$ in the norm  $\|\cdot\|_2$. As $\mathcal{B}_{\mathcal{N}}\subseteq\mathbb{R}^{N}$, $\left\{\mathbf{x}^k\right\}$ converges to some $\mathbf{x}^{\star}\in\mathbb{R}^{N}$.
By Lemma \ref{corollary:support_xk_any}, there  exists $V>0$ such that $N(\mathbf{x}^{\star})\subseteq N(\mathbf{x}^{k})$ for $k\geq V$. In addition, from the  definition of $\mathcal{B}_{\mathcal{N}}$, there holds $N(\mathbf{x}^{k})\subseteq \mathcal{N}$ for all $k\geq 1$. Thus, $N(\mathbf{x}^{\star})\subseteq \mathcal{N}$. Once again, by the definition \eqref{def:B_N} of $\mathcal{B}_{\mathcal{N}}$, we conclude that $\mathbf{x}^{\star}\in\mathcal{B}_{\mathcal{N}}$. Thus, $\mathcal{B}_{\mathcal{N}}$
is a closed set.
\end{proof}

For a given index set $\mathcal{N}$, we introduce the minimization problem on $\mathcal{B}_{\mathcal{N}}\times\mathbb{R}^{N}$ by
\begin{equation}\label{model:convex_model}
{\rm argmin}\left\{
G(\mathbf{x},\mathbf{y}),\quad (\mathbf{x},\mathbf{y})\in \mathcal{B}_{\mathcal{N}}\times\mathbb{R}^{N}\right\}.
\end{equation}
Since function $G$ is convex and set
$\mathcal{B}_{\mathcal{N}}\times\mathbb{R}^{N}$
is convex, problem \eqref{model:convex_model} is convex.

We next show that the non-convex model \eqref{model:two_variable} is equivalent to the convex model~\eqref{model:convex_model} with a properly chosen index set $\mathcal{N}$. To this end, we explore properties of the support set of certain sequences in $\mathbb{R}^{N}$.
For a given index set $\mathcal{N}$, we define an operator ${\rm P}_{\mathcal{B}_{\mathcal{N}}}:~\mathbb{R}^{N}\rightarrow \mathcal{B}_{\mathcal{N}}$ as
\begin{eqnarray}
{\rm P}_{\mathcal{B}_{\mathcal{N}}}(\mathbf{y}) ~:= \left\{
                             \begin{array}{ll}
                             y_i,\ \ & \hbox{if $i\in \mathcal{N}$;} \\
                               0,\ \ & \hbox{otherwise.}
                             \end{array}
                         \right.\label{P_B}
\end{eqnarray}
The next lemma confirms that the operator defined by \eqref{P_B} is indeed the orthogonal projection from $\mathbb{R}^{N}$ onto $\mathcal{B}_{\mathcal{N}}$.

\begin{lemma}\label{P_BN_L2}
If $\mathcal{N}$ is a subset of $\{1,2,\ldots,N\}$, then ${\rm P}_{\mathcal{B}_{\mathcal{N}}}$ defined by equation \eqref{P_B} is the orthogonal projection from $\mathbb{R}^{N}$ onto the closed convex set $\mathcal{B}_{\mathcal{N}}$.
\end{lemma}
\begin{proof} Let $\mathbf{y}\in\mathbb{R}^{N}$ be fixed.
For all $\mathbf{x}\in\mathcal{B}_{\mathcal{N}}$,
we have that
$$
\left\|\mathbf{x}-\mathbf{y}\right\|_2^2
=\sum_{j\in\mathcal{N}}\left|x_j-y_j\right|^2
+\sum_{j\notin\mathcal{N}}\left|y_j\right|^2.
$$
By the definition \eqref{P_B} of ${\rm P}_{\mathcal{B}_{\mathcal{N}}}$, we derive that
\begin{eqnarray}\label{eq:orthogonal_PN_xy}
 \left\|\mathbf{x}-\mathbf{y}\right\|_2^2= \left\|
 \mathbf{x}-
 {\rm P}_{\mathcal{B}_{\mathcal{N}}}\left(\mathbf{y}\right)\right\|_2^2
 +\left\|\mathbf{y}-
 {\rm P}_{\mathcal{B}_{\mathcal{N}}}\left(\mathbf{y}\right)\right\|_2^2.
 \end{eqnarray}
Therefore,  there holds
$$
\left\|\mathbf{y}- {\rm P}_{\mathcal{B}_{\mathcal{N}}}\left(\mathbf{y}\right)\right\|_2
\leq \left\|\mathbf{y}-\mathbf{x}\right\|_2,\ \  \mbox{for all}\ \ \mathbf{x}\in\mathcal{B}_{\mathcal{N}}.
$$
That is, ${\rm P}_{\mathcal{B}_{\mathcal{N}}}\left(\mathbf{y}\right)$ is the best  $\ell_2$ approximation of $\mathbf{y}$ from $\mathcal{B}_{\mathcal{N}}$. In other words,
 ${\rm P}_{\mathcal{B}_{\mathcal{N}}}$  is the orthogonal projection from $\mathbb{R}^{N}$ onto $\mathcal{B}_{\mathcal{N}}$.
\end{proof}

It is convenient to identify the proximity operator of the $\ell_0$ norm  with the projection ${\rm P}_{\mathcal{B}_{\mathcal{N}}}$.
We recall the closed-form formula of the proximity operator of the $\ell_0$ norm (see, for example \cite{SXZ}). For all $\mathbf{z}$ in $\mathbb{R}^{N}$,
\begin{equation}\label{prox_envl0-1}
\mathrm{prox}_{\beta\|\cdot\|_0} (\mathbf{z}) = \left[\mathrm{prox}_{\beta|\cdot|_0} (z_1),  \mathrm{prox}_{\beta|\cdot|_0} (z_2), \ldots, \mathrm{prox}_{\beta|\cdot|_0} (z_{N})\right]^\top,
\end{equation}
where $\top$ denotes the transpose and
\begin{equation}\label{prox_envl0-1-component}
\mathrm{prox}_{\beta|\cdot|_0} (z_i) = \left\{
                             \begin{array}{ll}
                              \left\{z_i\right\}, & \hbox{$|z_i|> \sqrt{2\beta}$;} \\
                              \left\{                                                  z_i,0\right\}, &
                              \hbox{$|z_i|= \sqrt{2\beta}$;} \\
                               \left\{0\right\}, & \hbox{otherwise.}
                             \end{array}
                           \right.
\end{equation}
\begin{lemma}\label{P_BN_prox_L0_two}
 Suppose that $\beta>0$  and $\mathbf{x},\mathbf{z} \in\mathbb{R}^{N}$ satisfy $\mathbf{x}={\rm P}_{\mathcal{B}_{N(\mathbf{x})}}(\mathbf{z})$.   If
 $$
 |z_j|\geq\sqrt{2\beta} \ \ \ {\rm for \ all}\  j\in N(\mathbf{x})\ \ {\rm and}\ \
 |z_j|<\sqrt{2\beta} \ \ \ {\rm for \ all}\  j\notin N(\mathbf{x}),
 $$
 then  $\mathbf{x}\in\mathrm{prox}_{\beta\|\cdot\|_0}\left(\mathbf{z}\right)$.
\end{lemma}
\begin{proof}
As $\mathbf{x}={\rm P}_{\mathcal{B}_{N(\mathbf{x})}}(\mathbf{z})$, using the definition \eqref{P_B} of ${\rm P}_{\mathcal{B}_{\mathcal{N}}}$ leads to that
$x_j=z_j$ for all $j\in N(\mathbf{x})$, and $x_j=0$ for all $j\notin N(\mathbf{x})$.
On the other hand, under the hypothesis, applying formulas \eqref{prox_envl0-1} and \eqref{prox_envl0-1-component} yields  that   $\mathrm{prox}_{\beta|\cdot|_0} (z_j)$ is equal to $\{z_j\}$ or $\{z_j,\ 0\}$ for all $j\in N(\mathbf{x})$,
and  $\mathrm{prox}_{\beta|\cdot|_0} (z_j)=\{0\}$
for all $j\notin N(\mathbf{x})$.
Summarizing the above leads to the desired conclusion.
\end{proof}


We next present a reverse result of Lemma \ref{P_BN_prox_L0_two}.
\begin{lemma}\label{P_BN_prox_L0}
Suppose that $\beta>0$ and $\mathbf{x},\mathbf{z} \in\mathbb{R}^{N}$ satisfy $\mathbf{x}\in\mathrm{prox}_{\beta\|\cdot\|_0}\left(\mathbf{z}\right)$. Then

(i) for all $j\in N(\mathbf{x})$, $|x_j|=|z_j|\geq\sqrt{2\beta}$;

(ii) for all $j\notin N(\mathbf{x})$, $x_j=0$ and $|z_j|\leq\sqrt{2\beta}$;

(iii) $\mathbf{x}={\rm P}_{\mathcal{B}_{N(\mathbf{x})}}(\mathbf{z})$.
\end{lemma}
\begin{proof}
Items (i) and (ii) follow immediately from formulas \eqref{prox_envl0-1} and \eqref{prox_envl0-1-component}.

It remains to prove Item (iii). Clearly, for all $j\notin N(\mathbf{x})$, there holds $x_j=0$. For all
$j\in N(\mathbf{x})$, by Item (i) of this lemma, we have that $|x_j|=|z_j|\geq \sqrt{2\beta}$. According to formula \eqref{prox_envl0-1-component}, we obtain that $x_j=z_j$. By the definition \eqref{P_B} of ${\rm P}_{\mathcal{B}_{{N}(\mathbf{x})}}(\mathbf{z})$, we confirm (iii).
\end{proof}

We now apply Lemma \ref{P_BN_prox_L0} to a sequence generated by $\mathrm{prox}_{\beta\|\cdot\|_0}$.

\begin{lemma}\label{lemma:x_sparsity}
If $\left\{\mathbf{z}^{k}\right\}$ is a sequence in $\mathbb{R}^{N}$ and for a fixed $\beta>0$, $\mathbf{x}^{k}\in\mathrm{prox}_{\beta\|\cdot\|_0}\left(\mathbf{z}^{k}\right)$, then

(i) $|x_j^k|\geq \sqrt{2\beta}$  for all $j\in N(\mathbf{x}^k)$,
and $x_j^k=0$ for all $j\in \overline{N(\mathbf{x}^k)}$,

(ii) if $N(\mathbf{x}^k)\neq N(\mathbf{x}^{k+1})$, $\|\mathbf{x}^{k+1}-\mathbf{x}^{k}\|_2\geq \sqrt{2\beta}$.

\noindent where $\overline{\mathcal{C}}$ denotes the complement of the set  $\mathcal{C}\subseteq\{1,2,\dots, N\}$ in $\{1,2,\dots, N\}$.
\end{lemma}
\begin{proof}
Item (i) is immediately obtained from Lemma \ref{P_BN_prox_L0}.

To prove Item (ii) we assume that $N(\mathbf{x}^k)\neq N(\mathbf{x}^{k+1})$
for some $k>0$. Then, there exists at least an index $j$ such that $j\in N(\mathbf{x}^k)\cap\overline{N(\mathbf{x}^{k+1})}$ or $j\in \overline{N(\mathbf{x}^{k})}\cap N(\mathbf{x}^{k+1})$. By Item (i), $|x_j^k|\geq\sqrt{2\beta}$ and $x_j^{k+1}=0$ for the first case, or $x_j^{k}=0$ and $|x_j^{k+1}|\geq\sqrt{2\beta}$ for the second case. For the both cases, there holds $|x_j^{k+1}-x_j^k|\geq \sqrt{2\beta}$. It follows that
$$
\|\mathbf{x}^{k+1}-\mathbf{x}^k\|_2\geq|x_j^{k+1}-x_j^{k}|\geq \sqrt{2\beta},
$$
proving Item (ii).
\end{proof}

When the sequence $\mathbf{x}^{k}$ described in  Lemma \ref{lemma:x_sparsity} is convergent, we have the following result.

\begin{lemma}\label{lemma:support_xstar}
If $\left\{\mathbf{z}^{k}\right\}$ is a sequence in $\mathbb{R}^{N}$ and for a fixed $\beta>0$, sequence $\{\mathbf{x}^k\}$ defined by $\mathbf{x}^{k}\in\mathrm{prox}_{\beta\|\cdot\|_0}\left(\mathbf{z}^{k}\right)$ converges to $\mathbf{x}^{\star}$, then there exists
an integer $V>0$ such that $N(\mathbf{x}^{k})=N(\mathbf{x}^{\star})$ for all $k\geq V$.
\end{lemma}
\begin{proof}
Since the sequence $\{\mathbf{x}^{k}\}$ converges to $\mathbf{x}^{\star}$, we have that
$$
\lim\limits_{k\rightarrow \infty}\|\mathbf{x}^{k+1}-\mathbf{x}^{k}\|_2=0.
$$
Hence, there exists a number $V>0$ such that
$$
\|\mathbf{x}^{k+1}-\mathbf{x}^{k}\|_2<\sqrt{2\beta}, \ \ {\rm for\ all}\ k\geq V.
$$
It follows from Item (ii) of Lemma \ref{lemma:x_sparsity} that $N(\mathbf{x}^{k})=N(\mathbf{x}^{k+1})$ for all $k\geq V$. Thus, for any fixed $k\geq V$ and all $m>0$, $N(\mathbf{x}^{k})=N(\mathbf{x}^{k+m})$.
Since $\{\mathbf{x}^{k}\}$ converges to $\mathbf{x}^{\star}$, letting $m\to \infty$, we obtain that $N(\mathbf{x}^{k})=N(\mathbf{x}^{\star})$ for all $k\geq V$.
\end{proof}

The next proposition is a special case of Item (ii) of Lemma 4.7 in \cite{Xu}. This proposition plays a crucial role in reducing the non-convex optimization \eqref{model:two_variable} to a convex one.

\begin{proposition}\label{pro:sequence_BN}
Let  $\left(\mathbf{x}^{\star},\mathbf{y}^{\star}\right)\in \mathbb{R}^{N}\times \mathbb{R}^{N}$ be given and set $\mathcal{N}:=N(\mathbf{x}^{\star})$. If $\left\{(\mathbf{x}^k,\mathbf{y}^k)\right\}$
is a convergent sequence in $\mathcal{B}_{\mathcal{N}}\times\mathbb{R}^{N}$ with limit
$(\mathbf{x}^{\star},\mathbf{y}^{\star})$, then, there  exists
an integer $V>0$  such  that  $N(\mathbf{x}^{k})=N(\mathbf{x}^{\star})$  for all $k\geq V$.
\end{proposition}
\begin{proof}
As $\left\{(\mathbf{x}^k,\mathbf{y}^k)\right\}$ converges to $(\mathbf{x}^{\star},\mathbf{y}^{\star})$, we have that the sequence $\{\mathbf{x}^k\}$ converges to $\mathbf{x}^{\star}$. By Lemma \ref{corollary:support_xk_any}, there exists an integer  $V>0$  such  that $N(\mathbf{x}^{\star})\subseteq N(\mathbf{x}^{k})$ for
all  $k\geq V$.  As $\mathbf{x}^{k}\in\mathcal{B}_{\mathcal{N}}$ for all $k$, the definition \eqref{def:B_N} of $\mathcal{B}_{\mathcal{N}}$ leads to $N(\mathbf{x}^{k})\subseteq N(\mathbf{x}^{\star})$. Combining the above two inclusions yields the desired result of this proposition.
\end{proof}

We are now ready to present a result which connects the non-convex model~\eqref{model:two_variable} with the convex model  \eqref{model:convex_model}.
We first review the definition of a local minimizer of the non-convex model~\eqref{model:two_variable}.
If there exists a $\delta>0$ such that
$$
F\left(\mathbf{x}^{\star},\mathbf{y}^{\star}\right)
\leq F\left(\mathbf{x}^{\star}+\triangle \mathbf{x},\mathbf{y}^{\star}+\triangle \mathbf{y}\right),
\ \ {\rm for \ all }\ \left\|\triangle \mathbf{x}\right\|\leq \delta, \ \ \left\|\triangle \mathbf{y}\right\|\leq \delta,
$$
we call $\left(\mathbf{x}^{\star},\mathbf{y}^{\star}\right)$ a local minimizer of the non-convex model~\eqref{model:two_variable}.

\begin{theorem}\label{lemma:relation_nonconv_convex}
Let  $\beta, \gamma>0$, and $\left(\mathbf{x}^{\star},\mathbf{y}^{\star}\right)
\in\mathbb{R}^{N}\times \mathbb{R}^{N}$ be given.
The pair $(\mathbf{x}^{\star},\mathbf{y}^{\star})$
is a local minimizer of the non-convex model~\eqref{model:two_variable}
if and only if $(\mathbf{x}^{\star},\mathbf{y}^{\star})$ is a minimizer of the convex model \eqref{model:convex_model} with $\mathcal{N}:=N(\mathbf{x}^{\star})$.
\end{theorem}
\begin{proof}
 This result follows directly from Corollary 4.9 of \cite{Xu}, with $\phi(\mathbf{y}):=\frac{1}{2}\|{\bf K}\mathbf{y}-\mathbf{r}\|_2^2$, $\mu:=\frac{\gamma}{2\beta}$ and $D:=\mathbf{I}$.
\end{proof}

In Proposition \ref{pro:model_relation}, we have identified a global minimizer of model~\eqref{model:two_variable} with that of model \eqref{model:general}.
In the following proposition, inspired by Theorem 4.4 in \cite{Xu}, we explore the relation between a local minimizer of
 model~\eqref{model:two_variable} and that of model \eqref{model:general}.

\begin{proposition}\label{pro:model_relation_localm}
Suppose that $\beta,  \gamma>0$ and $(\mathbf{x}^{\star},\mathbf{y}^{\star})\in\mathbb{R}^{N}\times \mathbb{R}^{N}$   satisfies  $\mathbf{x}^{\star}\in\mathrm{prox}_{\beta\|\cdot\|_0} \left(\mathbf{y}^{\star}\right)$.
If $(\mathbf{x}^{\star},\mathbf{y}^{\star})$ is a local minimizer of model~\eqref{model:two_variable} and $|y_j^{\star}|\neq\sqrt{2\beta}$ for all $j\in N(\mathbf{y}^{\star})$, then $\mathbf{y}^{\star}$ is a local minimizer of model \eqref{model:general}. Conversely, if $\mathbf{y}^{\star}$ is a local minimizer of model \eqref{model:general}, then  $(\mathbf{x}^{\star},\mathbf{y}^{\star})$ is a local minimizer of model~\eqref{model:two_variable}.
\end{proposition}
\begin{proof}
If the pair $(\mathbf{x}^{\star},\mathbf{y}^{\star})$ is a local minimizer of
 of model~\eqref{model:two_variable} and
 $|y_j^{\star}|\neq\sqrt{2\beta}$ for all $j\in N(\mathbf{y}^{\star})$, we prove that $\mathbf{y}^{\star}$
is a local minimizer of model \eqref{model:general} by contradiction.
As $(\mathbf{x}^{\star},\mathbf{y}^{\star})$ is a local minimizer of model~\eqref{model:two_variable}, there exists a number $\delta>0$ such that
 \begin{eqnarray}
 F\left(\mathbf{x}^{\star},\mathbf{y}^{\star}\right)
\leq F\left(\mathbf{x}^{\star}+\triangle \mathbf{x},\mathbf{y}^{\star}+\triangle \mathbf{y}\right),
\ \ {\rm for \ all }\ \left\|\triangle \mathbf{x}\right\|\leq \delta, \ \ \left\|\triangle \mathbf{y}\right\|\leq \delta.
\label{inq_local_F_Q}
 \end{eqnarray}
Assume that $\mathbf{y}^{\star}$
is not a local minimizer of model \eqref{model:general}. Hence, there exists a sequence
$\{\mathbf{y}^k\}$ satisfying
\begin{eqnarray}
\left\|\mathbf{y}^k-\mathbf{y}^{\star}\right\|_2\rightarrow 0, \ as\ k\rightarrow +\infty,
 \ \ \ \ \ {\rm and} \quad\qquad
Q(\mathbf{y}^k)<Q(\mathbf{y}^{\star}).
\label{Q_local_m}
\end{eqnarray}
It follows that for any $ j\in N(\mathbf{y}^{\star})$,
\begin{eqnarray}\label{lim_y_star_neqbeta}
\lim\limits_{k\rightarrow \infty}|{y}_j^k|=|{y}_j^{\star}|\neq \sqrt{2\beta}.
\end{eqnarray}
For  any $ j\notin N(\mathbf{y}^{\star})$, as ${y}_j^{\star}=0$, we have that
\begin{eqnarray}\label{lim_y_star_0}
\lim\limits_{k\rightarrow \infty}|{y}_j^k|=0.
\end{eqnarray}

For convenient presentation, we introduce three index sets
$$
\mathcal{N}_1:=\{j: j\in N(\mathbf{y}^{\star}) \ \mbox{and} \ |{y}_j^{\star}|>\sqrt{2\beta}\},\
\mathcal{N}_2:=\{j: j\in N(\mathbf{y}^{\star}) \  \mbox{and} \ |{y}_j^{\star}|<\sqrt{2\beta}\},
\
\mathcal{N}_3:=\overline{N(\mathbf{y}^{\star})}.
$$
Under the hypothesis, we have $\{1,2,\ldots, N\}=\mathcal{N}_1\bigcup\mathcal{N}_2\bigcup\mathcal{N}_3$ and $\mathcal{N}_i\bigcap\mathcal{N}_j=\emptyset$ if $i\neq j$.
For any $j\in \mathcal{N}_1$, from \eqref{lim_y_star_neqbeta} there exists an integer $V_j>0$ such that for all $k\geq V_j$,
\begin{eqnarray}
|{y}_j^k|> \sqrt{2\beta}.
\label{y_j_bigger_2beta}
\end{eqnarray}
For  any $j\in \mathcal{N}_2$, from \eqref{lim_y_star_neqbeta} there exists an integer $V_j>0$ such that for all $k\geq V_j$,
\begin{eqnarray}\label{y_j_smaller_2beta1}
|{y}_j^k|<\sqrt{2\beta}.
\end{eqnarray}
For any $j\in \mathcal{N}_3$, from \eqref{lim_y_star_0} there exists an integer $V_j>0$ such that for all $k\geq V_j$,
\begin{eqnarray}\label{y_j_smaller_2beta2}
|{y}_j^k|<\sqrt{2\beta}.
\end{eqnarray}
Let $V:=\max\{V_j: j=1,2,\ldots,N\}$
and $\mathbf{x}^k\in\mathrm{prox}_{\beta\|\cdot\|_0} \left(\mathbf{y}^k\right)$.

We next prove that, for $k\geq V$
there holds that $N(\mathbf{x}^k)=N(\mathbf{x}^{\star})$.
If $j\in N(\mathbf{x}^{\star})$, as $\mathbf{x}^{\star}\in\mathrm{prox}_{\beta\|\cdot\|_0} \left(\mathbf{y}^{\star}\right)$, by Lemma \ref{P_BN_prox_L0} we have that $|{y}^{\star}_j|\geq\sqrt{2\beta}$. Since $\sqrt{2\beta}\neq|{y}_j^{\star}|$ for all $j\in N(\mathbf{y}^{\star})$,
we derive that for any $j\in N(\mathbf{x}^{\star})$, there holds $|{y}^{\star}_j|>\sqrt{2\beta}$. Hence, $j\in \mathcal{N}_1$. By  \eqref{y_j_bigger_2beta},  when $k\geq V$, there holds $|{y}_j^k|> \sqrt{2\beta}.$ According to \eqref{prox_envl0-1-component}, we have that ${x}_j^k={y}_j^k$, which verifies that $j\in N(\mathbf{x}^k)$. Therefore, $N(\mathbf{x}^{\star})\subseteq N(\mathbf{x}^k)$ for all $k\geq V$. Conversely, if $k\geq V$ and $j\in N(\mathbf{x}^k)$, applying Lemma \ref{P_BN_prox_L0} yields $|{y}_j^k|\geq\sqrt{2\beta}$. As $|{y}_j^k|\rightarrow|{y}_j^\star|$ when $k\rightarrow +\infty $, there holds $|{y}_j^\star|\geq\sqrt{2\beta}$.  Thus,  $j\in \mathcal{N}_1$.
Using \eqref{prox_envl0-1-component}, we have that ${x}_j^{\star}={y}_j^{\star}$,
which confirms that $j\in N(\mathbf{x}^{\star})$. Therefore, $N(\mathbf{x}^k)\subseteq N(\mathbf{x}^{\star})$ for all $k\geq V$.
Summarizing the two inclusions obtained above yields  that $N(\mathbf{x}^k)=N(\mathbf{x}^{\star})$ for $k\geq V$.

Set $\mathcal{N}:=N(\mathbf{x}^{\star})$. Notice that for $k\geq V$,
 $N(\mathbf{x}^k)=N(\mathbf{x}^{\star})$.
This together with Lemmas \ref{P_BN_L2} and \ref{P_BN_prox_L0} yields that for $k\geq V$
$$
\left\|\mathbf{x}^k-\mathbf{x}^{\star}\right\|_2
=\left\|{\rm P}_{\mathcal{B}_{\mathcal{N}}}(\mathbf{y}^k)
-{\rm P}_{\mathcal{B}_{\mathcal{N}}}(\mathbf{y}^{\star})\right\|_2
\leq\left\|\mathbf{y}^k-\mathbf{y}^{\star}\right\|_2.
$$
Since $\left\|\mathbf{y}^k-\mathbf{y}^{\star}\right\|_2\rightarrow 0$ as
$k\rightarrow +\infty$, there exists $V_1\geq V>0$ such that for $k\geq V_1$
$$
\left\|\mathbf{x}^k-\mathbf{x}^{\star}\right\|\leq \frac{\delta}{2}
\ \ \mbox{and}\ \
  \left\|\mathbf{y}^k-\mathbf{y}^{\star}\right\|\leq \frac{\delta}{2}.
$$
For $k\geq V_1$, by \eqref{Q_local_m} and \eqref{Q_F_relation},
the pair $(\mathbf{x}^k,\mathbf{y}^k)$ with $\mathbf{x}^k\in\mathrm{prox}_{\beta\|\cdot\|_0} \left(\mathbf{y}^k\right)$
satisfies that
$$
F(\mathbf{x}^k,\mathbf{y}^k)=Q(\mathbf{y}^k)
<Q(\mathbf{y}^{\star})=F(\mathbf{x}^{\star},\mathbf{y}^{\star}),
$$
which contradicts inequality \eqref{inq_local_F_Q}. Therefore, $\mathbf{y}^{\star}$
is  a local minimizer of model \eqref{model:general}.

The second part of this proposition can be proved in a way similar to the proof of Proposition \ref{pro:model_relation}. In deed, under the hypothesis, if $\mathbf{y}^{\star}$
is a local minimizer of model \eqref{model:general}, we prove the pair $(\mathbf{x}^{\star},\mathbf{y}^{\star})$ with $\mathbf{x}^{\star}\in\mathrm{prox}_{\beta\|\cdot\|_0} \left(\mathbf{y}^{\star}\right)$ is a local minimizer of model~\eqref{model:two_variable}   by contradiction.
 As $\mathbf{y}^{\star}$ is a local minimizer of model \eqref{model:general}, there exists a number $\delta_1>0$ such that
  \begin{eqnarray}
 Q\left(\mathbf{y}^{\star}\right)
\leq Q\left(\mathbf{y}^{\star}+\triangle \mathbf{y}\right),
\ \ {\rm for \ all }\  \ \ \left\|\triangle \mathbf{y}\right\|\leq \delta_1.
\label{inq_localmini_Q}
 \end{eqnarray}
 Assume to the contrary that $(\mathbf{x}^{\star},\mathbf{y}^{\star})$ is not a local minimizer of model \eqref{model:two_variable}. Then, there exists a pair $\left(\widetilde{\mathbf{x}},\widetilde{\mathbf{y}}\right)\in\mathbb{R}^{N}\times\mathbb{R}^{N}$ satisfying
\begin{eqnarray*}
F(\widetilde{\mathbf{x}},\widetilde{\mathbf{y}})
<F(\mathbf{x}^{\star},\mathbf{y}^{\star})=Q(\mathbf{y}^{\star}),
\ \ \ \left\|\widetilde{\mathbf{x}}-\mathbf{x}^{\star}\right\|\leq \frac{\delta_1}{2},
 \ \left\|\widetilde{\mathbf{y}}-\mathbf{y}^{\star}\right\|\leq \frac{\delta_1}{2}.
\end{eqnarray*}
If $\widetilde{\mathbf{x}}\in\mathrm{prox}_{\beta\|\cdot\|_0} \left(\widetilde{\mathbf{y}}\right)$, we have that
$
Q\left(\widetilde{\mathbf{y}}\right)=F(\widetilde{\mathbf{x}},\widetilde{\mathbf{y}}).
$
If $\widetilde{\mathbf{x}}\notin\mathrm{prox}_{\beta\|\cdot\|_0} \left(\widetilde{\mathbf{y}}\right)$, by the definition of $\mathrm{prox}_{\beta\|\cdot\|_0}$ and $\mathrm{env}_{\beta\|\cdot\|_0}$, we have that
$$
\mathrm{env}_{\beta\|\cdot\|_0}\left(\widetilde{\mathbf{y}}\right)<\frac{1}{2\beta}\|\widetilde{\mathbf{x}}-\widetilde{\mathbf{y}}\|_2^2
+\|\widetilde{\mathbf{x}}\|_0.
$$
Combining this with the definition \eqref{Def-Q} of $Q$  and that \eqref{def:F} of $F$ yields
$
Q(\widetilde{\mathbf{y}})<F\left(\widetilde{\mathbf{x}},\widetilde{\mathbf{y}}\right).
$
Therefore, in either case, we have that
$
Q(\widetilde{\mathbf{y}})<Q(\mathbf{y}^{\star})
$
and $\left\|\widetilde{\mathbf{y}}-\mathbf{y}^{\star}\right\|\leq \frac{\delta_1}{2}$,
which contradicts  \eqref{inq_localmini_Q}. Thus, $(\mathbf{x}^{\star},\mathbf{y}^{\star})$ must be a local minimizer of model \eqref{model:two_variable}.
\end{proof}


\section{A Fixed-Point Formulation and a Fixed-Point Iterative Algorithm }\label{sec:MAM_alro}\setcounter{equation}{0}


In this section, we describe a fixed-point formulation of problem \eqref{model:two_variable} and then propose an iterative algorithm for finding a local minimizer of the minimization problem \eqref{model:two_variable} based on the fixed-point formulation.


In the following proposition, we characterize  a  minimizer  of the convex model \eqref{model:convex_model} with a proper set $\mathcal{N}$.

\begin{proposition}\label{pro:G_minimizer_exist}
Suppose  $\beta, \gamma>0$.  If  $\mathcal{C}$ is a subset of $\{1,2,\ldots,N\}$, then  model \eqref{model:convex_model} with $\mathcal{N}:=\mathcal{C}$ has a solution and a pair $\left(\mathbf{x}^{\star},\mathbf{y}^{\star}\right)
 \in\mathbb{R}^{N}\times \mathbb{R}^{N}$ is a solution of model \eqref{model:convex_model} with $\mathcal{N}:=\mathcal{C}$ if and only if
\begin{eqnarray}
\mathbf{x}^{\star}={\rm P}_{\mathcal{B}_{\mathcal{N}}}(\mathbf{y}^{\star}),
\ \ \ \mathbf{y}^{\star}=\mathbf{x}^{\star}
-\frac{\beta}{\gamma}\mathbf{K}^*\left(\mathbf{K}\mathbf{y}^{\star}
-\mathbf{r}\right).
\label{eq:fixed-eqs_convex_model_solution}
\end{eqnarray}
\end{proposition}

\begin{proof}
In the case when $\overline{\mathcal{N}}= \emptyset$, by the definition \eqref{P_B} of ${\rm P}_{\mathcal{B}_{\mathcal{N}}}$, we get that ${\rm P}_{\mathcal{B}_{\mathcal{N}}}\left(\mathbf{y}\right)=\mathbf{y}$ for all $\mathbf{y}\in\mathbb{R}^{N}$. Hence, model \eqref{model:convex_model}
reduces to be a convex model on $\mathbb{R}^{N}\times\mathbb{R}^{N}$. Applying the Fermat rule yields that $\left(\tilde{\mathbf{x}}, \tilde{\mathbf{y}}\right)$ is a solution of model \eqref{model:convex_model} if and only if $\left(\tilde{\mathbf{x}},\tilde{\mathbf{y}}\right)$ satisfies equation  \eqref{eq:fixed-eqs_convex_model_solution}.

We now consider the case when $\overline{\mathcal{N}}\neq \emptyset$. We first prove the existence of a minimizer of model \eqref{model:convex_model} with $\mathcal{N}$. To this end, noticing that
$ {\rm P}_{\mathcal{B}_{\mathcal{N}}}\left(\mathbf{x}\right)=\mathbf{x}$ for any $\mathbf{x}\in \mathcal{B}_{\mathcal{N}}$ and applying formula \eqref{eq:orthogonal_PN_xy},  we  rewrite model \eqref{model:convex_model} as
\begin{eqnarray}
{\rm argmin}\left\{
\frac{1}{2}\|\mathbf{Ky}-\mathbf{r}\|_2^2
+\frac{\gamma}{2\beta}\|\mathbf{y}-
 {\rm P}_{\mathcal{B}_{\mathcal{N}}}\left(\mathbf{y}\right)\|_2^2
+\frac{\gamma}{2\beta}\|\mathbf{x}-
 {\rm P}_{\mathcal{B}_{\mathcal{N}}}\left(\mathbf{y}\right)\|_2^2,\ \ (\mathbf{x},\mathbf{y})\in \mathcal{B}_{\mathcal{N}}\times\mathbb{R}^{N}\right\}.
\label{model:convex_model_eqv}
\end{eqnarray}
The model above can be solved by two steps. Firstly, since $\overline{\mathcal{N}}\neq \emptyset$, we have that
\begin{eqnarray*}
\lim\limits_{\|\mathbf{y}\|_2\rightarrow\infty}\left(\frac{1}{2}\left\|\mathbf{Ky}-\mathbf{r}\right\|_2^2
+\frac{\gamma}{2\beta}\left\|\mathbf{y}-
 {\rm P}_{\mathcal{B}_{\mathcal{N}}}\left(\mathbf{y}\right)\right\|_2^2\right)=+\infty.
\end{eqnarray*}
Hence, the objective function of the following model
\begin{eqnarray}
\min\left\{\frac{1}{2}\|\mathbf{Ky}-\mathbf{r}\|_2^2
+\frac{\gamma}{2\beta}\|\mathbf{y}-
 {\rm P}_{\mathcal{B}_{\mathcal{N}}}\left(\mathbf{y}\right)\|_2^2,~\ \ \ \mathbf{y}\in\mathbb{R}^{N}\right\}
\label{model_overlineN}
\end{eqnarray}
is coercive. Therefore, model \eqref{model_overlineN} has a solution $\tilde{\mathbf{y}}$.
Secondly, set $\tilde{\mathbf{x}}~:={\rm P}_{\mathcal{B}_{\mathcal{N}}}(\tilde{\mathbf{y}})$. By the definition \eqref{P_B} of ${\rm P}_{\mathcal{B}_{\mathcal{N}}}$, we have that $\tilde{\mathbf{x}}\in\mathcal{B}_{\mathcal{N}}$. By the expression of model \eqref{model:convex_model_eqv} and the discussion above, we derive that $\left(\tilde{\mathbf{x}},\tilde{\mathbf{y}}\right)$ is a solution of model \eqref{model:convex_model_eqv}. Thus, model \eqref{model:convex_model}  with $\mathcal{N}:=\mathcal{C}$ has a solution.

Note that model \eqref{model:convex_model} with $\mathcal{N}:=\mathcal{C}$ is a convex minimization problem with a differentiable objective function. Applying the Fermat rule yields that $\left(\mathbf{x}^{\star},\mathbf{y}^{\star}\right)$ is a solution of \eqref{model:convex_model}  with $\mathcal{N}:=\mathcal{C}$ if and only if
\begin{eqnarray}
G_{\mathbf{x}}(\mathbf{x}^{\star},\mathbf{y}^{\star})^T\left(\mathbf{x}-\mathbf{x}^{\star}\right)\geq\mathbf{0},\ \  \mbox{for all}\ \ \mathbf{x}\in\mathcal{B}_{\mathcal{N}},
\ \
G_{\mathbf{y}}(\mathbf{x}^{\star},\mathbf{y}^{\star})^T\left(\mathbf{y}-\mathbf{y}^{\star}\right)\geq\mathbf{0}, \ \  \mbox{for all}\ \ \mathbf{y}\in\mathbb{R}^{N},
\label{eq:Gx_G_y_N}
\end{eqnarray}
where $G_{\mathbf{x}}$ and $G_{\mathbf{y}}$ denote the derivative of $G$ with respect to $\mathbf{x}$ and $\mathbf{y}$, respectively. By Lemma \ref{P_BN_L2},  ${\rm P}_{\mathcal{B}_{\mathcal{N}}}$ is an orthogonal projection. Hence, for a pair $\left(\mathbf{x},\mathbf{y}\right) \in\mathcal{B}_{\mathcal{N}}\times \mathbb{R}^{N}$,
equation \eqref{eq:orthogonal_PN_xy} holds. By differentiation, we obtain that
\begin{eqnarray}
&&G_{\mathbf{x}}(\mathbf{x}^{\star},\mathbf{y}^{\star})=\frac{\gamma}{\beta}
\left(\mathbf{x}^{\star}-{\rm P}_{\mathcal{B}_{\mathcal{N}}}(\mathbf{y}^{\star})\right),
\label{eqn:G_x_N}
\\
&&G_{\mathbf{y}}(\mathbf{x}^{\star},\mathbf{y}^{\star})=\mathbf{K}^*\left(\mathbf{K}\mathbf{y}^{\star}-\mathbf{r}\right)
 +\frac{\gamma}{\beta}
 \left(\mathbf{y}^{\star}-\mathbf{x}^{\star}\right).
\label{eqn:G_y_N}
\end{eqnarray}
Substituting \eqref{eqn:G_x_N} into the first inequality of \eqref{eq:Gx_G_y_N} and letting
$\mathbf{x}:={\rm P}_{\mathcal{B}_{\mathcal{N}}}(\mathbf{y}^{\star})$ yield the first equation of \eqref{eq:fixed-eqs_convex_model_solution}. Substituting \eqref{eqn:G_y_N} into the second inequality of \eqref{eq:Gx_G_y_N} and choosing $\mathbf{y}:=-\mathbf{K}^*\left(\mathbf{K}\mathbf{y}^{\star}-\mathbf{r}\right)
 -\frac{\gamma}{\beta}\left(\mathbf{y}^{\star}-\mathbf{x}^{\star}\right)+
 \mathbf{y}^{\star}$ we obtain the second equation of \eqref{eq:fixed-eqs_convex_model_solution}.
 \end{proof}


Direct application of Proposition \ref{pro:G_minimizer_exist} to the case with $\mathcal{N}:=N(\mathbf{x}^{\star})$ leads to the next result.

\begin{proposition}\label{pro:F_local_minimizer}
 Let  $\beta, \gamma>0$, and $\left(\mathbf{x}^{\star},\mathbf{y}^{\star}\right)
 \in\mathbb{R}^{N}\times \mathbb{R}^{N}$ be given.  Then, the pair
 $\left(\mathbf{x}^{\star},\mathbf{y}^{\star}\right)$ is a solution of model \eqref{model:convex_model} with $\mathcal{N}:=N(\mathbf{x}^{\star})$ if and only if
\begin{eqnarray}
\mathbf{x}^{\star}={\rm P}_{\mathcal{B}_{{N}(\mathbf{x}^{\star})}}(\mathbf{y}^{\star}),
\ \ \ \mathbf{y}^{\star}=\mathbf{x}^{\star}
-\frac{\beta}{\gamma}\mathbf{K}^*\left(\mathbf{K}\mathbf{y}^{\star}
-\mathbf{r}\right).
\label{eq:fixed-eqs_two_var_iter_Bb}
\end{eqnarray}

\end{proposition}


Combining Theorem \ref{lemma:relation_nonconv_convex} and Proposition \ref{pro:F_local_minimizer} yields the following characterization of a local minimizer of the non-convex model \eqref{model:two_variable}.

\begin{theorem}\label{lemma:relation_nonconv_convex_char}
Let  $\beta, \gamma>0$ be fixed. A pair $\left(\mathbf{x}^{\star},\mathbf{y}^{\star}\right)
\in \mathbb{R}^{N}\times \mathbb{R}^{N}$ is a local minimizer of the non-convex minimization problem \eqref{model:two_variable}
if and only if $(\mathbf{x}^{\star},\mathbf{y}^{\star})$ satisfies equations \eqref{eq:fixed-eqs_two_var_iter_Bb}.
\end{theorem}

We next formulate a necessary condition for a global minimizer of the non-convex minimization problem \eqref{model:two_variable} as a fixed-point of a nonlinear map, and show that a fixed-point of the map is sufficiently a local minimizer of \eqref{model:two_variable}.

\begin{theorem}\label{solution_twovar_model}
Let $\beta, \gamma>0$ be fixed.
If a pair $(\mathbf{x}^{\star},\mathbf{y}^{\star})$ is a
solution of the minimization problem \eqref{model:two_variable}, then $(\mathbf{x}^{\star},\mathbf{y}^{\star})$ satisfies the fixed-point equation
\begin{eqnarray}
\mathbf{x}^{\star}\in\mathrm{prox}_{\beta\|\cdot\|_0} \left(\mathbf{y}^{\star}\right),
\qquad
\mathbf{y}^{\star}=\mathbf{x}^{\star}
-\frac{\beta}{\gamma}\mathbf{K}^*(\mathbf{K}\mathbf{y}^{\star}-\mathbf{r}).
\label{solution_twovar_model_fixedeq}
\end{eqnarray}
Conversely, if a pair $(\mathbf{x}^{\star},\mathbf{y}^{\star})$ satisfies \eqref{solution_twovar_model_fixedeq}, then $(\mathbf{x}^{\star},\mathbf{y}^{\star})$
is a local minimizer of \eqref{model:two_variable}.
\end{theorem}
\begin{proof}
Suppose that a pair $(\mathbf{x}^{\star},\mathbf{y}^{\star})$ is a solution of the minimization problem  \eqref{model:two_variable}.
The first inclusion of \eqref{solution_twovar_model_fixedeq} has been proved in Proposition \ref{pro:model_relation}. It remains to show the second equation of \eqref{solution_twovar_model_fixedeq}.
Clearly, $(\mathbf{x}^{\star},\mathbf{y}^{\star})$ is a local minimizer of the non-convex minimization problem \eqref{model:two_variable}.  By  Theorem \ref{lemma:relation_nonconv_convex_char}, the second equation of \eqref{solution_twovar_model_fixedeq} holds.

We next prove the second part of this theorem.
By Item (iii) of Lemma \ref{P_BN_prox_L0}, we have that $\mathbf{x}^{\star}=
{\rm P}_{\mathcal{B}_{N(\mathbf{x}^{\star})}}(\mathbf{y}^{\star})$.
This together with the second equation of \eqref{solution_twovar_model_fixedeq} yields that $(\mathbf{x}^{\star},\mathbf{y}^{\star})$
is a local minimizer of model~\eqref{model:two_variable},  by employing Theorem \ref{lemma:relation_nonconv_convex_char}.
\end{proof}


Theorem \ref{solution_twovar_model} motivates us to propose a fixed-point algorithm for solving the minimization problem \eqref{model:two_variable}.
Based on the fixed-point equations \eqref{solution_twovar_model_fixedeq} of Theorem \ref{solution_twovar_model}, we propose the iteration algorithm as
\begin{eqnarray}
\left\{
\begin{array}{l}
\mathbf{x}^{k+1}\in\mathrm{prox}_{\beta\|\cdot\|_0}
\left(\mathbf{y}^{k}\right),
\\
\mathbf{y}^{k+1}=\mathbf{x}^{k+1}
-\frac{\beta}{\gamma}\mathbf{K}^{*}(\mathbf{K}\mathbf{y}^{k+1}-\mathbf{r}).\ \ \
\end{array}
\right.\label{eq:fixed-eqs_two_var_iter}
\end{eqnarray}
Updates of both variables $\mathbf{x}$ and $\mathbf{y}$ in Algorithm \eqref{eq:fixed-eqs_two_var_iter}
at each iteration can be efficiently implemented. The first subproblem in \eqref{eq:fixed-eqs_two_var_iter} can be explicitly solved by employing the closed-form formula \eqref{prox_envl0-1} of the proximity operator of the $\ell_0$ norm. Once the value $\mathbf{x}^{k+1}$ is obtained, we can solve the linear system of \eqref{eq:fixed-eqs_two_var_iter} for  $\mathbf{y}^{k+1}$.

The unique solvability of the linear system of \eqref{eq:fixed-eqs_two_var_iter} requires further consideration. To this end, we first exam the eigenvalues of matrix $\mathbf{K}^*\mathbf{K}$.
By the definition of $\mathbf{R}$ and $\mathbf{F}$, and the property of the tight framelet $\mathbf{W}$ that $\mathbf{W}^*\mathbf{W}=\mathbf{I}$ \cite{Cai-Chan-Shen-Shen}, it can be verified that
$
\mathbf{KK}^*=\mathbf{I}.
$
This property of $\mathbf{K}$ leads to the following results.

\begin{lemma}\label{lemma:K*K_eigvalue}
If $d$, $M$ and $N$ are positive integers with $d\leq M\leq N$,  $\mathbf{W}$ is an $N\times M$ real matrix satisfying $\mathbf{W}^*\mathbf{W}=\mathbf{I}$, $\mathbf{R}$ is a $d \times M$ row selection matrix,
then the matrix $\mathbf{K}^*\mathbf{K}$ defined by the second equation of \eqref{def:y_K} is a real matrix and the eigenvalues $\mu_j$ of  $\mathbf{K}^*\mathbf{K}$  are
\begin{eqnarray}
\mu_1=\mu_2=\ldots=\mu_{d}=1,
\ \ \
\mu_{d+1}=\mu_{d+2}=\ldots=\mu_{N-d}=0.
\label{eq:K*K_eigvalue}
\end{eqnarray}
\end{lemma}
\begin{proof}
Using the definitions of $\mathbf{R}$ and $\mathbf{F}$  in Section \ref{sec:MAM2} yields that matrix $\mathbf{F}^*\mathbf{R}^*\mathbf{R}\mathbf{F}$ is real.
Since $\mathbf{W}$ is a real matrix, applying the definition of $\mathbf{K}$ in \eqref{def:y_K} leads to that  $\mathbf{K}^*\mathbf{K}$ is a real matrix. Furthermore,
as $\mathbf{K}\mathbf{K}^*=\mathbf{I}$,
we have that $\mathbf{K}^*\mathbf{K}$ is an idempotent matrix. Immediately, it follows from \cite{Horn-Johnson} that if $\mu$ is an eigenvalue  of the matrix $\mathbf{K}^*\mathbf{K}$, then  $\mu=0$ or $\mu=1$. Let $r(\mathbf{A})$ denote the rank of matrix $\mathbf{A}$. It can be proved that $r(\mathbf{K}^*\mathbf{K})=d$. This together with the above discussion leads to the desired conclusion of this lemma.
\end{proof}


The next lemma follows from Lemma \ref{lemma:K*K_eigvalue}.

\begin{lemma}\label{lemma:y_existence}
If  $\beta, \gamma>0$, then $\mathbf{I}+\frac{\beta}{\gamma}\mathbf{K}^*\mathbf{K}$ is invertible.
\end{lemma}
\begin{proof}
If $\beta, \gamma>0$, employing Lemma \ref{lemma:K*K_eigvalue} yields that $\mathbf{I}+\frac{\beta}{\gamma}\mathbf{K}^*\mathbf{K}$ is a positive definite matrix. Hence, it is invertible.
\end{proof}

Lemma \ref{lemma:y_existence} ensures that given $\mathbf{x}\in \mathbb{R}^N$, the linear system
$$
\mathbf{y}=\mathbf{x}
-\frac{\beta}{\gamma}\mathbf{K}^*(\mathbf{K}\mathbf{y}-\mathbf{r})
$$
has a unique solution $\mathbf{y}\in \mathbb{R}^N$. It can be solved by an internal iteration:
\begin{eqnarray}
\left\{
\begin{array}{l}
\mathbf{v}_k^0=\mathbf{y}^{k},
\\
{\rm For~j\geq 1}
\\
\mathbf{v}_k^{j}=\mathbf{x}^{k+1}
-\frac{\beta}{\gamma}\mathbf{K}^*\left(\mathbf{K}\mathbf{v}_k^{j-1}-\mathbf{r}\right),
\\
\mathbf{y}^{k+1}=\mathbf{v}_k^{\infty}.\ \ \
\end{array}
\right.\label{inner_loop}
\end{eqnarray}
Integrating iteration \eqref{inner_loop} with Algorithm \eqref{eq:fixed-eqs_two_var_iter}, we have the following double-loop iteration algorithm for solving the model \eqref{model:two_variable}:
\begin{eqnarray}
\left\{
\begin{array}{l}
\mathbf{x}^{k+1}\in\mathrm{prox}_{\beta\|\cdot\|_0}
\left(\mathbf{y}^{k}\right),
\\
~~~~~~~\mathbf{v}_k^0=\mathbf{y}^{k},
\\
~~~~~~~ {\rm For~j\geq 1}
\\
~~~~~~~\mathbf{v}_k^{j}=\mathbf{x}^{k+1}
-\frac{\beta}{\gamma}\mathbf{K}^*\left(\mathbf{K}\mathbf{v}_k^{j-1}
-\mathbf{r}\right),
\\
\mathbf{y}^{k+1}=\mathbf{v}_k^{\infty}.\ \ \
\end{array}
\right.\label{eq:fixed-eqs_two_var_double_loop}
\end{eqnarray}


\section{Convergence Analysis }\label{sec:MAM_conv}\setcounter{equation}{0}
In this section, we study the convergence property of Algorithm \eqref{eq:fixed-eqs_two_var_iter}.
Specifically, we show that the support of the sparse variable $\mathbf{x}^k$ generated by Algorithm \eqref{eq:fixed-eqs_two_var_iter} will remain unchanged after a finite number of iterations, and thus, Algorithm \eqref{eq:fixed-eqs_two_var_iter} solving the non-convex optimization problem \eqref{model:two_variable} reduces to solving a convex optimization problem on the support. The convergence analysis of Algorithm \eqref{eq:fixed-eqs_two_var_iter} is then boiled down to analyzing convergence of a convex optimization problem.

We now outline the steps of the convergence analysis.  Firstly, a function $E$  is introduced. Under the assumption that $\left(\mathbf{x},\mathbf{y}\right)\in\mathbb{R}^{N}\times \mathbb{R}^{N}$ satisfy the second equation of \eqref{solution_twovar_model_fixedeq},
the function $F$ defined by \eqref{def:F} is then rewritten by  $E$.  Let $\{(\mathbf{x}^k,\mathbf{y}^k)\}$ be a sequence generated by Algorithm \eqref{eq:fixed-eqs_two_var_iter}. To prove that the sequence $\{F(\mathbf{x}^{k},\mathbf{y}^{k})\}$ is convergent,
the property of  $E$ is further explored to present a relation between $E(\mathbf{y}^{k+1})$ and $E(\mathbf{y}^k)$. Applying the property of $\alpha$-averaged nonexpansive operators, the sequence $\{(\mathbf{x}^k,\mathbf{y}^k)\}$ is then proved to converge to   a minimizer $(\mathbf{x}^{\star},\mathbf{y}^{\star})$ of the convex optimization model \eqref{model:convex_model} with $\mathcal{N}: =N(\mathbf{x}^{\star})$. This together  with Theorem
\ref{lemma:relation_nonconv_convex} shows that $(\mathbf{x}^{\star},\mathbf{y}^{\star})$ is a local minimizer of the non-convex model \eqref{model:two_variable}.  Finally, it follows from Proposition \ref{pro:model_relation_localm} that $\mathbf{y}^{\star}$ is a local minimizer of the proposed model \eqref{model:general}.

We first consider a function $E$, which is closely related to both functions $F$ and $G$. Specifically, we define  $E~:\mathbb{R}^{d}\times\mathbb{R}^{N}\rightarrow\mathbb{R}$ at $\mathbf{y}\in\mathbb{R}^{N}$ by
\begin{eqnarray}
E(\mathbf{r},\mathbf{y})~:=\frac{L}{2}\left\|\mathbf{K}\mathbf{y}-\mathbf{r}\right\|_2^2,
\label{def:E_x}
\end{eqnarray}
where $L:=1+\frac{\beta}{\gamma}$.
As $\mathbf{r}\in\mathbb{R}^{d}$ in the problem of inverting incomplete Fourier transform  \eqref{seismogram-model-origin-app} is fixed,  we write $E(\mathbf{r},\mathbf{y})$ as $E(\mathbf{y})$ for short.
In the following lemma, we rewrite the objective function $F$ given in \eqref{model:two_variable} in terms of $E$.

\begin{lemma}\label{lemma:E_G_F_relation}
Let $\mathbf{x}\in\mathbb{R}^{N}$. If $\mathbf{y}=\mathbf{x}
-\frac{\beta}{\gamma}\mathbf{K}^*(\mathbf{K}\mathbf{y}-\mathbf{r})$,
then
\begin{eqnarray}
G(\mathbf{x},\mathbf{y})=E(\mathbf{y})
 \ \ {\rm and} \ \ F(\mathbf{x},\mathbf{y})=E(\mathbf{y})+\gamma\left\|\mathbf{x}\right\|_0.
 \label{relation_EFG}
 \end{eqnarray}
\end{lemma}
\begin{proof}
We prove the first equation of \eqref{relation_EFG}. If
$\mathbf{y}=\mathbf{x}
-\frac{\beta}{\gamma}\mathbf{K}^*(\mathbf{K}\mathbf{y}-\mathbf{r})$,  there holds
$$
\frac{\gamma}{2\beta}\left\|\mathbf{x}-\mathbf{y}\right\|_2^2
=\frac{\gamma}{2\beta}\left\|\frac{\beta}{\gamma}\mathbf{K}^* \mathbf{K}\mathbf{y}-\frac{\beta}{\gamma}\mathbf{K}^*\mathbf{r}\right\|_2^2.
$$
Note that $\mathbf{KK}^*=\mathbf{I}$. Combining this with the definition of $\|\cdot\|_2$, the above equation leads to
$
\frac{\gamma}{2\beta}\left\|\mathbf{x}-\mathbf{y}\right\|_2^2
=\frac{\beta}{2\gamma}\left\|\mathbf{K}\mathbf{y}-\mathbf{r}\right\|_2^2.
$
This together with the definition \eqref{def:G} of $G$   yields the first equation of \eqref{relation_EFG}.
The second equation of \eqref{relation_EFG} follows from the first equation and the relation between $F$  and  $G$.
\end{proof}

We next explore the property that the function $E$ is bounded above by a quadratic function. This property plays a crucial role in our convergence analysis.

\begin{lemma}\label{lemma:E_property}
If $E$ is defined by \eqref{def:E_x}, then
 \begin{eqnarray*}
E(\mathbf{z})\leq E(\mathbf{s})+\left<\nabla E(\mathbf{s}),\mathbf{z}-\mathbf{s}\right>
+\frac{L}{2}\|\mathbf{z}-\mathbf{s}\|_2^2,
\ \ {\rm for\ all}~\mathbf{z},\mathbf{s}\in\mathbb{R}^{N}.
 \end{eqnarray*}

\end{lemma}
\begin{proof}
We first show that $\nabla E$ is Lipschitz continuous with Lipschitz constant $L$. It is observed that
$\left\|\nabla E(\mathbf{z})-\nabla E(\mathbf{s})\right\|_2^2 =L^2
\left\|\mathbf{K}^* \mathbf{K}(\mathbf{z}-\mathbf{s})\right\|_2^2.$
As $\|\mathbf{K}^*\mathbf{K}\|_2=1$,
we obtain that
$
 \|\nabla E(\mathbf{z})-\nabla E(\mathbf{s})\|_2^2
 \leq L^2
 \|\mathbf{z}-\mathbf{s}\|_2^2,
$
which ensures that $\nabla E$ is Lipschitz continuous with Lipschitz constant $L$.
The result of this lemma follows immediately from the well-known property of a differentiable convex function with a Lipschitz continuous gradient.
\end{proof}


We follow \cite{SXZ,ZSX,ZLKSZX} to establish a convergence theorem of Algorithm \eqref{eq:fixed-eqs_two_var_iter}.

\begin{theorem}\label{theorem:conv_analysis}
Let $\left\{(\mathbf{x}^k,\mathbf{y}^k)\right\}$
be a sequence generated by Algorithm \eqref{eq:fixed-eqs_two_var_iter} with an initial
$(\mathbf{x}^0,\mathbf{y}^0)\in\mathbb{R}^{N}\times\mathbb{R}^{N}$ for model \eqref{model:two_variable}. If $\beta,\gamma$ are positive numbers and  $0<\frac{\beta}{\gamma}<\frac{\sqrt{5}-1}{2}$, then the following statements hold:
\begin{align}
&(i)~~ F(\mathbf{x}^{k+1},\mathbf{y}^{k+1})\leq F(\mathbf{x}^{k},\mathbf{y}^{k}) \ for\ all\ k\geq 0 \ and\ the \ sequence\  \{F(\mathbf{x}^{k},\mathbf{y}^{k})\}\ converges.
\nonumber
\\
&(ii)~~The\ sequence \ \{(\mathbf{x}^{k},\mathbf{y}^{k})\} \ has \ a \ finite
\ length, \ that\ is
\nonumber
\\
& ~~~~~~~~
\sum_{k=0}^{+\infty}\left\|\mathbf{x}^{k+1}-\mathbf{x}^{k}\right\|_2^2<+\infty,
\ \ \ \ \ \ \ \ \ \
\sum_{k=0}^{+\infty}\left\|\mathbf{y}^{k+1}-\mathbf{y}^{k}\right\|_2^2<+\infty.
\label{sum_x_y_ser}
\\
&(iii)~~\lim\limits_{k\rightarrow \infty}\left\|\mathbf{x}^{k+1}-\mathbf{x}^{k}\right\|_2
=\lim\limits_{k\rightarrow \infty}\left\|\mathbf{y}^{k+1}-\mathbf{y}^{k}\right\|_2=0.
\nonumber
\end{align}
\end{theorem}
\begin{proof}
We first prove Item (i). The second part of Item (i) follows directly from the first part and the fact that $F(\mathbf{x}^{k},\mathbf{y}^{k})\geq 0$ for all positive integer $k$. It remains to prove the first part of Item (i).
By the second equation in Algorithm \eqref{eq:fixed-eqs_two_var_iter}
and the second equation of \eqref{relation_EFG} in Lemma \ref{lemma:E_G_F_relation}, we have that
$
F\left(\mathbf{x}^{k+1},\mathbf{y}^{k+1}\right)=E\left(\mathbf{y}^{k+1}\right)+\gamma\left\|\mathbf{x}^{k+1}\right\|_0.
$
Combining this equation with Lemma \ref{lemma:E_property} yields that
\begin{eqnarray}
F\left(\mathbf{x}^{k+1},\mathbf{y}^{k+1}\right)
\leq E\left(\mathbf{y}^{k}\right)+\left<\nabla E\left(\mathbf{y}^{k}\right),\mathbf{y}^{k+1}-\mathbf{y}^{k}\right>
+\frac{L}{2}\left\|\mathbf{y}^{k+1}-\mathbf{y}^{k}
\right\|_2^2+\gamma\left\|\mathbf{x}^{k+1}\right\|_0.
\label{iter_analysis_1}
\end{eqnarray}
We next exam the second and third terms on the right hand side of inequality \eqref{iter_analysis_1}. Specifically, we shall establish that
\begin{eqnarray}
\left<\nabla E\left(\mathbf{y}^{k}\right),\mathbf{y}^{k+1}-\mathbf{y}^{k}\right>=\frac{\gamma}{\gamma+\beta}\left<\nabla E\left(\mathbf{y}^{k}\right),\mathbf{x}^{k+1}-\mathbf{x}^{k}\right>,
\label{iter_analysis_F1W}
\end{eqnarray}
and
\begin{eqnarray}
L\left\|\mathbf{y}^{k+1}-\mathbf{y}^{k}\right\|_2^2
\leq \frac{\gamma}{\beta}\left\|\mathbf{x}^{k+1}-\mathbf{x}^{k}\right\|_2^2.
\label{iter_analysis_F2W}
\end{eqnarray}

We first prove \eqref{iter_analysis_F1W}. To this end,
we differentiate $E$ defined by \eqref{def:E_x} to get that
\begin{eqnarray}
\nabla E\left(\mathbf{y}^{k}\right)
=L\left(\mathbf{K}^* \mathbf{K}\mathbf{y}^{k}-\mathbf{K}^*\mathbf{r}\right).
\label{iter_analysis_2}
\end{eqnarray}
Combining the second equation in \eqref{eq:fixed-eqs_two_var_iter} and equation \eqref{iter_analysis_2}, we obtain that
$$
\frac{\gamma}{\gamma+\beta}\left<\nabla E\left(\mathbf{y}^{k}\right),\mathbf{x}^{k+1}-\mathbf{x}^{k}\right>
=\left<\left(\mathbf{K}^* \mathbf{K}\mathbf{y}^{k}-\mathbf{K}^*\mathbf{r}\right),\mathbf{y}^{k+1}
-\mathbf{y}^{k}+\frac{\beta}{\gamma}\mathbf{K}^* \mathbf{K}\left(\mathbf{y}^{k+1}-\mathbf{y}^{k}\right)
\right>.
$$
Expanding the right hand side of the equation above with the fact that $\mathbf{KK}^*=\mathbf{I}$ and combining the like terms with noticing the definition of $L$ yield
$$
\frac{\gamma}{\gamma+\beta}\left<\nabla E\left(\mathbf{y}^{k}\right),\mathbf{x}^{k+1}-\mathbf{x}^{k}\right>
=\left<L\left(\mathbf{K}^* \mathbf{K}\mathbf{y}^{k}-\mathbf{K}^*\mathbf{r}\right),\mathbf{y}^{k+1}-\mathbf{y}^{k}\right>.
$$
Substituting the left hand side of equation \eqref{iter_analysis_2} into the right hand side of the above equation yields equation \eqref{iter_analysis_F1W}.

We now prove that inequality \eqref{iter_analysis_F2W} holds for $0<\frac{\beta}{\gamma} <\frac{\sqrt{5}-1}{2}$. By using the second equation in Algorithm \eqref{eq:fixed-eqs_two_var_iter}, expanding the resulting expression and employing the relation $\mathbf{KK}^*=\mathbf{I}$, we get that
\begin{equation}
\left\|\mathbf{x}^{k+1}-\mathbf{x}^{k}\right\|_2^2
=\left\|\mathbf{y}^{k+1}-\mathbf{y}^{k}\right\|_2^2
+\left(\frac{2\beta}{\gamma}+\frac{\beta^2}{\gamma^2}\right)\left\|\mathbf{K}\left(\mathbf{y}^{k+1}-\mathbf{y}^{k}\right)\right\|_2^2.
\label{iter_analysis_F3W}
\end{equation}
Equation \eqref{iter_analysis_F3W} implies that
\begin{equation}
\left\|\mathbf{y}^{k+1}-\mathbf{y}^{k}\right\|_2^2
\leq\left\|\mathbf{x}^{k+1}-\mathbf{x}^{k}\right\|_2^2.
\label{iter_analysis_F3W**}
\end{equation}
Since $0<\frac{\beta}{\gamma} <\frac{\sqrt{5}-1}{2}$, there holds
\begin{eqnarray}
L= 1+\frac{\beta}{\gamma}<\frac{\gamma}{\beta}.
 \label{iter_analysis_betagammaW}
\end{eqnarray}
Combining inequalities \eqref{iter_analysis_F3W**} and \eqref{iter_analysis_betagammaW}  yields  inequality \eqref{iter_analysis_F2W}.


Substituting \eqref{iter_analysis_F1W} and \eqref{iter_analysis_F2W} into the right hand side of inequality \eqref{iter_analysis_1} yields
\begin{eqnarray}
F\left(\mathbf{x}^{k+1},\mathbf{y}^{k+1}\right)
\leq E\left(\mathbf{y}^{k}\right)+\frac{\gamma}{\gamma+\beta}\left<\nabla E\left(\mathbf{y}^{k}\right),\mathbf{x}^{k+1}-\mathbf{x}^{k}\right>
+\frac{\gamma}{2\beta}\left\|\mathbf{x}^{k+1}-\mathbf{x}^{k}
\right\|_2^2+\gamma\left\|\mathbf{x}^{k+1}\right\|_0.
\label{iter_analysis_1***}
\end{eqnarray}
To prove the first part of Item (i), it suffices to show
\begin{equation}
   \frac{\gamma}{\gamma+\beta}\left<\nabla E\left(\mathbf{y}^{k}\right),\mathbf{x}^{k+1}-\mathbf{x}^{k}\right>
+\frac{\gamma}{2\beta}\left\|\mathbf{x}^{k+1}-\mathbf{x}^{k}
\right\|_2^2+\gamma\left\|\mathbf{x}^{k+1}\right\|_0\leq \gamma\left\|\mathbf{x}^k\right\|_0.
\label{iter_analysis_last}
\end{equation}
To this end, we note that equation \eqref{iter_analysis_2} together with the second equation in Algorithm \eqref{eq:fixed-eqs_two_var_iter} leads to
$
\mathbf{y}^k=\mathbf{x}^k
-\frac{\beta}{\gamma+\beta}\nabla E\left(\mathbf{y}^{k}\right).
$
Substituting this equation into ${\rm prox}_{\beta\|\cdot\|_0}(\mathbf{y}^k)$ and using the definition of the proximity operator \eqref{def:prox_beta_ell0}, we may rewrite
the first part of Algorithm \eqref{eq:fixed-eqs_two_var_iter}
as
\begin{eqnarray}
\mathbf{x}^{k+1}\in {\rm argmin}
\left\{\frac{1}{2\beta}
\left\|\mathbf{x}-\mathbf{x}^k
+\frac{\beta}{\gamma+\beta}\nabla E\left(\mathbf{y}^{k}\right)\right\|_2^2+\left\|\mathbf{x}\right\|_0,\ \ \ \mathbf{x}\in\mathbb{R}^{N}\right\}.
\label{iter_analysis_3}
\end{eqnarray}
Expanding the quadratic term
in \eqref{iter_analysis_3} as
\begin{eqnarray*}
\frac{\beta}{2(\gamma+\beta)^2}\left\|\nabla E\left(\mathbf{y}^{k}\right)\right\|_2^2
+\frac{1}{\gamma+\beta}\left<\nabla E\left(\mathbf{y}^{k}\right),\mathbf{x}-\mathbf{x}^{k}\right>
+\frac{1}{2\beta}\left\|\mathbf{x}-\mathbf{x}^{k}\right\|_2^2
\end{eqnarray*}
and noticing that the first term of the above expression is constant with respect to $\mathbf{x}$, the inclusion relation \eqref{iter_analysis_3} becomes
\begin{eqnarray*}
\mathbf{x}^{k+1}\in {\rm argmin}
\left\{\frac{1}{\gamma+\beta}\left<\nabla E\left(\mathbf{y}^{k}\right),\mathbf{x}-\mathbf{x}^{k}\right>
+\frac{1}{2\beta}\left\|\mathbf{x}-\mathbf{x}^{k}\right\|_2^2+\left\|\mathbf{x}\right\|_0,\ \ \ \mathbf{x}\in\mathbb{R}^{N}\right\}.
\end{eqnarray*}
This ensures the validity of
\eqref{iter_analysis_last} and thus completes the proof of the first part of Item (i).

We next prove Item (ii). It follows from inequality \eqref{iter_analysis_last} that
\begin{eqnarray*}
E(\mathbf{y}^{k})+\frac{\gamma}{\gamma+\beta}\left<\nabla E(\mathbf{y}^{k}),\mathbf{x}^{k+1}-\mathbf{x}^{k}\right>
+\frac{\gamma}{2\beta}\left\|\mathbf{x}^{k+1}-\mathbf{x}^{k}
\right\|_2^2+\gamma\left\|\mathbf{x}^{k+1}\right\|_0\leq F\left(\mathbf{x}^{k},\mathbf{y}^{k}\right).
\end{eqnarray*}
In addition, from \eqref{iter_analysis_1} we get that
\begin{eqnarray*}
 -E(\mathbf{y}^{k})-\left<\nabla E(\mathbf{y}^{k}),\mathbf{y}^{k+1}-\mathbf{y}^{k}\right>
-\frac{L}{2}\left\|\mathbf{y}^{k+1}-\mathbf{y}^{k}
\right\|_2^2-\gamma\left\|\mathbf{x}^{k+1}\right\|_0
\leq -F\left(\mathbf{x}^{k+1},\mathbf{y}^{k+1}\right).
\end{eqnarray*}
Summing the above  two inequalities and noticing that  \eqref{iter_analysis_F1W} holds, we obtain that
\begin{eqnarray*}
\frac{\gamma}{2\beta}\left\|\mathbf{x}^{k+1}-\mathbf{x}^{k}\right\|_2^2
-\frac{L}{2}\left\|\mathbf{y}^{k+1}-\mathbf{y}^{k}
\right\|_2^2
\leq F(\mathbf{x}^{k},\mathbf{y}^{k})-F(\mathbf{x}^{k+1},\mathbf{y}^{k+1}).
\end{eqnarray*}
Again, since $0<\frac{\beta}{\gamma} <\frac{\sqrt{5}-1}{2}$,  inequality \eqref{iter_analysis_betagammaW} holds.  Substituting \eqref{iter_analysis_F3W}  into above inequality  yields
\begin{eqnarray}
&& 0\leq
\frac{1}{2}\left(\frac{\gamma}{\beta}-L\right)\left\|\mathbf{y}^{k+1}-\mathbf{y}^{k}\right\|_2^2
+\left(1+\frac{\beta}{2\gamma}\right)\left\|\mathbf{K}(\mathbf{y}^{k+1}-\mathbf{y}^{k})\right\|_2^2
\nonumber
\\
&&\quad
\leq F(\mathbf{x}^{k},\mathbf{y}^{k})-F(\mathbf{x}^{k+1},\mathbf{y}^{k+1}).
\label{eq:cov_xk_Kyk}
\end{eqnarray}
For $V\geq 1$, summing  inequality \eqref{eq:cov_xk_Kyk} from
$k=0$ to $V-1$ yields
\begin{eqnarray*}
0\leq\sum_{k=0}^{V-1}
\left[\frac{1}{2}\left(\frac{\gamma}{\beta}-L\right)\left\|\mathbf{y}^{k+1}-\mathbf{y}^{k}\right\|_2^2
+\left(1+\frac{\beta}{2\gamma}\right)\left\|\mathbf{K}\left(\mathbf{y}^{k+1}-\mathbf{y}^{k}
\right)\right\|_2^2\right]
\leq F(\mathbf{x}^{0},\mathbf{y}^{0})-F(\mathbf{x}^{V},\mathbf{y}^{V}).
\end{eqnarray*}
Noticing that the sequence $\{F(\mathbf{x}^{k},\mathbf{y}^{k})\}$ converges,  $\beta>0,\ \gamma>0$ and \eqref{iter_analysis_betagammaW}  holds,   letting $V\rightarrow\infty$, the following inequality is obtained
\begin{eqnarray*}
\sum_{k=0}^{+\infty}
\left[\frac{1}{2}\left(\frac{\gamma}{\beta}-L\right)\left\|\mathbf{y}^{k+1}-\mathbf{y}^{k}\right\|_2^2
+\left(1+\frac{\beta}{2\gamma}\right)\left\|\mathbf{K}\left(\mathbf{y}^{k+1}-\mathbf{y}^{k}\right)\right\|_2^2\right]
< +\infty.
\end{eqnarray*}
Hence, we have
\begin{eqnarray}
\sum_{k=0}^{+\infty}
\left\|\mathbf{y}^{k+1}-\mathbf{y}^{k}\right\|_2^2
< +\infty,
\ \ \ \ \ \
 \sum_{k=0}^{+\infty}
\left\|\mathbf{K}\left(\mathbf{y}^{k+1}-\mathbf{y}^{k}\right)\right\|_2^2
< +\infty.\label{ineq:y_k_conv}
\end{eqnarray}
This together with \eqref{iter_analysis_F3W} yields
\begin{eqnarray}
\sum_{k=0}^{+\infty}
\left\|\mathbf{x}^{k+1}-\mathbf{x}^{k}\right\|_2^2
< +\infty.
\label{ineq:x_k_conv}
\end{eqnarray}
Therefore,  Item (ii) is obtained.

Item (iii) is obtained directly from Item (ii).
\end{proof}

We next confirm the existence of the invariant support set of the sequence generated by Algorithm \eqref{eq:fixed-eqs_two_var_iter} for the non-convex model \eqref{model:two_variable}.

\begin{lemma}\label{lemma:x_sparsity_two}
Let $\left\{(\mathbf{x}^k,\mathbf{y}^k)\right\}$
be a sequence generated by algorithm \eqref{eq:fixed-eqs_two_var_iter} with an initial
$(\mathbf{x}^0,\mathbf{y}^0)\in\mathbb{R}^{N}\times\mathbb{R}^{N}$ for model \eqref{model:two_variable}.
If $\beta,\gamma$ are positive numbers with $0<\frac{\beta}{\gamma}<\frac{\sqrt{5}-1}{2}$, then there exists a $V>0$ such that
$\ N(\mathbf{x}^k)=N(\mathbf{x}^{V})$  for all $k\geq V$.
\end{lemma}
\begin{proof}
 Item (iii) of Theorem \ref{theorem:conv_analysis} implies that there exists a number $V>0$ such that
\begin{eqnarray*}
\left\|\mathbf{x}^{k+1}-\mathbf{x}^{k}\right\|_2<\sqrt{2\beta},
\ \ \ \ {\rm for\ all} \ k\geq V.
\end{eqnarray*}
As $\mathbf{x}^{k}\in\mathrm{prox}_{\beta\|\cdot\|_0}
\left(\mathbf{y}^{k-1}\right)$, by  (ii) of Lemma \ref{lemma:x_sparsity}, sets $N(\mathbf{x}^k)$ for all $k\geq V$ must be identical.
\end{proof}

We next show that the sequence $\left\{(\mathbf{x}^k,\mathbf{y}^k)\right\}$ generated by Algorithm \eqref{eq:fixed-eqs_two_var_iter} with an initial
$(\mathbf{x}^0,\mathbf{y}^0)\in\mathbb{R}^{N}\times\mathbb{R}^{N}$ converges. We need the notion of the nonexpansive averaged operator.
\begin{definition}\label{def:nonexpansive_operator}
A nonlinear operator $P:\mathbb{R}^{N}\rightarrow\mathbb{R}^{N}$ is called nonexpansive if
$$
\left\|P\left(\mathbf{x}\right)-P\left(\mathbf{y}\right)\right\|_2
\leq\left\|\mathbf{x}-\mathbf{y}\right\|_2,
\ \ \  {\rm for\ all}\ \mathbf{x},\mathbf{y}\in\mathbb{R}^{N}.
$$
\end{definition}
\begin{definition}\label{def:nonexpansive_averaged_operator}
A nonlinear operator $P:\mathbb{R}^{N}\rightarrow\mathbb{R}^{N}$ is called nonexpansive averaged if
there are a number $\alpha\in\left(0,1\right)$ and a nonexpansive operator $S$ such that
$$
P=\alpha I+\left(1-\alpha\right)S,
$$
where $I$ is the identity operator. In this case, $P$ is called nonexpansive $\alpha$-averaged.
\end{definition}

For a nonlinear operator $P:\mathbb{R}^{N}\rightarrow\mathbb{R}^{N}$ and any vector  $\mathbf{x}^0\in \mathbb{R}^{N}$, the sequence $\mathbf{x}^{k+1}=P(\mathbf{x}^k)$, $k=0, 1, \dots$ is called a Picard sequence of $P$. It is known  \cite{BLT} that a Picard sequence of a nonexpansive averaged operator converges to a fixed-point of $P$. We state this result next.

\begin{theorem}\label{theorem:averaged_nonexp_con}
Let $\alpha\in(0,1)$, $P:\mathbb{R}^{N}\rightarrow\mathbb{R}^{N}$ be an $\alpha$-averaged nonexpansive operator such that
$Fix(P)\neq \emptyset$. Here, $Fix(P)$ denotes the set of fixed points of $P$. Then for any given $\mathbf{x}^0\in\mathbb{R}^{N}$, the Picard sequence $\left\{\mathbf{x}^k\right\}_{k\geq 1}$ of $P$ converges to a
$\mathbf{x}^*\in Fix(P)$.
\end{theorem}

Convergence of the sequence $\left\{(\mathbf{x}^k,\mathbf{y}^k)\right\}$ generated by Algorithm \eqref{eq:fixed-eqs_two_var_iter} will be analyzed by showing that the sequence $\left\{\mathbf{x}^k\right\}$ is a Picard sequence of a nonexpansive averaged operator which is the composition of three operators.
To this end, we need a lemma regarding the composition of two nonexpansive averaged operators \cite{CY}.

\begin{lemma}\label{lemma:compos_nonexp_ave_op}
If $P_1$ is nonexpansive $\alpha_1$-averaged and $P_2$ is nonexpansive $\alpha_2$-averaged, then
$P_1\circ P_2$ is nonexpansive $\alpha_1\alpha_2$-averaged.
\end{lemma}

We consider three operators involved in Algorithm \eqref{eq:fixed-eqs_two_var_iter}.  Lemma \ref{lemma:x_sparsity_two} allows us to identify a positive integer $V$ such that $N(\mathbf{x}^k)=N(\mathbf{x}^{V})$ for all $k\geq V$. If we set $\mathcal{N}:=N(\mathbf{x}^{V})$, then the next lemma shows that the projection $\mathrm{P}_{\mathcal{B}_{\mathcal{N}}}$ defined by \eqref{P_B} is a nonexpansive averaged operator.

\begin{lemma}\label{lemma:P_N_alpha1_ave}
If $\mathcal{N}:=N(\mathbf{x}^{V})$, then the  projection $\mathrm{P}_{\mathcal{B}_{\mathcal{N}}}$ defined by \eqref{P_B} is nonexpansive $\frac{1}{2}$-averaged.
\end{lemma}
\begin{proof}
For the index set $\mathcal{N}$, we let
$
\mathbf{S}_\mathcal{N}~:={\rm diag}\left(s_1,s_2,\ldots,s_N\right),
$
where $s_j:=1$ if $j\in\mathcal{N}$ and $s_j:=0$ otherwise.
We observe that $\mathrm{P}_{\mathcal{B}_{\mathcal{N}}}\left(\mathbf{x}\right)
=\mathbf{S}_\mathcal{N}\mathbf{x}$, for all $\mathbf{x}\in\mathbb{R}^{N}$.
We define another matrix
$
\tilde{\mathbf{S}}_\mathcal{N}~:={\rm diag}\left(\tilde{s}_1,\tilde{s}_2,\ldots,\tilde{s}_N\right),
$
where
$\tilde{s}_j~:= 1$, if $j\in\mathcal{N}$ and $\tilde{s}_j~:= -1$,  otherwise.
It can be verified that
$
\mathbf{S}_\mathcal{N}=\frac{1}{2}\mathbf{I}+\left(1-\frac{1}{2}\right)\tilde{\mathbf{S}}_\mathcal{N}.
$
Clearly, $\|\tilde{\mathbf{S}}_\mathcal{N}\|_2= 1$,
which implies that $\tilde{\mathbf{S}}_\mathcal{N}$ is nonexpansive. By the definition of nonexpansive averaged operators, $\mathbf{S}_\mathcal{N}$ is  nonexpansive $\frac{1}{2}$-averaged. In other words, the  projection $\mathrm{P}_{\mathcal{B}_{\mathcal{N}}}$  is  nonexpansive $\frac{1}{2}$-averaged.
\end{proof}

Lemma \ref{lemma:y_existence} shows that for $\beta>0,\ \gamma>0$, the matrix $\mathbf{I}+\frac{\beta}{\gamma}\mathbf{K}^*\mathbf{K}$ is invertible. We next show that the matrix $\left(\mathbf{I}+\frac{\beta}{\gamma}\mathbf{K}^*\mathbf{K}\right)^{-1}$ is nonexpansive $\frac{1}{2}$-averaged.

\begin{lemma}\label{lemma:IK*K_alpha2_ave}
If $\beta, \gamma>0$ with $0<\frac{\beta}{\gamma}<1$, then the  matrix $\left(\mathbf{I}+\frac{\beta}{\gamma}\mathbf{K}^*\mathbf{K}\right)^{-1}$ is nonexpansive $\frac{1}{2}$-averaged.
\end{lemma}
\begin{proof}
We write $\left(\mathbf{I}+\frac{\beta}{\gamma}\mathbf{K}^*\mathbf{K}\right)^{-1}
=\frac{1}{2}\mathbf{I}+\left(1-\frac{1}{2}\right)\mathbf{S}$ with
$\mathbf{S}~:=2\left(\mathbf{I}+\frac{\beta}{\gamma}\mathbf{K}^*
\mathbf{K}\right)^{-1}-\mathbf{I}$. It suffices to show that $\mathbf{S}$
is nonexpansive. For all $\mathbf{x},\mathbf{y}\in\mathbb{R}^{N}$, we have that
\begin{equation}
    \left\|\mathbf{Sx}-\mathbf{Sy}\right\|_2\leq 2\left\|\left(\mathbf{I}+\frac{\beta}{\gamma}\mathbf{K}^*
\mathbf{K}\right)^{-1}
-\frac{1}{2}\mathbf{I}\right\|_2\left\|\mathbf{x}-\mathbf{y}\right\|_2.
\label{equ_Sx_Sy}
\end{equation}
Note that
\begin{eqnarray}\label{eqn_IKK}
\left(\mathbf{I}+\frac{\beta}{\gamma}\mathbf{K}^*\mathbf{K}\right)^{-1}
-\frac{1}{2}\mathbf{I}
=\frac{1}{2}\left(\mathbf{I}+\frac{\beta}{\gamma}
\mathbf{K}^*\mathbf{K}\right)^{-1}
\left(\mathbf{I}-\frac{\beta}{\gamma}\mathbf{K}^*\mathbf{K}\right).
\end{eqnarray}
According to Lemma \ref{lemma:K*K_eigvalue}, simple computation leads to
the eigenvalues  of $\left(\mathbf{I}+\frac{\beta}{\gamma}\mathbf{K}^*\mathbf{K}\right)^{-1}$ are either $\mu=1$ or $\mu=\frac{\gamma}{\beta+\gamma}$. Likewise, we obtain the eigenvalues of $\left(\mathbf{I}-\frac{\beta}{\gamma}\mathbf{K}^*\mathbf{K}\right)$ are $\iota=1$ or $\iota=1-\frac{\beta}{\gamma}$.
Moreover, both matrix $\left(\mathbf{I}+\frac{\beta}{\gamma}\mathbf{K}^*\mathbf{K}\right)^{-1}$ and $\left(\mathbf{I}-\frac{\beta}{\gamma}\mathbf{K}^*\mathbf{K}\right)$ are symmetric. Therefore,
\begin{eqnarray}\label{eig_IKK}
 \left\|\left(\mathbf{I}+\frac{\beta}{\gamma}\mathbf{K}^*\mathbf{K}\right)^{-1}
 \right\|_2\leq \max\left\{1,\frac{\gamma}{\beta+\gamma}\right\}\leq 1
\end{eqnarray}
and
\begin{eqnarray}\label{eig_alpha_IKK}
\left\|\mathbf{I}-\frac{\beta}{\gamma}\mathbf{K}^*\mathbf{K}
 \right\|_2\leq \max\left\{1,1-\frac{\beta}{\gamma}\right\}\leq 1.
\end{eqnarray}
Combining \eqref{eqn_IKK},  \eqref{eig_IKK} and \eqref{eig_alpha_IKK} yields
$$
 \left\|\left(\mathbf{I}+\frac{\beta}{\gamma}\mathbf{K}^*\mathbf{K}\right)^{-1}
-\frac{1}{2}\mathbf{I}\right\|_2\leq \frac{1}{2}.
$$
Substituting the above inequality into inequality \eqref{equ_Sx_Sy} leads to
$
\left\|\mathbf{S}\left(\mathbf{x}\right)-\mathbf{S}\left(\mathbf{y}\right)\right\|_2
\leq\left\|\mathbf{x}-\mathbf{y}\right\|_2,
$
that is, $\mathbf{S}$ is nonexpansive. Therefore, $\left(\mathbf{I}+\frac{\beta}{\gamma}\mathbf{K}^*\mathbf{K}\right)^{-1}$  is nonexpansive $\frac{1}{2}$-averaged.
\end{proof}

We now consider the translation operator defined by $H\mathbf{x}~:=\mathbf{x}+\frac{\beta}{\gamma}\mathbf{K}^*\mathbf{r}$, for all $\mathbf{x}\in\mathbb{R}^{N}$.

\begin{lemma}\label{lemma:H_aver_operator}
The operator $H$ is nonexpansive $\alpha$-averaged for any $\alpha\in (0,1)$. \end{lemma}
\begin{proof}
We write
$H=\alpha I+\left(1-\alpha\right)S$, for any fixed $\alpha\in (0,1)$, with
$$
S\mathbf{x}~:=\mathbf{x}+\frac{\beta}{\left(1-\alpha\right)\gamma}\mathbf{K}^*
\mathbf{r}, \ \ \mbox{for all}\ \ \mathbf{x}\in\mathbb{R}^{N}.
$$
Clearly, for all $\mathbf{x},\mathbf{y}\in\mathbb{R}^{N}$, we observe that
$\left\|S\mathbf{x}-S\mathbf{y}\right\|_2=\|\mathbf{x}-\mathbf{y}\|_2$,
which implies that $S$ is nonexpansive. Hence, $H$ is nonexpansive $\alpha$-averaged.
\end{proof}

For an index set $\mathcal{N}$ of $\{1,2,\dots, N\}$, we define
$$
Q_{\mathcal{N}}:=\mathrm{P}_{\mathcal{B}_{\mathcal{N}}}\circ\left(\mathbf{I}+\frac{\beta}{\gamma}\mathbf{K}^*\mathbf{K}\right)^{-1}\circ H.
$$

\begin{lemma}\label{composition-operator}
The operator $Q_{\mathcal{N}}$ is  nonexpansive $\frac{1}{4}\alpha$-averaged, for any $\alpha\in (0,1)$.
\end{lemma}
\begin{proof}
This result follows directly from Lemma \ref{lemma:compos_nonexp_ave_op} with Lemmas \ref{lemma:P_N_alpha1_ave}, \ref{lemma:IK*K_alpha2_ave} and \ref{lemma:H_aver_operator}.
\end{proof}

Algorithm \eqref{eq:fixed-eqs_two_var_iter} is essentially solving the convex optimization model \eqref{model:convex_model} with$\ \mathcal{N}:=N(\mathbf{x}^{V})$, for $k\geq V$. We confirm this fact in the next lemma.

\begin{lemma}\label{theorem:conv_analysis_app}
Let $\left\{(\mathbf{x}^k,\mathbf{y}^k)\right\}$ be a sequence generated by Algorithm \eqref{eq:fixed-eqs_two_var_iter} with an initial
$(\mathbf{x}^0,\mathbf{y}^0)\in\mathbb{R}^{N}\times\mathbb{R}^{N}$.  Let $V$ be a positive integer such that $N(\mathbf{x}^k)=N(\mathbf{x}^{V})$  for all $k\geq V$.
If  $\beta,\gamma>0$ are chosen such that $0<\frac{\beta}{\gamma} <\frac{\sqrt{5}-1}{2}$,
then the subsequence $\left\{(\mathbf{x}^k,\mathbf{y}^k)\right\}_{k\geq V}$  converges to a solution
$(\mathbf{x}^{\star},\mathbf{y}^{\star})$ of the convex model \eqref{model:convex_model} with $\mathcal{N}:=N(\mathbf{x}^{V})$. In addition, there holds $\mathcal{N}=N(\mathbf{x}^{\star})$.
\end{lemma}
\begin{proof}
By Proposition \ref{pro:G_minimizer_exist}, model \eqref{model:convex_model} with $\mathcal{N}:=N(\mathbf{x}^{V})$ has a solution, and a pair $(\tilde{\mathbf{x}},\tilde{\mathbf{y}})$ is a solution of model \eqref{model:convex_model} with  $\mathcal{N}:=N(\mathbf{x}^{V})$ if and only if $(\tilde{\mathbf{x}},\tilde{\mathbf{y}})$ satisfies equations \eqref{eq:fixed-eqs_convex_model_solution}. According to Proposition \ref{pro:G_minimizer_exist}, if $(\tilde{\mathbf{x}},\tilde{\mathbf{y}})$ is a solution of model \eqref{model:convex_model} with  $\mathcal{N}:=N(\mathbf{x}^{V})$, then  $\tilde{\mathbf{x}}$ is a solution of the equation
\begin{eqnarray}
\mathbf{x}=\mathrm{P}_{\mathcal{B}_{\mathcal{N}}}
\left[\left(\mathbf{I}+\frac{\beta}{\gamma}\mathbf{K}^*\mathbf{K}\right)^{-1}
\left(\mathbf{x}+\frac{\beta}{\gamma}\mathbf{K}^*\mathbf{r}\right)\right]
\label{eq:PN_IK8K_H_fixedeq}
\end{eqnarray}
and in this case, the above equation has a solution.

We next verify that the subsequence $\left\{(\mathbf{x}^k,\mathbf{y}^k)\right\}_{k\geq V}$
satisfies the equations
\begin{eqnarray}
\left\{
\begin{array}{l}
\mathbf{x}^{k+1}=\mathrm{P}_{\mathcal{B}_{\mathcal{N}}}
\left(\mathbf{y}^{k}\right),
\\
\mathbf{y}^{k+1}=\mathbf{x}^{k+1}
-\frac{\beta}{\gamma}\mathbf{K}^*\left(\mathbf{K}\mathbf{y}^{k+1}-\mathbf{r}\right).\ \ \
\end{array}
\right.\label{eq:fixed-eqs_two_var_iter_B}
\end{eqnarray}
Since $\left\{(\mathbf{x}^k,\mathbf{y}^k)\right\}$
is a sequence generated by Algorithm \eqref{eq:fixed-eqs_two_var_iter}, there holds $\mathbf{x}^{k+1}\in\mathrm{prox}_{\beta\|\cdot\|_0}
\left(\mathbf{y}^{k}\right)$. Applying Item (iii) of Lemma \ref{P_BN_prox_L0} yields
$\mathbf{x}^{k+1}=\mathrm{P}_{\mathcal{B}_{N(\mathbf{x}^{k+1})}}
\left(\mathbf{y}^{k}\right)$. Moreover,  by hypothesis we have $\mathcal{N}=N(\mathbf{x}^{k+1})$ for any $k\geq V$. Therefore, we derive that $\mathbf{x}^{k+1}=\mathrm{P}_{\mathcal{B}_{\mathcal{N}}}\left(\mathbf{y}^{k}\right)$ for all $k\geq V$.  This confirms that \eqref{eq:fixed-eqs_two_var_iter_B} holds for all $k\geq V$.

We now show that the subsequence $\left\{\mathbf{x}^k\right\}_{k\geq V}$ is convergent. In \eqref{eq:fixed-eqs_two_var_iter_B},  substituting the second equation of \eqref{eq:fixed-eqs_two_var_iter_B} into the first one to eliminate $\mathbf{y}^{k}$, we obtain that
$$
\mathbf{x}^{k+1}=\mathrm{P}_{\mathcal{B}_{\mathcal{N}}}
\left[\left(\mathbf{I}+\frac{\beta}{\gamma}\mathbf{K}^*\mathbf{K}\right)^{-1}
\left(\mathbf{x}^k+\frac{\beta}{\gamma}\mathbf{K}^*\mathbf{r}\right)\right].
$$
By the definition of operator $Q_{\mathcal{N}}$ with  $\mathcal{N}:=N(\mathbf{x}^{V})$, we see that the subsequence $\left\{\mathbf{x}^k\right\}_{k\geq V}$ is a Picard sequence of  $Q_{\mathcal{N}}$.
Lemma \ref{composition-operator} confirms that $Q_{\mathcal{N}}$ is  nonexpansive $\frac{1}{4}\alpha$-averaged.
Moreover,  the existence of a solution  of equation \eqref{eq:PN_IK8K_H_fixedeq} ensures that $Fix(Q_{\mathcal{N}})\neq \emptyset$.
Theorem \ref{theorem:averaged_nonexp_con} implies that the subsequence $\left\{\mathbf{x}^k\right\}_{k\geq V}$ converges to a solution $\mathbf{x}^\star$ of equation \eqref{eq:PN_IK8K_H_fixedeq}.
Note that $\left\{\mathbf{x}^k\right\}_{k\geq V}$ is a sequence in  $\mathcal{B}_{\mathcal{N}}$. Since $\mathcal{B}_{\mathcal{N}}$ is closed, we get that $\mathbf{x}^\star\in\mathcal{B}_{\mathcal{N}}$. In addition, applying Lemma \ref{lemma:support_xstar} yields  $\mathcal{N}=N(\mathbf{x}^{V})=N(\mathbf{x}^{\star})$.

It remains to show that the subsequence $\left\{\mathbf{y}^k\right\}_{k\geq V}$ is convergent. It suffices to prove that $\left\{\mathbf{y}^k\right\}_{k\geq V}$
is a Cauchy sequence in $\mathbb{R}^{N}$.  By \eqref{eq:fixed-eqs_two_var_iter_B} and Lemma \ref{lemma:y_existence}, we have that
\begin{eqnarray}\label{eq_yk_xk}
\mathbf{y}^{k+1}=\left(\mathbf{I}+\frac{\beta}{\gamma}\mathbf{K}^*\mathbf{K}
\right)^{-1}
\left(\mathbf{x}^{k+1}+\frac{\beta}{\gamma}\mathbf{K}^*\mathbf{r}\right).
\end{eqnarray}
Thus, for all $m>n>V$, by the above equation we have that
$$
\left\|\mathbf{y}^{m}-\mathbf{y}^{n}\right\|_2
=\left\|\left(\mathbf{I}+\frac{\beta}{\gamma}\mathbf{K}^*\mathbf{K}\right)^{-1}
\left(\mathbf{x}^{m}-\mathbf{x}^{n}\right)\right\|_2
\leq \left\|\left(\mathbf{I}+\frac{\beta}{\gamma}\mathbf{K}^*\mathbf{K}\right)^{-1}
\right\|_2
\left\|\mathbf{x}^{m}-\mathbf{x}^{n}\right\|_2,
$$
which together with inequality \eqref{eig_IKK} yields
$
\left\|\mathbf{y}^{m}-\mathbf{y}^{n}\right\|_2
\leq\left\|\mathbf{x}^{m}-\mathbf{x}^{n}\right\|_2.
$
Since $\left\{\mathbf{x}^k\right\}_{k\geq V}$ is convergent, it is a Cauchy sequence in $\mathcal{B}_{\mathcal{N}}$. Thus, $\left\{\mathbf{y}^k\right\}_{k\geq V}$
is a Cauchy sequence in $\mathbb{R}^{N}$ and is convergent.
Denote $\mathbf{y}^\star$ as the limit point of $\left\{\mathbf{y}^k\right\}_{k\geq V}$.

Finally, in equation \eqref{eq_yk_xk}, we let $k\rightarrow +\infty$ and obtain that
$
\mathbf{y}^{\star}=\left(\mathbf{I}+\frac{\beta}{\gamma}\mathbf{K}^*\mathbf{K}
\right)^{-1}\left(\mathbf{x}^{\star}+\frac{\beta}{\gamma}\mathbf{K}^*\mathbf{r}\right),
$
where $\mathbf{x}^\star$ is a solution of equation \eqref{eq:PN_IK8K_H_fixedeq}.
Consequently, the pair $\left(\mathbf{x}^\star,\mathbf{y}^\star\right)\in\mathcal{B}_{\mathcal{N}}\times\mathbb{R}^{N}$ satisfies equations \eqref{eq:fixed-eqs_two_var_iter_Bb}. According to Proposition \ref{pro:F_local_minimizer}, $\left(\mathbf{x}^\star,\mathbf{y}^\star\right)$ is a solution of model \eqref{model:convex_model} with $\mathcal{N}:=N(\mathbf{x}^{\star})$.
\end{proof}


We are now ready to prove the main theorem on convergence of Algorithm \eqref{eq:fixed-eqs_two_var_iter}.

\begin{theorem}\label{theorem:conv_analysis_main1}
 Let $\left\{(\mathbf{x}^k,\mathbf{y}^k)\right\}$
be a sequence generated by Algorithm \eqref{eq:fixed-eqs_two_var_iter} with an initial
$(\mathbf{x}^0,\mathbf{y}^0)\in\mathbb{R}^{N}\times\mathbb{R}^{N}$. If $\beta,\gamma>0$ are chosen such that  $0<\frac{\beta}{\gamma}<\frac{\sqrt{5}-1}{2}$,
then $\left\{(\mathbf{x}^k,\mathbf{y}^k)\right\}$  converges to a local minimizer
 $(\mathbf{x}^{\star},\mathbf{y}^{\star})$ of  model \eqref{model:two_variable}.
Moreover, the sequence $\left\{F(\mathbf{x}^k,\mathbf{y}^k)\right\}$ converges to $F(\mathbf{x}^{\star},\mathbf{y}^{\star})$.  Furthermore, if $|y_j^{\star}|\neq\sqrt{2\beta}$ for all $j\in N(\mathbf{y}^{\star})$, then $\mathbf{y}^{\star}$ is a local minimizer of model \eqref{model:general}.
\end{theorem}
\begin{proof}
 By Lemma \ref{theorem:conv_analysis_app},  the sequence $\left\{(\mathbf{x}^k,\mathbf{y}^k)\right\}$  converges to  a pair $(\mathbf{x}^{\star},\mathbf{y}^{\star})$, which is a solution of the convex model \eqref{model:convex_model} with $\mathcal{N}:=N(\mathbf{x}^{\star})$.
  It follows from Theorem \ref{lemma:relation_nonconv_convex} that $(\mathbf{x}^{\star},\mathbf{y}^{\star})$ is a local minimizer of the non-convex model \eqref{model:two_variable}.

We next prove that $\left\{F(\mathbf{x}^k,\mathbf{y}^k)\right\}$ converges to $F(\mathbf{x}^{\star},\mathbf{y}^{\star})$. By the hypothesis and
Theorem \ref{theorem:conv_analysis} we have that the sequence  $\left\{F(\mathbf{x}^k,\mathbf{y}^k)\right\}$ is convergent. As the function $G$ defined by \eqref{def:G} is continuous, we have that
$
\lim\limits_{k\rightarrow \infty}G(\mathbf{x}^k,\mathbf{y}^k)=G(\mathbf{x}^{\star},\mathbf{y}^{\star}).
$
Furthermore, by Lemma
\ref{lemma:x_sparsity_two}, there exists a positive integer  $V$ such that
$\ N(\mathbf{x}^k)=N(\mathbf{x}^{V})$  for all $k\geq V$. Hence,
$\ N(\mathbf{x}^k)=N(\mathbf{x}^{\star})$ and $\left\|\mathbf{x}^k\right\|_0=\left\|\mathbf{x}^{\star}\right\|_0$  for all $k\geq V$. Noting that  $F(\mathbf{x}^k,\mathbf{y}^k)=G(\mathbf{x}^k,\mathbf{y}^k)
+\gamma\left\|\mathbf{x}^k\right\|_0$ and letting $k\rightarrow \infty$ yield
$
\lim\limits_{k\rightarrow \infty}F(\mathbf{x}^k,\mathbf{y}^k)=F(\mathbf{x}^{\star},\mathbf{y}^{\star}).
$

Notice that $\mathbf{x}^{\star}\in\mathrm{prox}_{\beta\|\cdot\|_0}
\left(\mathbf{y}^{\star}\right)$. Since $(\mathbf{x}^{\star},\mathbf{y}^{\star})$  is a local minimizer of  model \eqref{model:two_variable} and $|y_j^{\star}|\neq\sqrt{2\beta}$ for all $j\in N(\mathbf{y}^{\star})$,
 applying Proposition  \ref{pro:model_relation_localm} yields that $\mathbf{y}^{\star}$ is a local minimizer of model \eqref{model:general}.
\end{proof}

\section{Applications in Seismic  Wavefield Modeling in the Frequency Domain}\label{sec:MAM4}\setcounter{equation}{0}

In this section, we consider applications of the incomplete Fourier transform method developed in section \ref{sec:MAM_alro} to seismic  wavefield modeling in the frequency domain.

Frequency domain modeling for the generation of synthetic seismograms and crosshole tomography has been an active field of
research since 1970s \cite{LD,M1}. Modeling of seismic wavefield in the frequency domain requires solving a sequence of boundary value problems of the Helmholtz equation with different wave numbers (frequencies) \cite{J1,Lin_Lebed_Erlangga}.  When solutions for all frequencies satisfying the Nyquist-Shannon criterion are available, we can obtain the corresponding time domain results by the inverse discrete Fourier transform (IDFT)  \cite{Brigham,H3,Riyanti-Kononov-Erlanggga-Vuik}. However, it is a challenging task to obtain solutions for the boundary value problems corresponding to high frequencies \cite{B2}. To overcome this difficulty, an incomplete Fourier transform model  \cite{WSX} was proposed for frequency domain modeling, by using an $\ell_1$-norm regularization method. According to \cite{Fan-Li2001}, the $\ell_1$-norm is not an ideal sparsity promotion function since it would cause biases. The regularization with the envelop of the $\ell_0$ norm developed in the previous sections can avoid biases and allow us to recover data from incomplete Fourier transforms with only lower frequencies. In this way, we do not have to solve boundary value problems of the Helmholtz equation with large wave numbers.

We now recall the seismic wavefield modeling in the frequency domain. In the time domain, the 2D acoustic wave equation has the form
\begin{equation}
\label{wave_equation_time_domain}
 \frac{1}{v^2}\frac{\partial^2u}{\partial t^2} -\Delta
u=g,\ \ \mbox{on}\ \ \mathbb{R}^2,
\end{equation}
where $u$, $v$ and $g$ denote, respectively, the unknown pressure of
the wave field, the background velocity and the source term in the
medium. Both $u$ and $g$ are functions in the spatial-time space $\mathbb{R}^2 \times \mathbb{R}_+$, while $v$ is a function in the spatial space $\mathbb{R}^2$. By using the Fourier transform, one may convert the wave equation as a family of the Helmholtz equations.
Upon solving these Helmholtz equations and converting back to the solution of the wave equation, one can model propagation of seismic wavefields. This is the frequency approach for modeling the seismic wave propagation.

Next, we present  the 2D acoustic wave equation in the frequency domain. To this end, we call the definition of the continuous Fourier transform.
For a function $\psi$ defined on $\mathbb{R}$, its  continuous Fourier transform at frequency $f \in \mathbb{R}$ is given as
\begin{equation}\label{continuous_FT}
\widehat{\psi}(f) := \int_{\mathbb{R}}\psi(t)e^{-i2\pi ft} dt.
\end{equation}
With the Fourier transform, the acoustic wave equation \eqref{wave_equation_time_domain} is converted to the well-known Helmholtz equation
\begin{equation}
\label{Helmholtz_Equation}
-\Delta \widehat{u}-\kappa^2\widehat{u}=\widehat{g},\ \
\end{equation}
where $\kappa$ is the wave number defined by $\kappa  :=\frac{2\pi f}{v}$, with $f$ being the frequency in ${\rm Hz}$. For any $(x,z) \in \mathbb{R}^2$,  $\widehat{u}(x,z, f)$ and $\widehat{g}(x,z, f)$ represent, respectively, the continuous Fourier transforms at the frequency $f$ of the functions  $u(x,z, \cdot)$ and $g(x,z, \cdot)$ which appear in \eqref{wave_equation_time_domain}. The
solution  $u(x,z,t)$ of the acoustic wave equation
\eqref{wave_equation_time_domain} may be obtained by the inverse
Fourier transform from the solutions $\widehat{u}(x,z,f)$ of the Helmholtz equation, for all
$f\in\mathbb{R}$. Therefore, the
fundamental problem for the acoustic wave modeling in the
frequency domain is to solve the family of the Helmholtz equations
\eqref{Helmholtz_Equation} for all $f \in \mathbb{R}$.

We now discuss the generation of synthetic seismograms in frequency domain modeling. For convenience of expression, we describe only the generation of the synthetic seismogram of a fixed point $(x_r,z_r)$. We assume that there exists $T>0$ such that the solution $u$ of the wave equation (\ref{wave_equation_time_domain})  satisfies the condition  $u(x_r,z_r,t)= 0$ for all $t\notin[0,T]$. Mathematically, the synthetic
seismogram of the point $(x_r,z_r)$ is the function  $u(x_r,z_r,t)$, where $t\in[0,T]$. We will use the notation
$u_r(t)\ :=u(x_r,z_r,t)$ and $\widehat u_r(f):=\widehat{u}(x_r,z_r,f)$.  In the context of seismic wavefield modeling,  we  say that  a receiver is located at the point
$(x_r,z_r)$ \cite{J1}, and the synthetic seismogram  $u(x_r,z_r,t)$ is the wave which the receiver receives.
Our goal is to obtain the values of $u_r(t)$ at $M$ equally spaced points taken in the interval $[0, T]$, where $M$ is a  positive integer. By the definition of the continuous Fourier transform \eqref{continuous_FT} of $u_r(t)$, with the rectangle quadrature method, we have that
\begin{eqnarray}
\widehat{u}_r(f)=\lambda\sum_{n=0}^{M-1}u_r\left(\lambda n\right)e^{-i2\pi f\lambda n},\ \  \mbox{for all}\ \ f\in
\mathbb{R},\label{problem_Ftransform_nume_inte2}
\end{eqnarray}
where $\lambda\ :=\frac{T}{M}$.
In the frequency domain seismic wavefield modeling, $\widehat{u}_r(f)$  is usually obtained through solving the Helmholtz equation \eqref{Helmholtz_Equation} by finite difference \cite{CCFW,J1} or finite element methods \cite{B2}.
In addition, as the source function in the seismic case approximately has a limited spectrum (see, \cite{WSX}),  we denote by  $f_{nmax}$ an approximate highest frequency of the source, and assume that equation \eqref{problem_Ftransform_nume_inte2} holds approximately for $f\in[0,f_{nmax}]$.
Following the Nyquist
sampling theorem, we require the frequency step size $\Delta f$  to satisfy the condition
$\Delta f \leq\frac{1}{T}.$
Hence, we choose $f_m=\frac{m}{T}$, for $m=0,1,\ldots,M-1$, and the
corresponding frequency step size $\Delta f =\frac{1}{T}$. Also, we let
$$
\mathbf{u}:=\left[u_r(0), u_r\left(\frac{T}{M}\right), \ldots \ u_r\left(\frac{(M-1)T}{M}\right) \right]^\top
$$
and
$$
\widehat{\mathbf{u}}:=\left[\widehat{u}_r(0), \widehat{u}_r\left(\frac{1}{T}\right), \ldots, \widehat{u}_r\left(\frac{M-1}{T}\right)\right]^\top.\nonumber
$$
Hence, from equation~\eqref{problem_Ftransform_nume_inte2}, we have that
\begin{eqnarray}
\widehat{\mathbf{u}}=\lambda \sqrt{M} \mathbf{F} \mathbf{u},
\label{eq:u_uhat_relation}
\end{eqnarray}
where  $\mathbf{F}$  is the $M \times M$ discrete Fourier transform
matrix defined as before. We can reconstruct
$\mathbf{u}$ from $\widehat{\mathbf{u}}$ with IDFT, that is,
$$
\mathbf{u}=\frac{1}{\lambda}\frac{1}{\sqrt{M}}\mathbf{F}^*  \widehat{\mathbf{u}}.
$$



For a large scale problem, solving the Helmholtz equation with large wave numbers is a difficult task (see, \cite{B2,H3}).  We even have difficulty in obtaining $\widehat{u}_r(f)$ for frequencies $f$ satisfying $f\leq f_{nmax}$. In fact, only some components of $\hat{\mathbf{u}}$ are available. We use $\mathbf{r}_{obs}$ to denote the vector formed by the $\widehat{\mathbf{u}}$ by removing its components that are not available. Thus, there exists a row selector matrix $\mathbf{R}$ such that
\begin{equation}\label{partialFT}
 \mathbf{R}\widehat{\mathbf{u}}=\mathbf{r}_{obs}.
\end{equation}
Let $\mathbf{r} :=\frac{1}{\sqrt{M}}\mathbf{r}_{obs}$. Substituting equation  \eqref{eq:u_uhat_relation} into equation \eqref{partialFT} and noticing
the definition of  $\mathbf{r}$  yield
\begin{equation}
\label{seismogram-model-origin-app-new}
\lambda \mathbf{RF}\mathbf{u}=
\mathbf{r},
\end{equation}
which is indeed an incomplete Fourier transform system \eqref{seismogram-model-origin-app} with $\mathbf{v}:=\lambda\mathbf{u}$.
We also note that different row selector matrices correspond to different ways of sampling.
To avoid solving the Helmholtz equation with
large wave numbers, we follow \cite{WSX} to sample only lower frequencies in recovering the seismic wavefield. We then adopt the sparse regularization model \eqref{model:two_variable} developed in Section \ref{sec:MAM2} to find an approximation of $\mathbf{u}$ of equation \eqref{seismogram-model-origin-app-new} by employing Algorithm \eqref{eq:fixed-eqs_two_var_iter} proposed in Section \ref{sec:MAM_alro}.




\section{Numerical Experiments}\label{sec:MAM5}\setcounter{equation}{0}
In this section, we present four numerical experiments to demonstrate the effectiveness of the proposed sparse regularization model \eqref{model:two_variable}. All the experiments are performed on an Intel Xeon (4-core) with 3.60 GHz, 16 Gb RAM and Matlab 7v.

We begin with setting up equation (\ref{Inverse_fourier_problem-new}) and model \eqref{model:two_variable}. The selector matrix $\mathbf{R}$ depends on the sampling method to be specified later.  The $N\times M$ matrix $\mathbf{W}$, with $N=lM$ and $l$ being a positive integer,  is constructed from the piecewise linear spline tight framelets system described in \cite{Cai-Chan-Shen-Shen}.
By $\Delta f$ we denote the step size of the frequency $f$, and $f_{min}$ and $f_{max}$ represent, respectively,  the lowest frequency and the highest frequency that we compute.
%
We will use Algorithm
\eqref{eq:fixed-eqs_two_var_iter} to solve the model \eqref{model:two_variable} (EL0M). When implementing Algorithm \eqref{eq:fixed-eqs_two_var_iter},  $\mathrm{prox}_{\beta\|\cdot\|_0} (\mathbf{z})$, for all $\mathbf{z}\in\mathbb{R}^{N}$, is computed by equation \eqref{prox_envl0-1} with the following formula
\begin{equation*}
\mathrm{prox}_{\beta|\cdot|_0} (z_i) = \left\{
                             \begin{array}{ll}
                              \{z_i\}, & \hbox{$|z_i|> \sqrt{2\beta}$;} \\
                              \{ 0\}, & \hbox{otherwise.}
                             \end{array}
                           \right.
\end{equation*}
We point out that  the variable $\mathbf{y}$ in model \eqref{model:two_variable} is used to generate the numerical results of EL0M in all the following tables and figures. This is because our goal is to solve model \eqref{model:general} and Theorem \ref{theorem:conv_analysis_main1} verifies that a sequence $\{(\mathbf{x}^{k},\mathbf{y}^{k})\}$ generated by Algorithm \eqref{eq:fixed-eqs_two_var_iter} converges to a pair $(\mathbf{x}^{\star},\mathbf{y}^{\star})$,  with $\mathbf{y}^{\star}$ being a local minimizer of model \eqref{model:general}.

We shall compare the effectiveness of model  EL0M with that of the  $\ell_1$-norm model (L1M) which has the form
\begin{eqnarray*}
{\rm argmin}\left\{\frac{1}{2}\|\mathbf{K}\mathbf{y}-\mathbf{r}\|_2^2
+\gamma\|\mathbf{y}\|_1\right\},
\end{eqnarray*}
proposed in \cite{WSX} for inversion of incomplete Fourier transforms.
Model L1M will be solved by Algorithm \ref{algo-uniform-better} \cite{WSX}.
\begin{algorithm}
\caption{ }
 \begin{algorithmic}[htb]
    \State Input: the matrix $\mathbf{K}$, the vector $\mathbf{r}$, and the diagonal matrix $\mathbf{\Gamma}=\gamma \mathbf{I}$
   \State Initialization: $\mathbf{y}^0=\mathbf{v}^1=0$, $t_1=1$.
   \Repeat { ($k\ge 0$)}
\begin{eqnarray*}
\begin{array}{l}
\mathbf{y}^{k}={\rm prox}_{ \|\cdot\|_{1} \circ \mathbf{\Gamma}}\left(\mathbf{v}^k-\mathbf{K}^*(\mathbf{K}\mathbf{v}^k-\mathbf{r})\right)\\
t_{k+1} = \frac{1+\sqrt{1+4t^2_k}}{2}\\
\mathbf{v}^{k+1} = \mathbf{y}^{k}+\left(\frac{t_k-1}{t_{k+1}}\right) (\mathbf{y}^{k}-\mathbf{y}^{k-1})
\end{array}
\label{alg_dual2_procedure}
\end{eqnarray*}
   \Until{$\|\mathbf{y}^k-\mathbf{y}^{k-1}\|_2/\|\mathbf{y}^{k-1}\|_2>tol$}

   \State Return: $\mathbf{u}=\frac{1}{\lambda}\mathbf{W}^\top \mathbf{y}^\infty$
 \end{algorithmic}\label{algo-uniform-better}
 \end{algorithm}
The quality of the restored signals will be measured by the signal-to-noise ratio (SNR) defined by
$$
{\rm SNR}\ :=10
\log_{10}\left(\frac{\|data_{orig}\|_2^2}{\|data_{orig}-data_{reco}\|_2^2}\right),
$$
where $data_{orig}$ and $data_{reco}$  represent the
original data and the recovered data, respectively.

\subsection{Recovering time-domain data with  exact frequency data}\label{example_gaussian}

In this subsection, we consider recovering the first
derivative of Gaussian function with its insufficient {\it exact} frequency data by model EL0M, with comparison to model L1M, for random sampling and uniform sampling.

The first
derivative of the Gaussian function has the form
\begin{eqnarray}
G(t,t_0,\tilde{\alpha})=-2\tilde{\alpha}(t-t_0)\exp(-\tilde{\alpha}(t-t_0)^2),
\label{Gaussian_time}
\end{eqnarray}
and its Fourier transform is given by
\begin{eqnarray}
\widehat{G}(f,t_0,\tilde{\alpha})=2\sqrt{\frac{\pi}{\tilde{\alpha}}}\pi
f\exp\left(-\frac{\pi^2f^2}{\tilde{\alpha}}\right)\left[\sin(2\pi
ft_0)+i\cos(2\pi ft_0)\right]. \label{Gaussian_frequency}
\end{eqnarray}
In this experiment, we set $t_0:=1$ and $\tilde{\alpha}:=200$, and choose $T:=2s$ and $M:=129$, where $s$ denotes second.
The natural maximum frequency of $\widehat{G}(f,1,200)$ is approximately equal to $15$Hz,  that is, $f_{nmax}=15$Hz.
In the remaining part of this paper, we will always use Hz as the frequency unit without mentioning it.
Furthermore, we set the tolerance as $tol=10^{-6}$ for iterations in implementing the two algorithms, and obtain the SNR-values for reconstructed data.

We first test the restoration ability of
EL0M for {\it exact} data with random sampling. We randomly select $50\%, 40\%, 30\%, 20\%$ frequencies from the set $\left\{0.5,1,1.5, \ldots, 15\right\}$. We report numerical results in Table \ref{table:gaussian_real_envl0L1_randsamp}, where each SNR-value reported is the average of five runs.

\begin{table}
\caption{The SNR results of EL0M with comparison to L1M for recovering time-domain data from incomplete {\it exact} frequency data with random sampling. } \vskip 4mm
\centering
    \begin{tabular}{c|c|c|c|c}
    \hline
    \hline
              &  {\rm 50\%}         &  {\rm 40\%}          & {\rm 30\%}        &  {\rm 20\%}      \\
     \hline
     L1M   & 21.0756    & 21.2153   & 10.8219   & 3.7084       \\
     EL0M  & {\bf 37.0167}   & {\bf 30.9599}   & {\bf 16.0520} & {\bf 4.8128 } \\
     \hline
   \end{tabular}
\label{table:gaussian_real_envl0L1_randsamp}
\end{table}

We then test the restoration ability of EL0M with comparison to L1M for {\it exact} frequency data of only low Fourier frequencies. In this test, we choose $f_{max}$ as  $7.5$, $6$, $4.5$ and $3$, and sample {\it exact} data from intervals $\left[0.5, f_{max}\right]$ with the uniform step size $\Delta f:=0.5$.
The selector matrices $\mathbf{R}$ for this example are chosen according to the values of $f_{max}$.
Note that $f_{max}$ in each of these cases is much smaller than $f_{nmax}:=15$ required by the Nyquist sampling theorem.
We report numerical results in Table \ref{table:gaussian_real_envl0_sample1}.
In Figure \ref{figure:gaussian_real_envl0}, the
figures obtained by EL0M  are presented. Here, ``Original Signal" represents the real signal in time and ``EL0M" represents the signal restored by EL0M.

From Tables \ref{table:gaussian_real_envl0L1_randsamp} and  \ref{table:gaussian_real_envl0_sample1}, we find that EL0M outperforms L1M significantly in both of the tests. Moreover, we observe that the results from the uniform sampling are better than those from random sampling. These numerical results and Figure \ref{figure:gaussian_real_envl0} confirm that EL0M can well recover the first derivative of the Gaussian function with exact lower frequency data .


\begin{table}
\caption{The SNR results of EL0M with comparison to L1M for recovering time-domain data from {\it exact} frequency data with uniform sampling from intervals $\left[0.5,f_{max}\right]$ with $\Delta f=0.5$. } \vskip 4mm
\centering
    \begin{tabular}{c|c|c|c|c}
    \hline
    \hline
       \backslashbox{Model}    {$f_{max}$} &  7.5{\rm }         &  6{\rm }          & 4.5{\rm }        &  3{\rm }      \\
     \hline
     L1M   & 24.5150     & 16.0542   & 13.6275    & 13.4741       \\
     EL0M  & {\bf 39.0745}   & {\bf 33.8784}   & {\bf 34.6615} & {\bf 21.2628 } \\
     \hline
   \end{tabular}
\label{table:gaussian_real_envl0_sample1}
\end{table}

\begin{figure}[h]
\centering
\begin{tabular}{cc}
\includegraphics[width=2.6in]{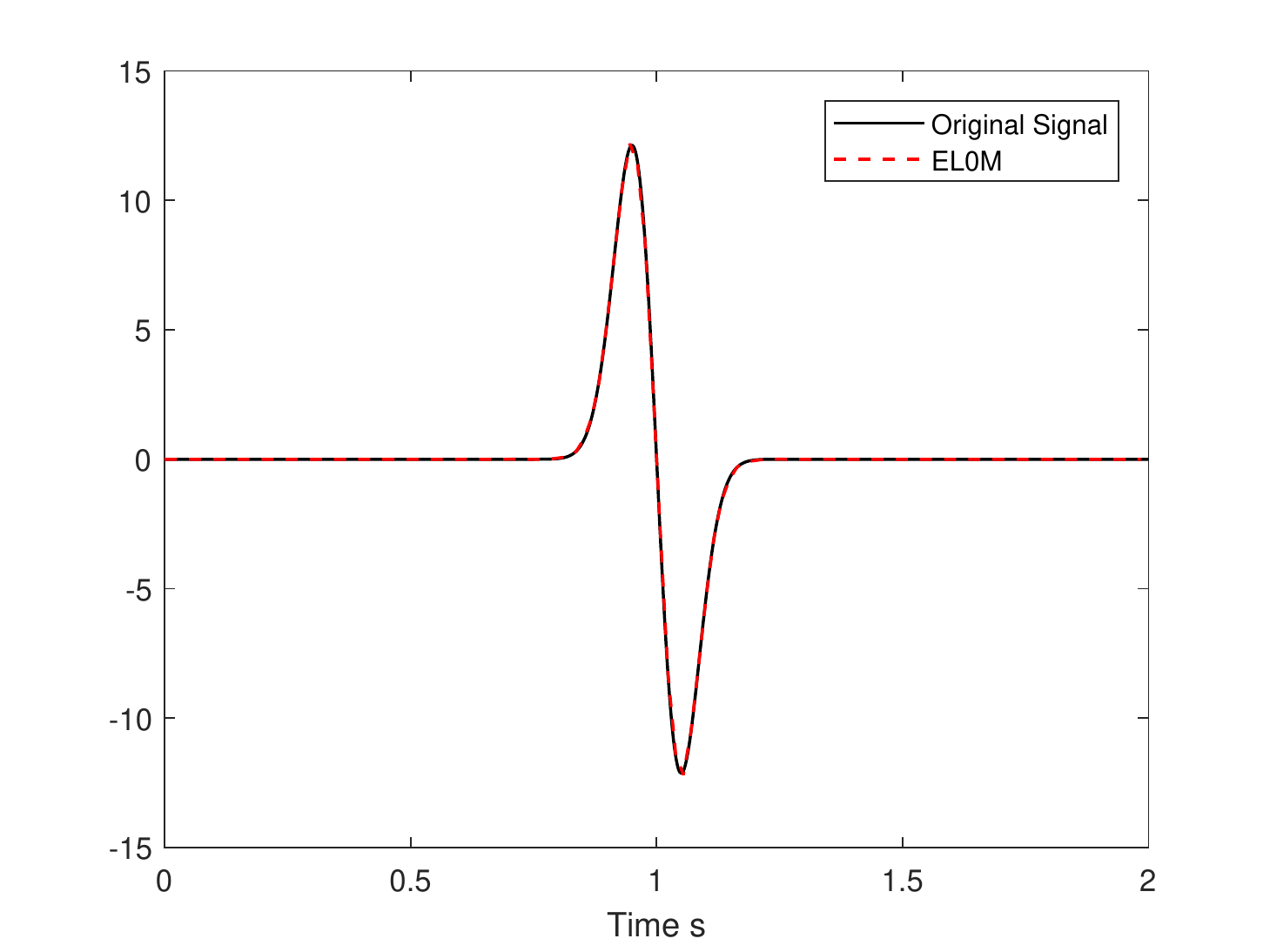}&\includegraphics[width=2.6in]{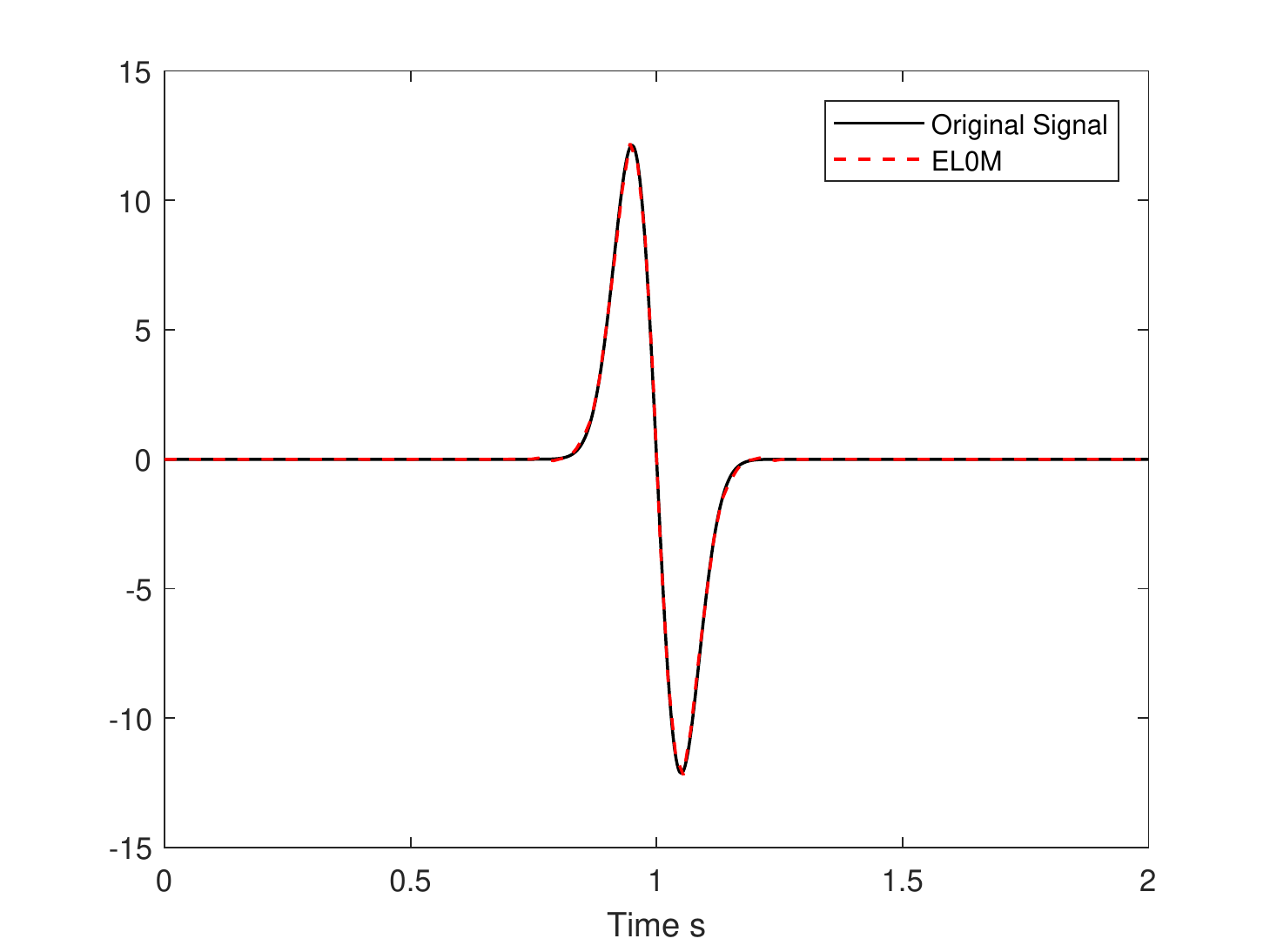}\\
(a)&(b)\\
\includegraphics[width=2.6in]{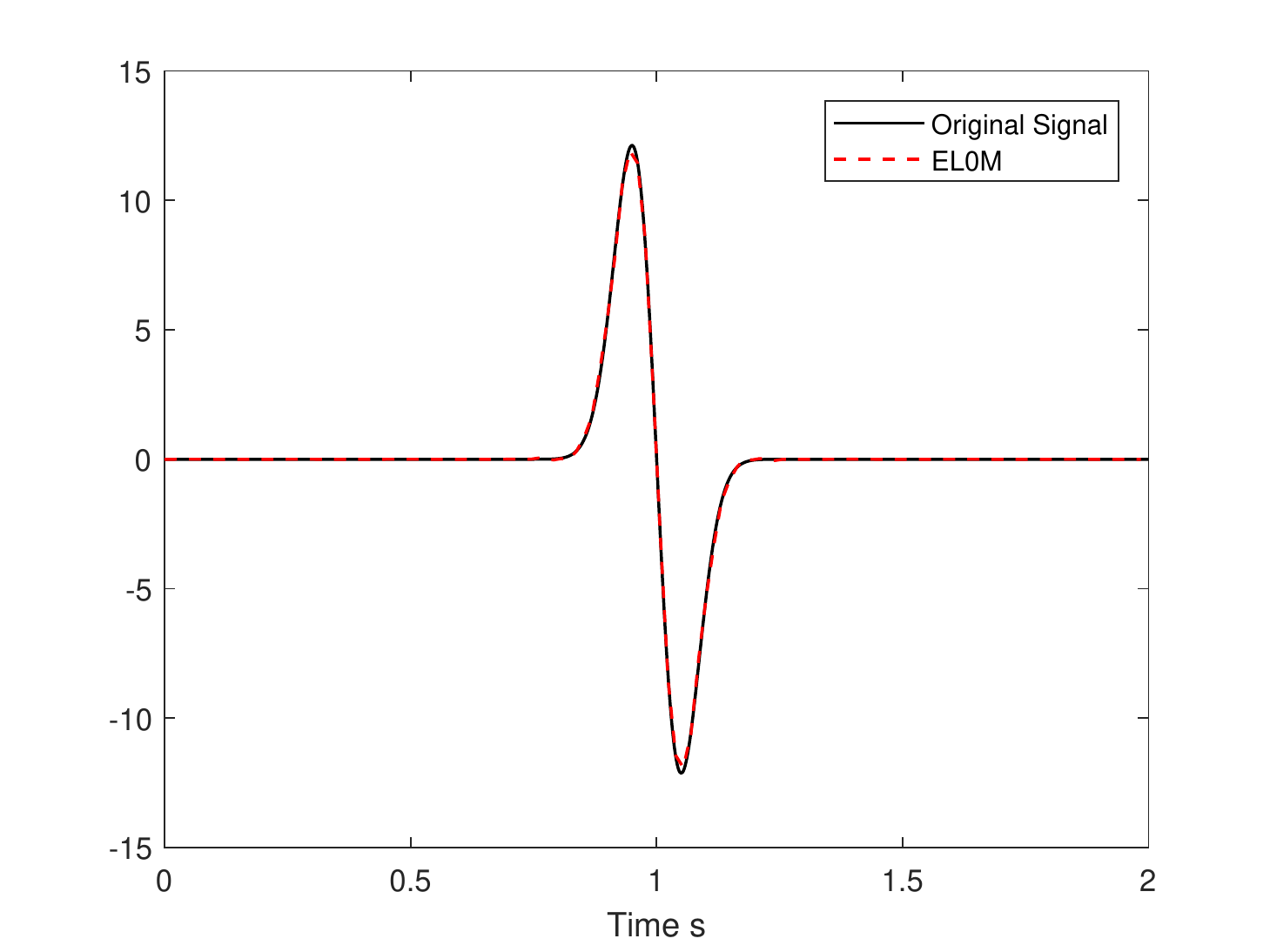}&\includegraphics[width=2.6in]{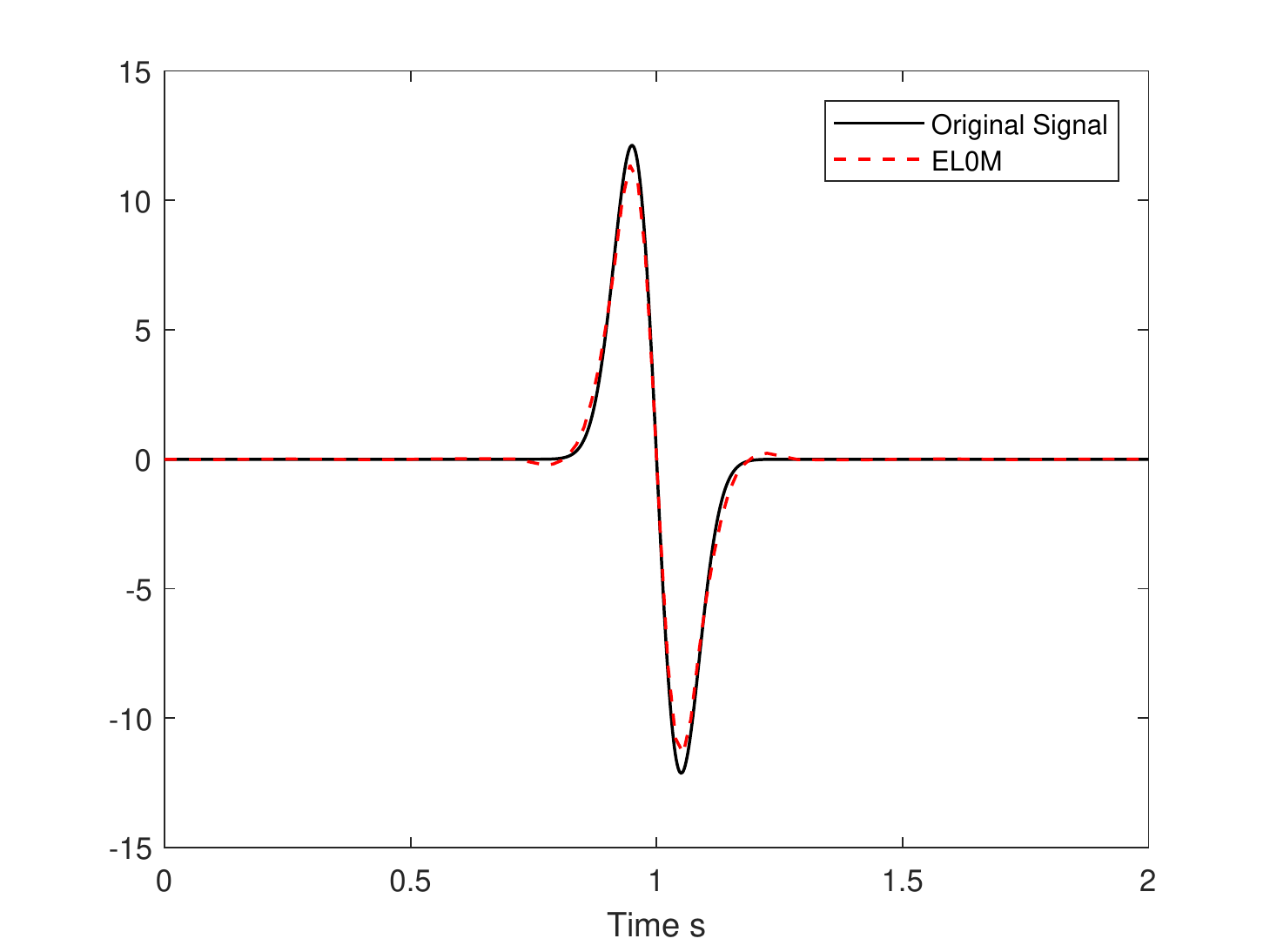}\\
(c)&(d)
\end{tabular}
\caption{ EL0M,\ $\Delta f=0.5 $,\ $l=3$:
(a) $f_{max}=7.5$, $\gamma=0.0202$, $\beta=0.0100$, (b) $f_{max}=6$,
$\gamma=3.1053$, $\beta=1.9000$, (c) $f_{max}=4.5$, $\gamma=1.0460$, $\beta=0.6400$,  (d)
$f_{max}=3$, $\gamma=0.6211$, $\beta=0.3800$. }
\label{figure:gaussian_real_envl0}
\end{figure}



\subsection{Recovering time-domain data from noisy low
frequency data}\label{example_gaussian_noised}

We consider in this subsection recovering the first
derivative of Gaussian function from its noisy low
frequency data by model EL0M, with comparison to model L1M. All conditions imposed in this example are the same as those in the last subsection for uniform sampling, except data are contaminated with Gaussian noise of standard deviations $\sigma=0.1,\ 0.3,\ 0.5$, respectively. The tolerance for iterations is again set as $tol=10^{-6}$. We report numerical results for this example in Table \ref{table:gaussian_real_envl0_noise0135-deltf05}, where each SNR-value is the average of five runs. From Table \ref{table:gaussian_real_envl0_noise0135-deltf05}, we find that model EL0M can restore well signals from their {\it noisy} low frequency  data and model EL0M once again outperforms model L1M significantly.



\begin{table}[h]
\caption{A summary of the SNR results of  L1M and EL0M for recovering time-domain data from {\it noisy} data sampled from intervals $\left[0.5,f_{max}\right]$ with $\Delta f=0.5$ and $\sigma=0.1,~0.3,~0.5$. }
\begin{center}\scriptsize
\begin{tabular}{c|cccc}\hline
  \backslashbox{Model}    {$f_{max}$} & $7.5$ & $6$ & $4.5$ & $3$  \\
\hline \multicolumn{5}{c}{$\sigma=0.1$} \\ \cline{1-5}
L1M &  18.8278  &  13.4524 & 13.3601 & 13.0725  \\
EL0M & {\bf 24.9187} &  {\bf 24.0535} & {\bf 24.1239} & {\bf  14.1184}  \\
\hline \multicolumn{5}{c}{$\sigma=0.3$} \\ \cline{1-5}
L1M & 15.6276 &  12.5826 & 12.1345 &10.7584  \\
EL0M &  {\bf 18.5610} &  {\bf 17.8948}
& {\bf 14.5159} & {\bf 12.5585} \\
 \hline\multicolumn{5}{c}{$\sigma=0.5$} \\ \cline{1-5}
L1M &  11.5818 &  11.4273 &10.3744 & 6.7984  \\
EL0M &  {\bf 16.2983} &  {\bf 13.9437 } & {\bf 13.5271} & {\bf 11.9359 } \\
\hline
\end{tabular}
\end{center}
\label{table:gaussian_real_envl0_noise0135-deltf05}
\end{table}

\subsection{The homogeneous velocity model}\label{subsec:Homo_Model_1}

In this subsection,  we consider the homogeneous velocity model for generating synthetic seismograms with a source function $q$.
This requires solving equation \eqref{wave_equation_time_domain} with constant velocity $v$ illustrated by Figure \ref{figure:The_velocity_model} (a) and the source function  $g(x,z,t):=\delta(x-x_s,z-z_s)q(t)$, where $\delta$ denotes the Dirac delta function and $(x_s, z_s)$ is the coordinates of the source location. We will consider two source functions the Ricker wavelet and the first derivative of the Gaussian function.

\begin{figure}[h]
\centering
\begin{tabular}{cc}
\includegraphics[width=3.1in]{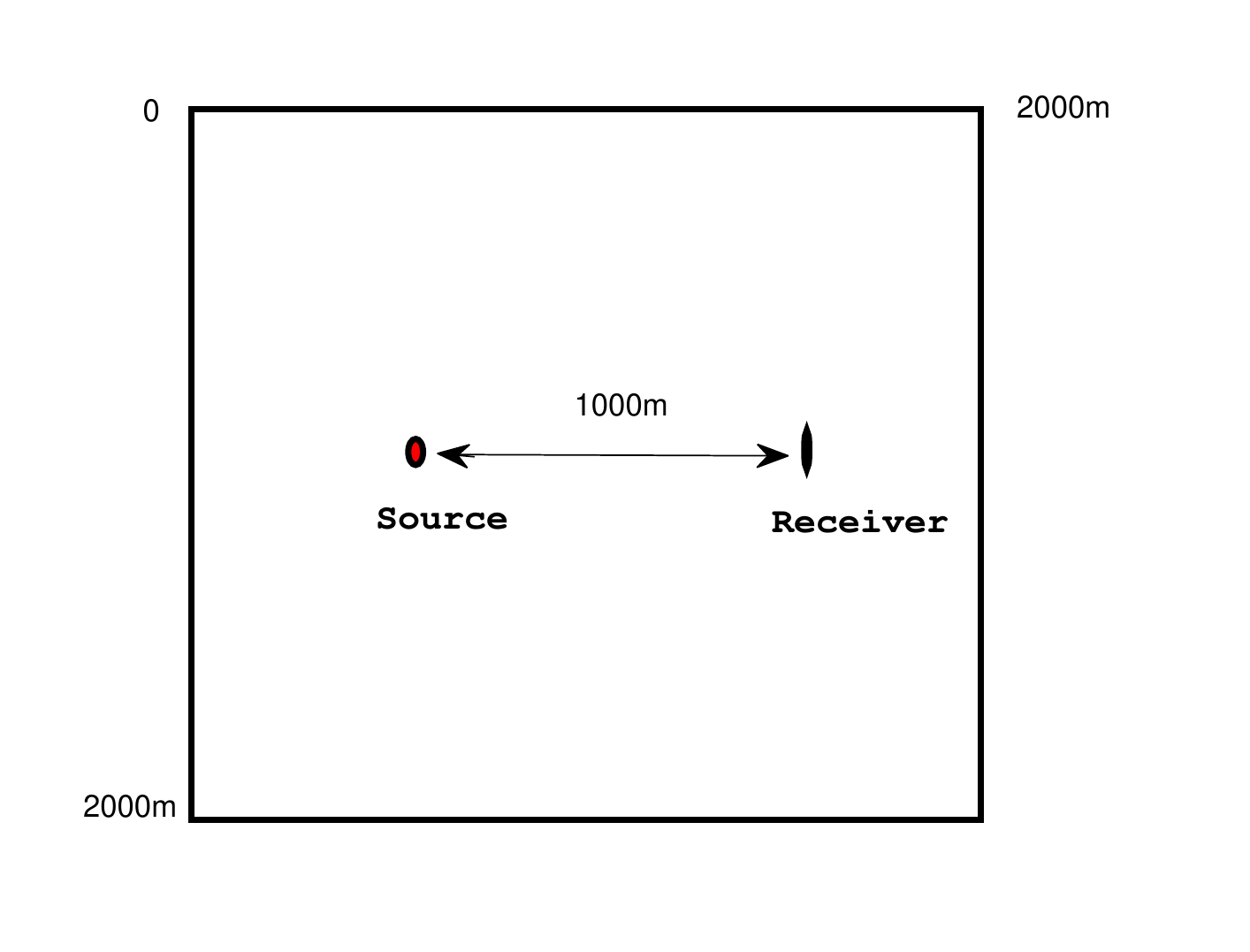}&
\includegraphics[width=3.1in]{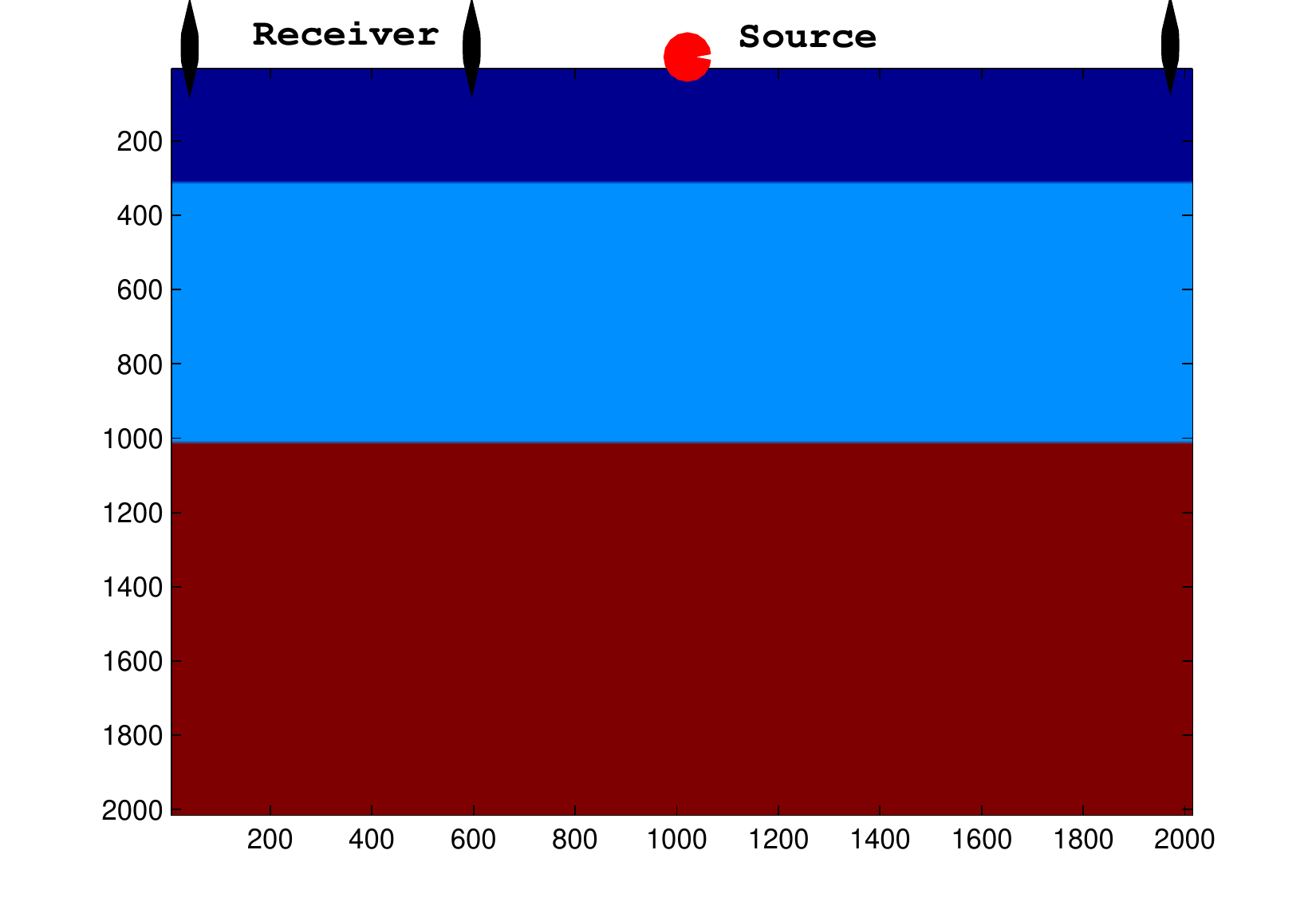}
\\
(a)&(b)\\
\end{tabular}
\caption{ Velocity models: (a) The homogenous model; (b) The layered
model.} \label{figure:The_velocity_model}
\end{figure}


We employ the Frequency domain modeling to generate the synthetic seismogram by solving a sequence of 2D Helmholtz equations \eqref{Helmholtz_Equation} with different frequencies and then inverting the Fourier transform using model EL0M with Algorithm  \eqref{eq:fixed-eqs_two_var_iter}. Our interested domain is $[0,2000]\times[0,2000]$, having meter $m$ as the length unit. The velocity  $v$ of the medium is $1500m/s$, where $s$ denotes the time unit second. The receiver described in Section \ref{sec:MAM4} is located at the point
$(x_r,z_r):=(1500, 1000)$. The source location is $(x_s,z_s)=(500,1000)$, and the  natural maximum frequency of the source function is denoted by $f_{nmax}$.
To obtain the synthetic seismogram at the receiver point $(x_r,z_r)$, we  choose proper parameters $T$, $M$,  and let $\Delta f:=\frac{1}{T}$,
where $T,\ M,\  \Delta f$ are defined in Section \ref{sec:MAM4}. By $f_{min}$ and $f_{max}$ we denote respectively the minimum and maximum frequency used in generating the synthetic seismogram. We let $N_H$ be the smallest positive integer such that $\frac{f_{max}}{\Delta f}\leq N_H$. Thus,  $N_H$ is the number of Helmholtz equations \eqref{Helmholtz_Equation} we need to solve for a particular $f_{max}$ chosen.
If $f_{max}$ were chosen as $f_{nmax}$, we need to solve many Helmholtz equations
\eqref{Helmholtz_Equation} and some of these equations have large wave numbers. We will choose $f_{max}$
smaller than $f_{nmax}$ and reconstruct the  synthetic seismogram (the solution of equation
\eqref{wave_equation_time_domain}) by inverting an incomplete Fourier transform (solutions of  Helmholtz
equations \eqref{Helmholtz_Equation} with only small wave numbers).
To this end, we sample frequencies
$f$ from intervals $[f_{min}, f_{max}]$, with  four different  $f_{max}$ values
and solve the resulting Helmholtz equations  \eqref{Helmholtz_Equation} by employing the finite
difference method developed in \cite{CCFW}, with  the same step size $h:=10$ for both variables $x$ and $z$. To invert the corresponding incomplete Fourier transform, we construct the tight framelet matrix $\mathbf{W}$ with a parameter $l$, and then apply model EL0M with Algorithm  \eqref{eq:fixed-eqs_two_var_iter}.

For comparison purposes the exact solution of the wave equation (\ref{wave_equation_time_domain}) with $v$ and $g$ described above can be obtained by the D'Alembert formula:
$$
u(x,z,t)=\frac{1}{4\pi r}q\left(t-\frac{r}{v},f_0\right),
$$
where $r:=\sqrt{\left(x-x_s\right)^2+\left(z-z_s\right)^2}$.
In this experiment, we take the signal
$u(x_r,z_r,t)$, $t\in[0,T]$ obtained by the D'Alembert formula as the original signal for the comparison purpose.

In our first example, we choose $q(t):=R(t,f_0)$, where $R(t,f_0)$  is the Ricker wavelet  defined by
\begin{eqnarray}
R(t,f_0):=(1-2\pi^2f_0^2t^2)\exp(-\pi^2f_0^2t^2), \label{Ricker_wave_ME}
\end{eqnarray}
with $f_0:=25$. Note that the natural maximum frequency of the  Ricker wavelet is approximately equal to $60$. In this example,  we choose $T:=1.3440s$, $M:=168$ and  $l:=3$. If $f_{max}$ were chosen as $f_{nmax}:=60$, we would need to solve $N_H:=81$ number of Helmholtz equations \eqref{Helmholtz_Equation}. This requires significantly large computational efforts to perform the task. We instead sample frequencies $f$ from intervals
$[f_{min}, f_{max}]$,  with $f_{min}=1$ and $f_{max}<f_{nmax}$.
We illustrate in Figure \ref{figure:Richer_double_loop_real} the synthetic seismogram generated from this source function by model EL0M, with comparison to the original signal and those by the IDFT and L1M.  In Figure \ref{figure:Richer_double_loop_real},  all results of IDFT are obtained with frequencies sampled  from the interval $\left[1, 60\right]$, while the synthetic seismograms generated by L1M and EL0M are obtained with frequencies sampled from intervals $\left[1, f_{max}\right]$, where $f_{max}$ are $54$, $48$, $42$ and $36$, respectively.

\begin{figure}[h]
\centering
\begin{tabular}{cc}
\includegraphics[width=2.6in]{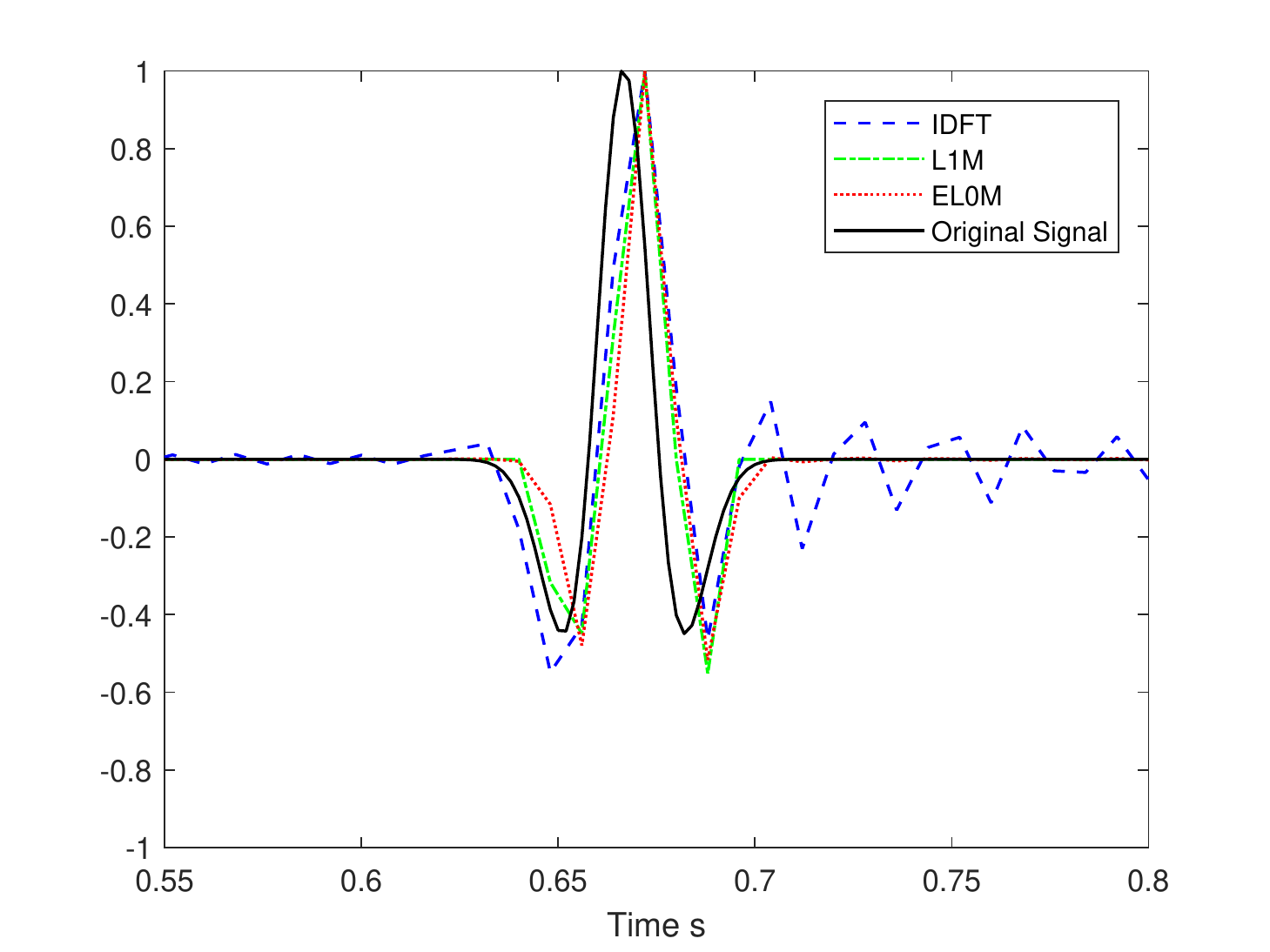}&\includegraphics[width=2.6in]{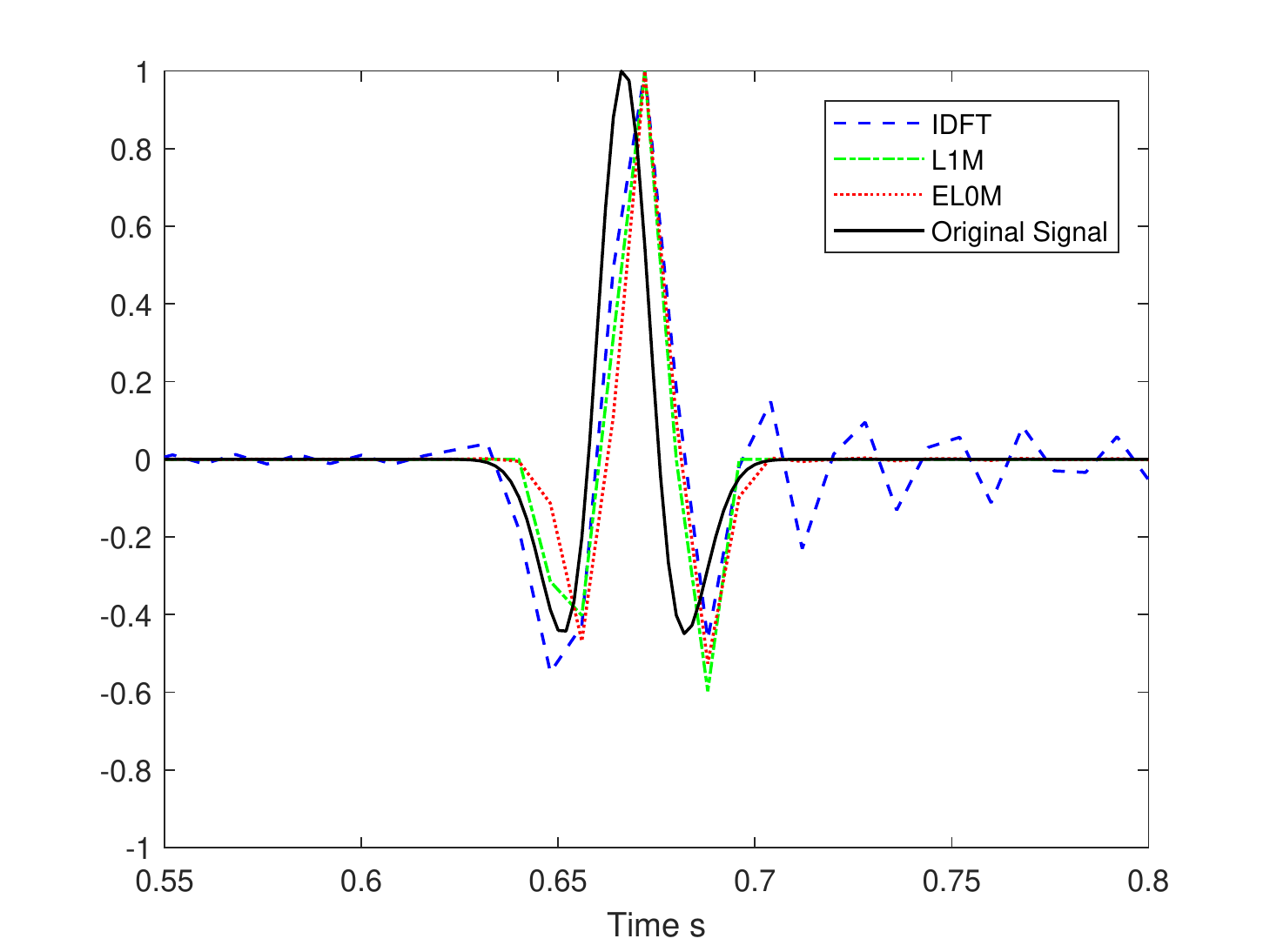}\\
(a)&(b)\\
\includegraphics[width=2.6in]{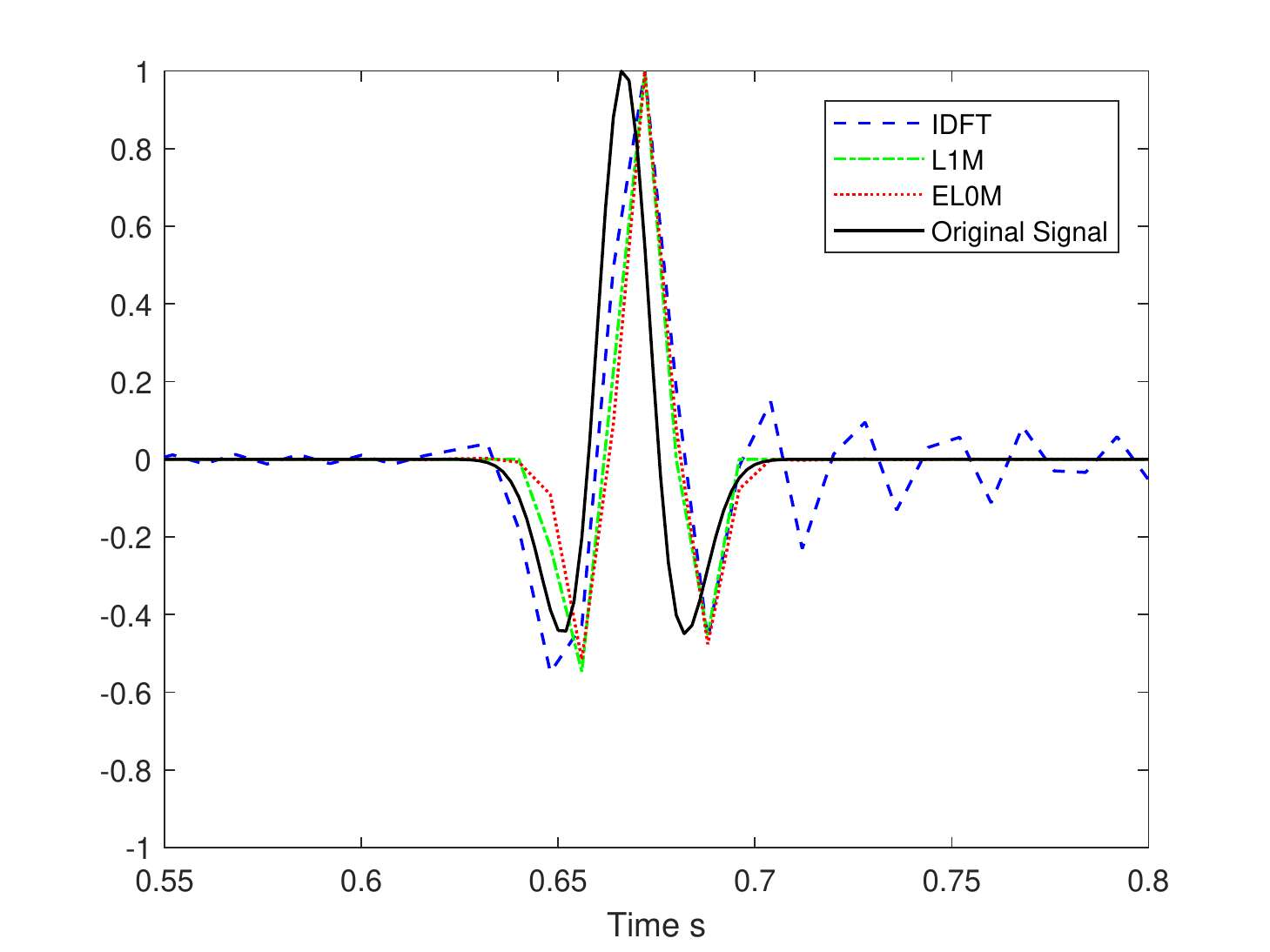}&\includegraphics[width=2.6in]{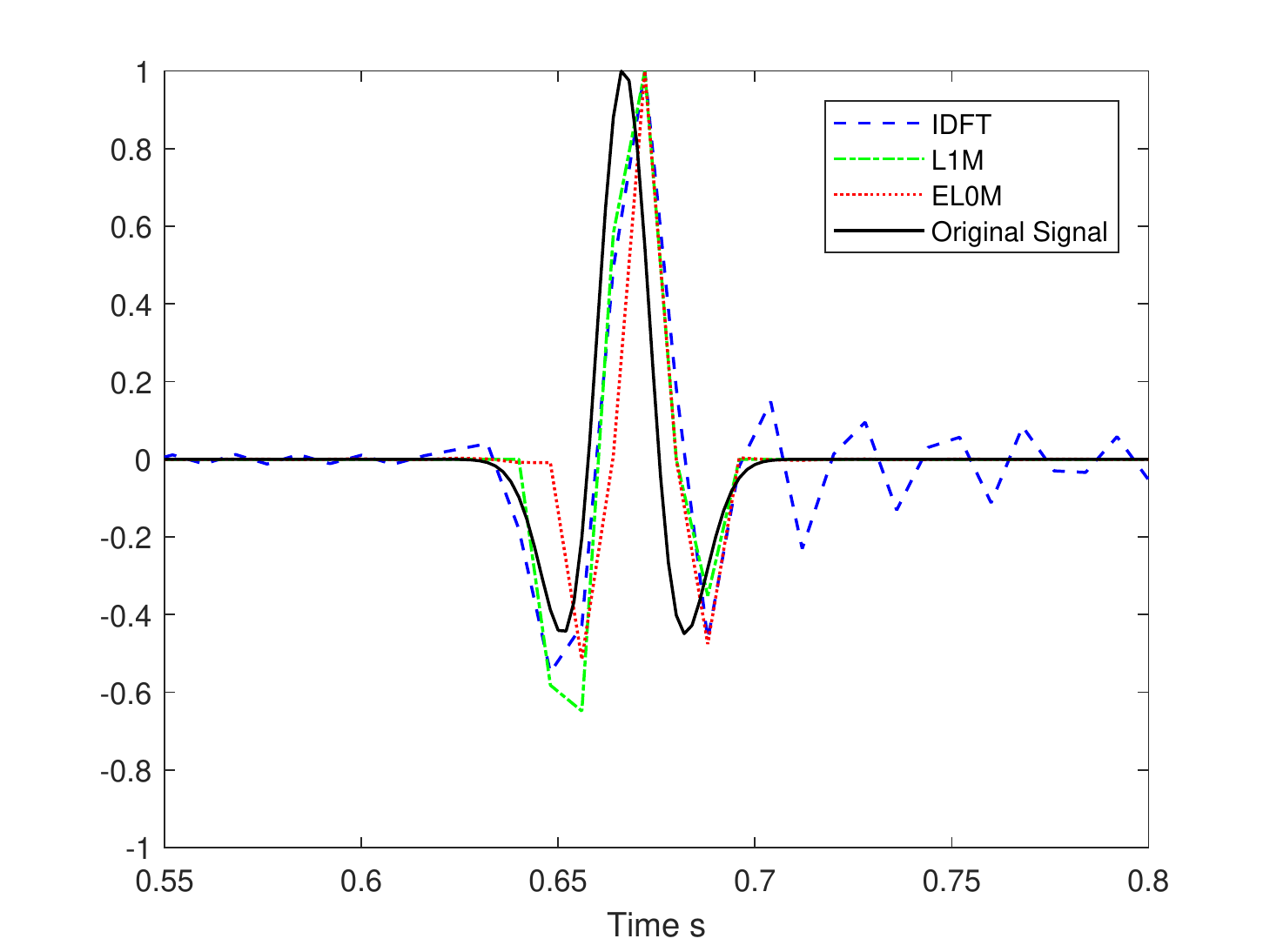}\\
(c)&(d)
\end{tabular}
\caption{ Synthetic seismograms generated by different methods for the homogeneous model with the Ricker wave as the source function: (a) $f_{max}=54$, (b) $f_{max}=48$,  (c) $f_{max}=42$,  (d) $f_{max}=36$.}
\label{figure:Richer_double_loop_real}
\end{figure}


Our second example considers $q(t):=G(t,t_0,\tilde{\alpha})$, where $G(t,t_0,\tilde{\alpha})$ is the first derivative of the Gaussian function defined by \eqref{Gaussian_time}, with $t_0=0.3,~\tilde{\alpha}=200$. The natural maximum frequency of the first derivative of the Gaussian function is approximately equal to $15$, that is, $f_{nmax}=15$. In this example, we choose $T:=2s$, $M:=129$, $l:=4$ and $f_{max}<f_{nmax}$. Specifically, we sample frequencies $f$ from intervals
$[f_{min}, f_{max}]$, with $f_{min}=0.5$ and $f_{max}=9,
7.5, 6, 4.5$. We illustrate in Figure  \ref{figure:Gaussian_double_loop_Real} the synthetic seismogram generated from this source function by model EL0M, with comparison to the original signal and those by the IDFT and L1M. In Figure \ref{figure:Gaussian_double_loop_Real},  all results of IDFT are obtained with frequencies sampled  from the interval $\left[0.5, 15\right]$, while the synthetic seismograms generated by L1M and EL0M are obtained with frequencies sampled from intervals $\left[0.5, f_{max}\right]$, where $f_{max}$ are $9$, $7.5$, $6$ and $4.5$, respectively.


\begin{figure}[h]
\centering
\begin{tabular}{cc}
\includegraphics[width=2.6in]{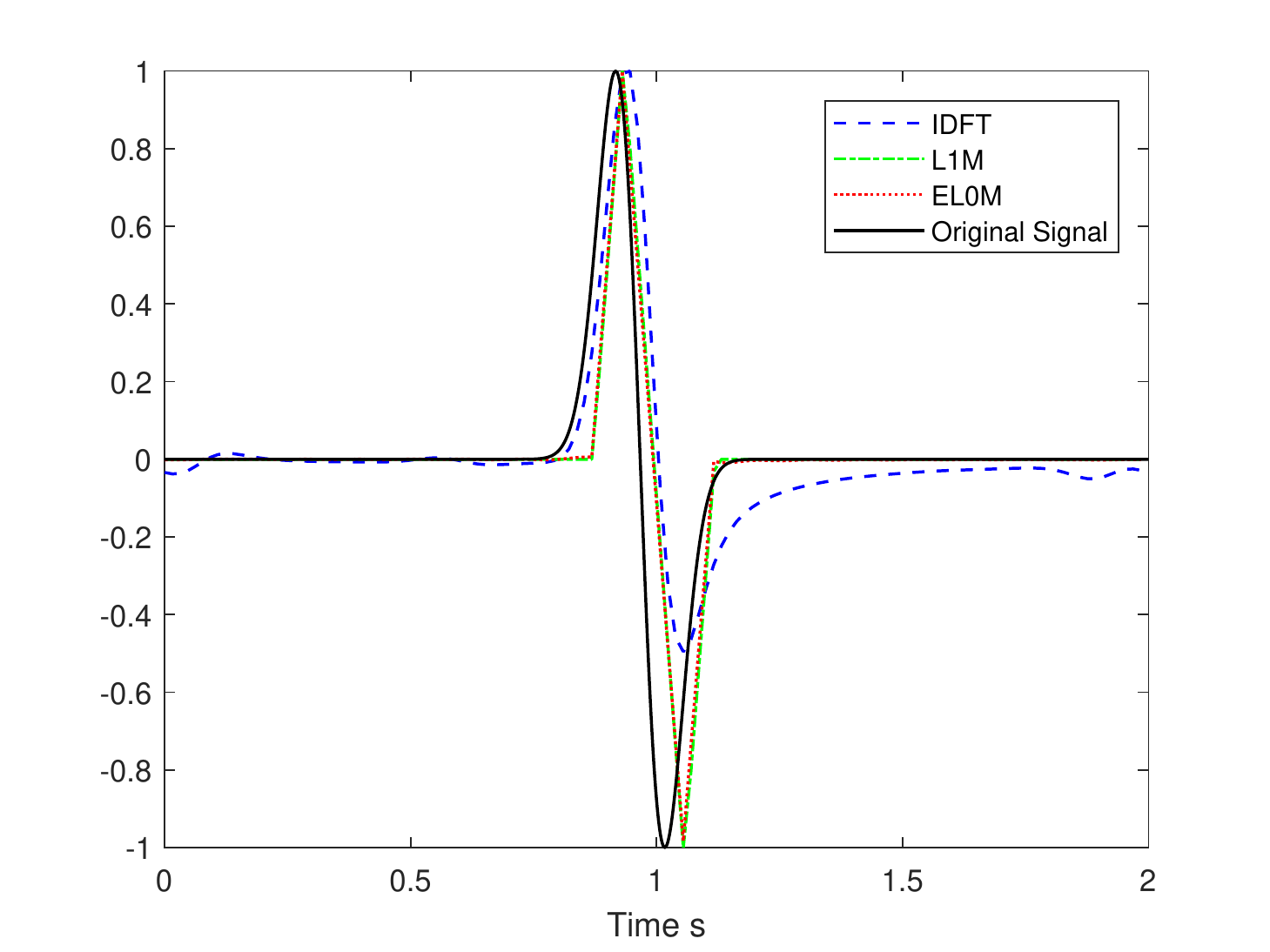}&\includegraphics[width=2.6in]{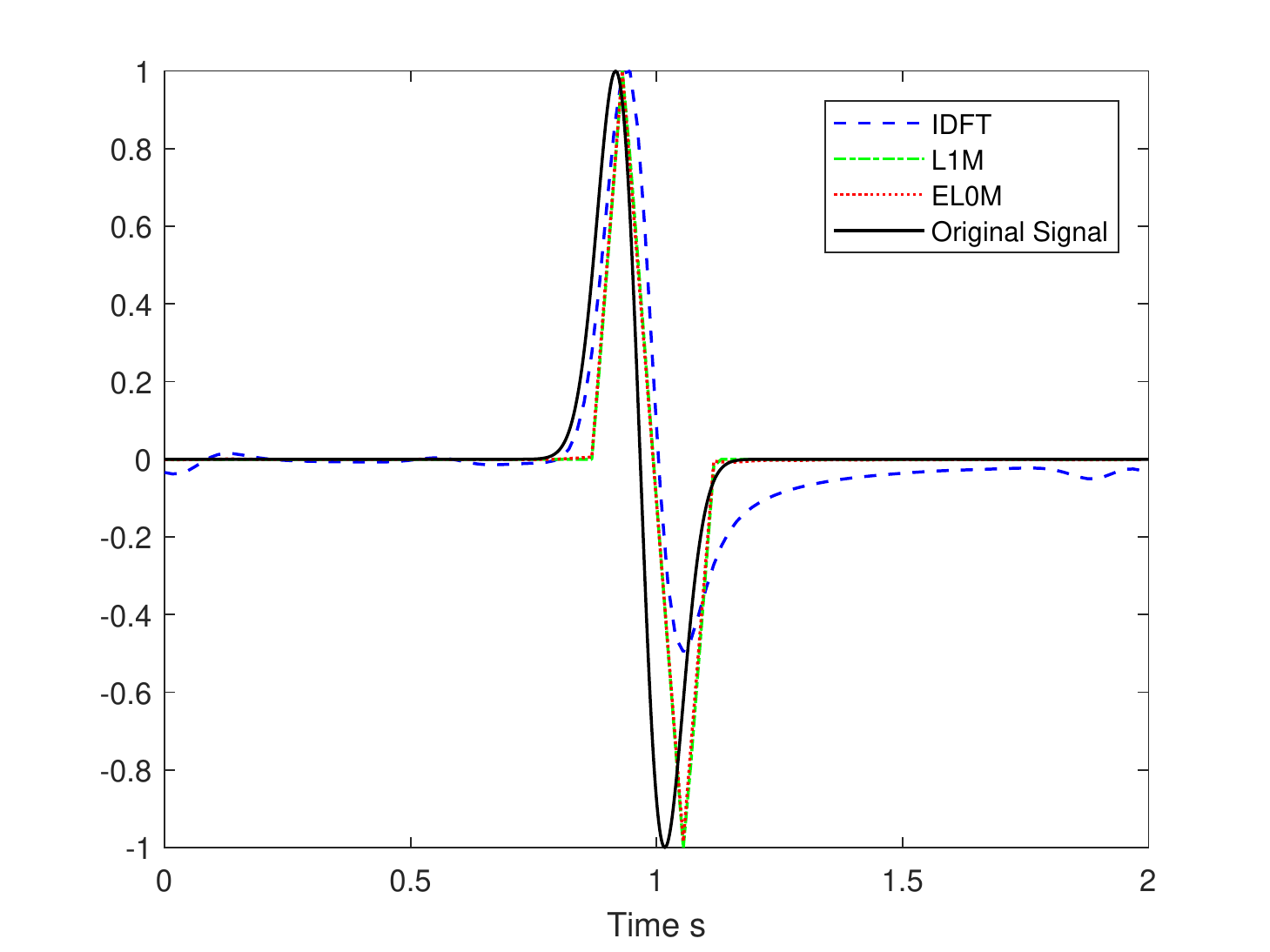}\\
(a)&(b)\\
\includegraphics[width=2.6in]{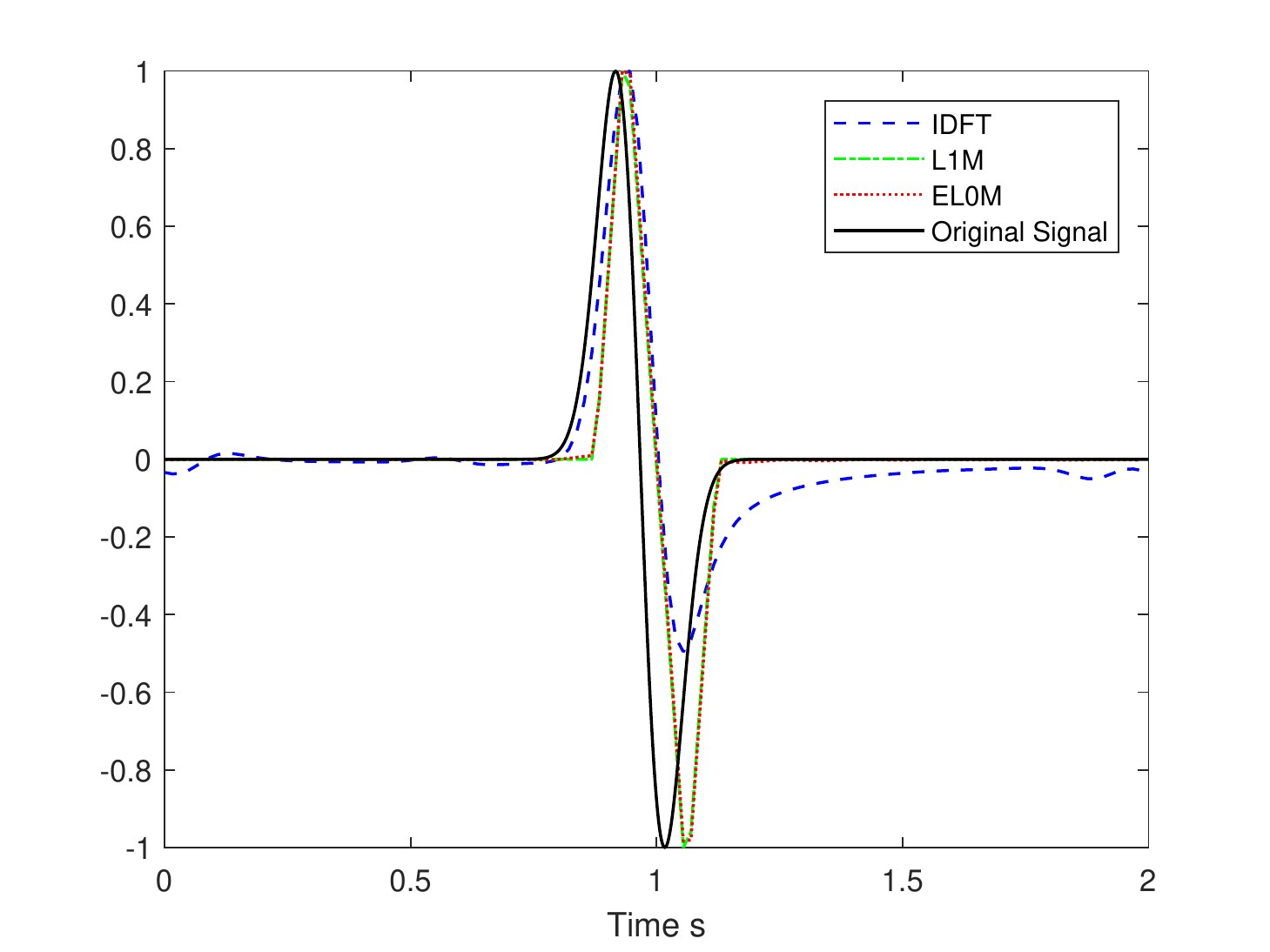}&\includegraphics[width=2.6in]{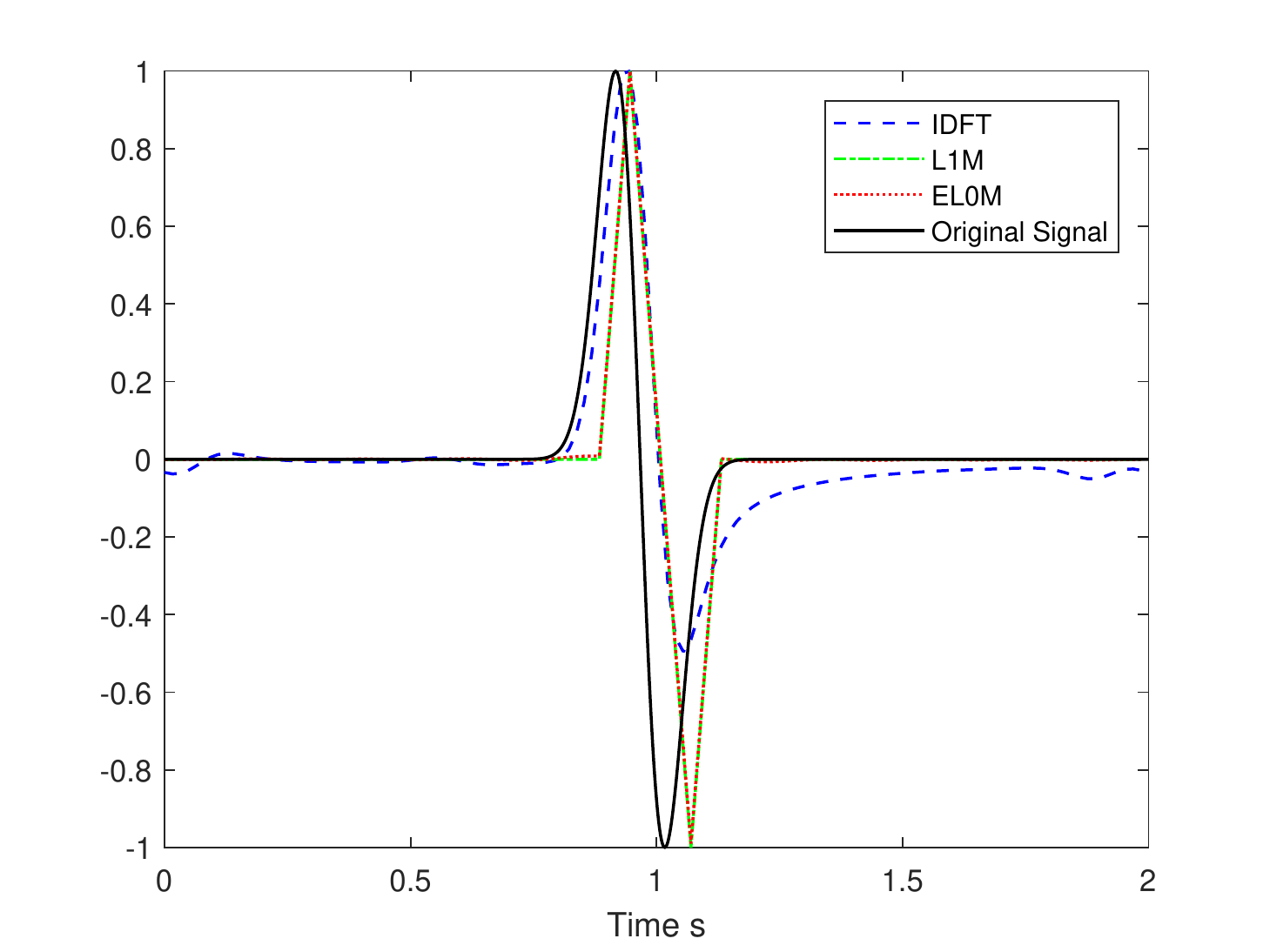}\\
(c)&(d)
\end{tabular}
\caption{ Synthetic seismograms generated by different methods for the homogeneous model with the first order  derivative of the Gaussian function as the source function: (a)$f_{max}=9$,
$\beta=0.73$, $\gamma=21.4757$, (b) $f_{max}=7.5$, $\beta=0.73$, $\gamma=21.4757$, (c) $f_{max}=6$, $\beta=0.75$, $\gamma=21.2899$,  (d) $f_{max}=4.5$, $\beta=0.75$, $\gamma=18.6696$.}
\label{figure:Gaussian_double_loop_Real}
\end{figure}

From  Figures \ref{figure:Richer_double_loop_real} and \ref{figure:Gaussian_double_loop_Real}, in both the examples, we find that even though only low frequencies are used, the best synthetic seismogram is obtained by EL0M. Specifically, the signals recovered by EL0M are much better than those by IDFT, as the IDFT creates many spurious oscillations in the recovered signals, and EL0M performs better than L1M.
In passing, we comment that  there are phase displacements between the original signal and each of the synthetic seismograms obtained by the IDFT, L1M and EL0M (see Figures \ref{figure:Richer_double_loop_real} and \ref{figure:Gaussian_double_loop_Real}), which is due to the difference between the numerical phase velocity and the exact velocity (see FIG 2 in \cite{J1}).

\subsection{The three layered velocity model}\label{subsec:Layered_Model}

In this subsection, we consider the three layered velocity model illustrated by Figure \ref{figure:The_velocity_model} (b) for generating common-shot-point records (shot profiles) with the source function being the same Ricker wavelet as in the last subsection. In other words, we will solve equation \eqref{wave_equation_time_domain}  in the heterogenous medium (the three layered velocity model) by the Frequency domain modeling.

Our interested domain is the same as that in the last subsection. The three layered velocity model is different from the homogeneous velocity model considered in Subsection \ref{subsec:Homo_Model_1}, as there are three velocities: $v=2,000m/s$, $2,500m/s$, $4,000m/s$, from the top to the bottom in this model.
The source function is located at the point $(x_s,z_s):=(0,1000)$,
and the receivers are located on the top ground, that is, they are located at points   $(x_j,0)$, where $x_j:=jh$ and $h:=10$ for $j=0,1,\ldots,200$. In this example, we choose $T:=2.2400s$, $M:=280$ and  $\Delta f=0.4464$. 
We generate the synthetic seismograms in the frequency domain by solving a sequence of 2D Helmholtz equations \eqref{Helmholtz_Equation} using the finite difference
method developed in \cite{CCFW}, with the grid size $\Delta x=\Delta z:=10$. We then invert the Fourier transform using model EL0M with Algorithm  \eqref{eq:fixed-eqs_two_var_iter} 
and obtain the common-shot-point records,  the image of  $u(x_j,0,t_k)$ for $j=0,1,\ldots,200$ and $t_k:=k\Delta t\in[0,T]$ with $\Delta t:=8\times10^{-3}s$.

\begin{figure}[h]
\centering
\begin{tabular}{ccc}
\includegraphics[width=2in]{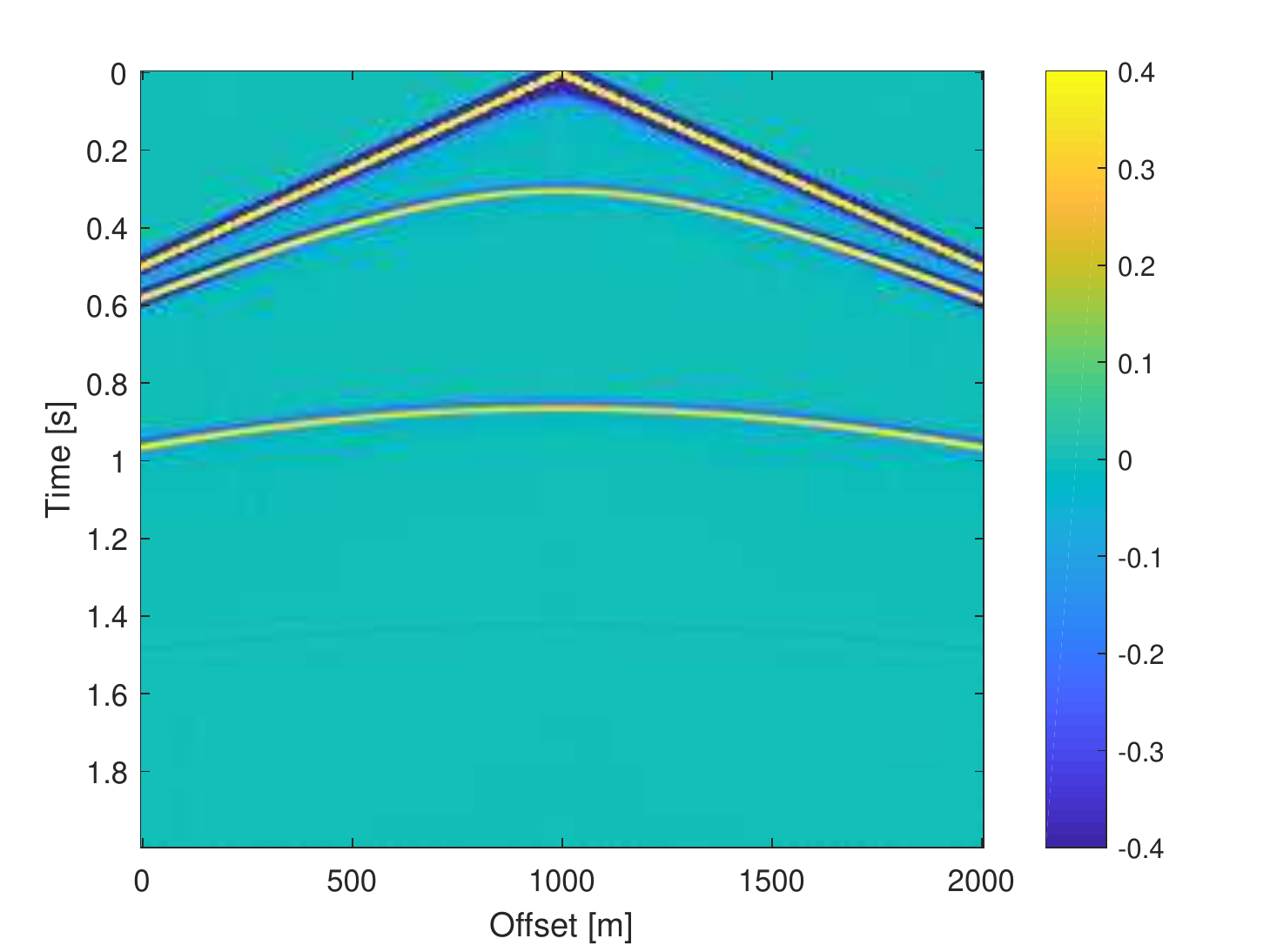}
&\includegraphics[width=2in]{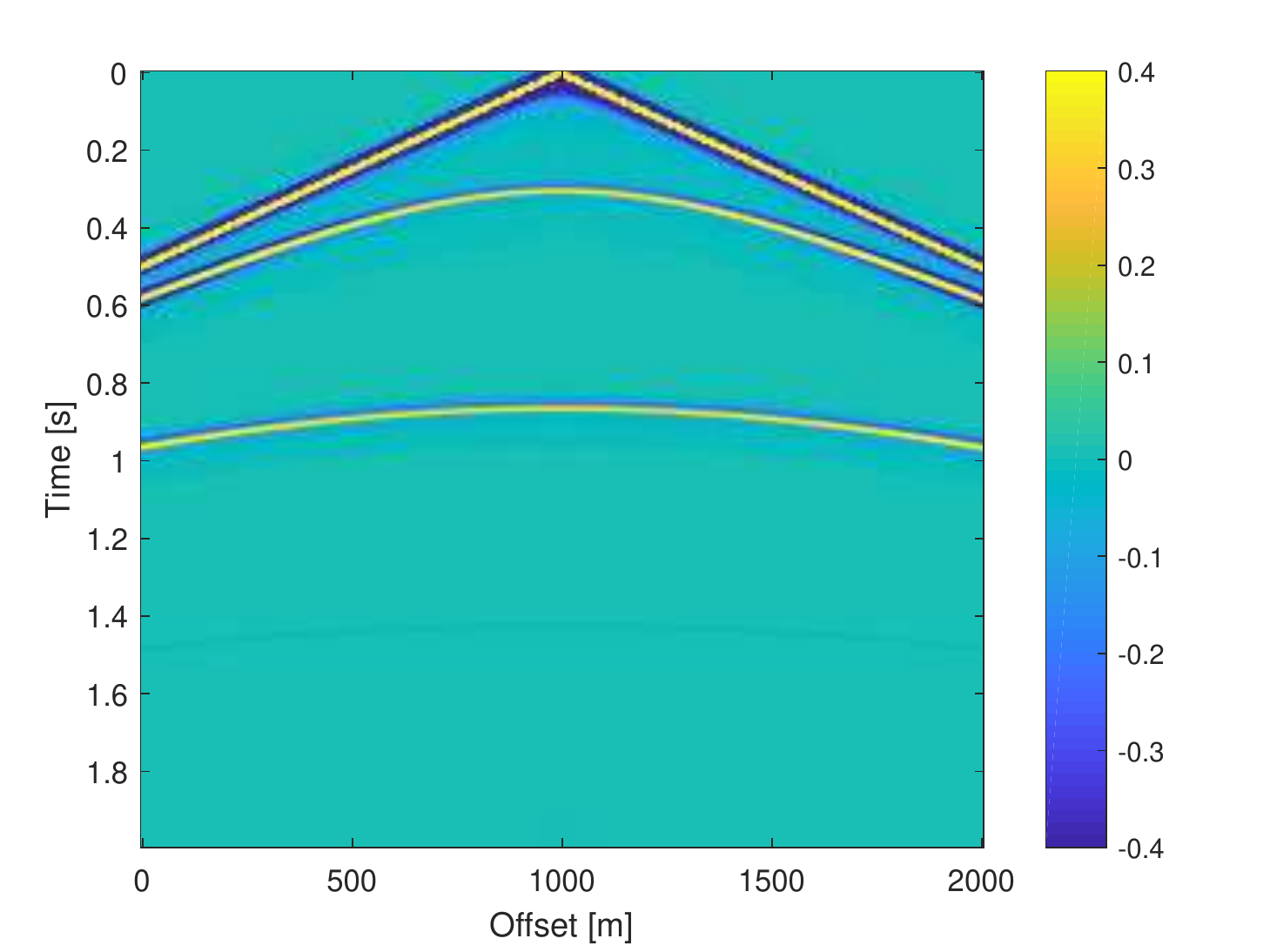}
&\includegraphics[width=2in]{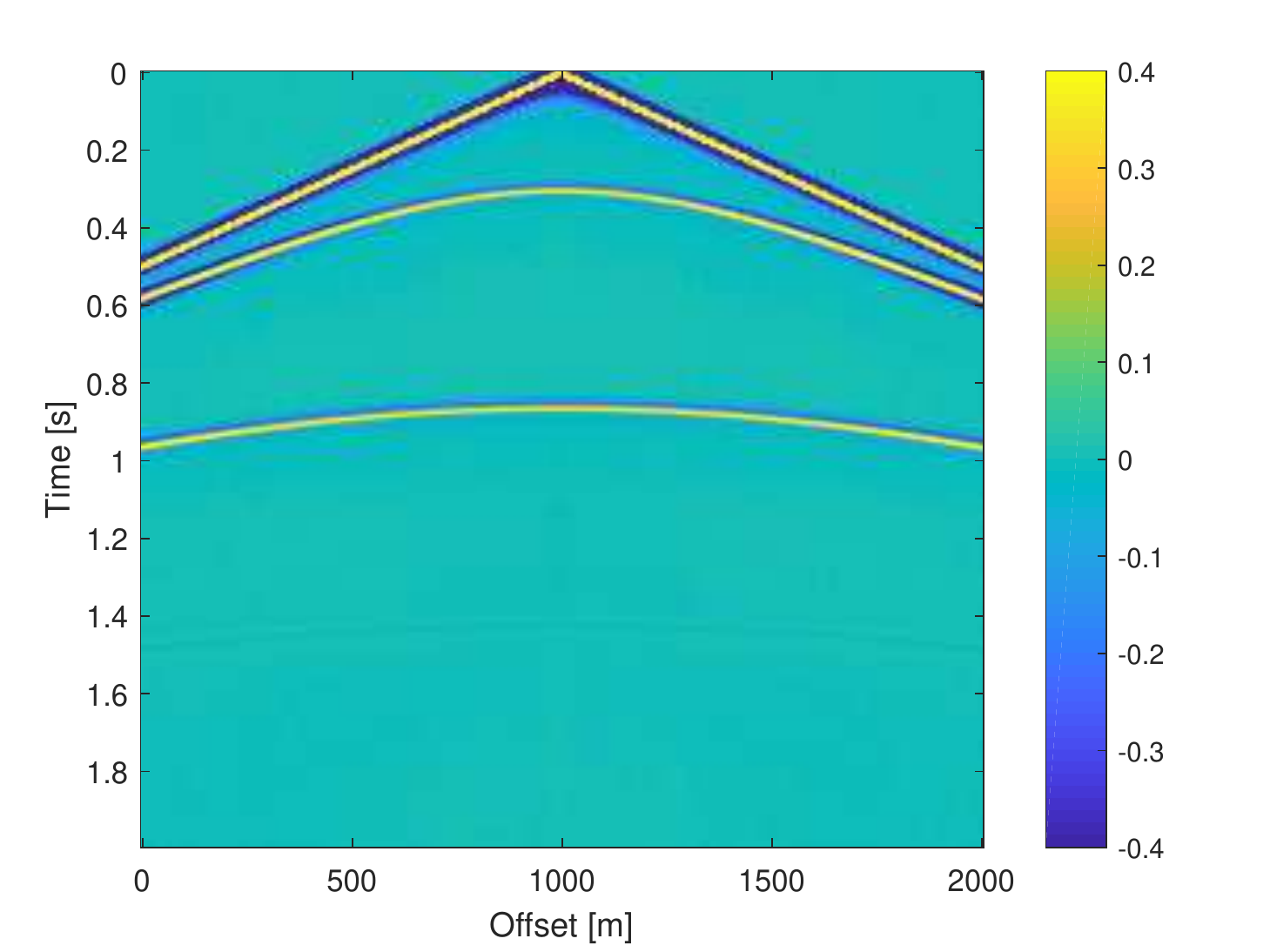}\\
(a)&(b)&(c)
\end{tabular}
\caption{ The common-shot-point records via different methods
with the frequency samples taken from $[1, 60]$:~(a) IDFT; (b)
L1M; (c) EL0M.} \label{figure:layered_model_60HZ}
\end{figure}

\begin{figure}[h]
\centering
\begin{tabular}{ccc}
\includegraphics[width=2in]{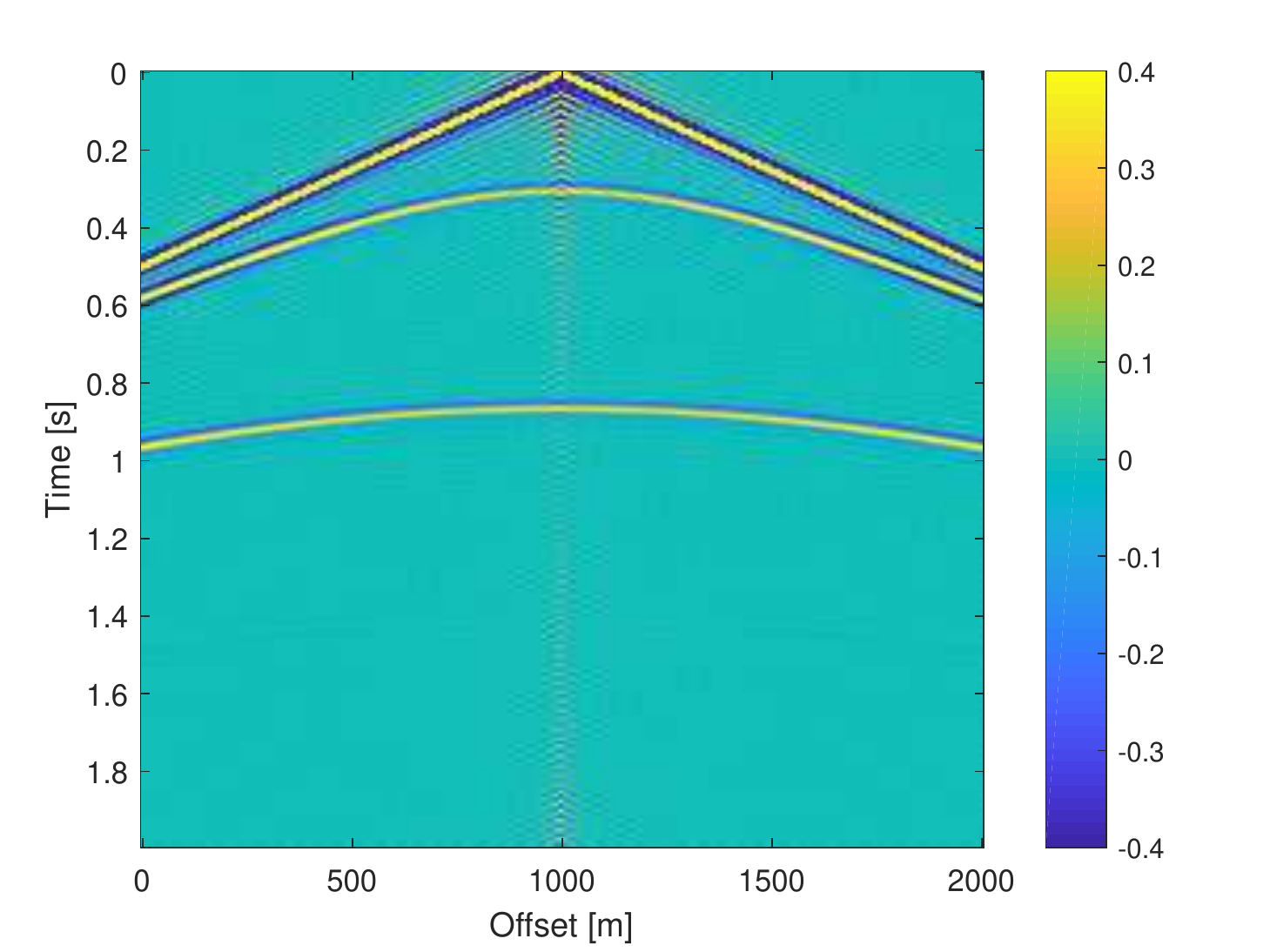}
&\includegraphics[width=2in]{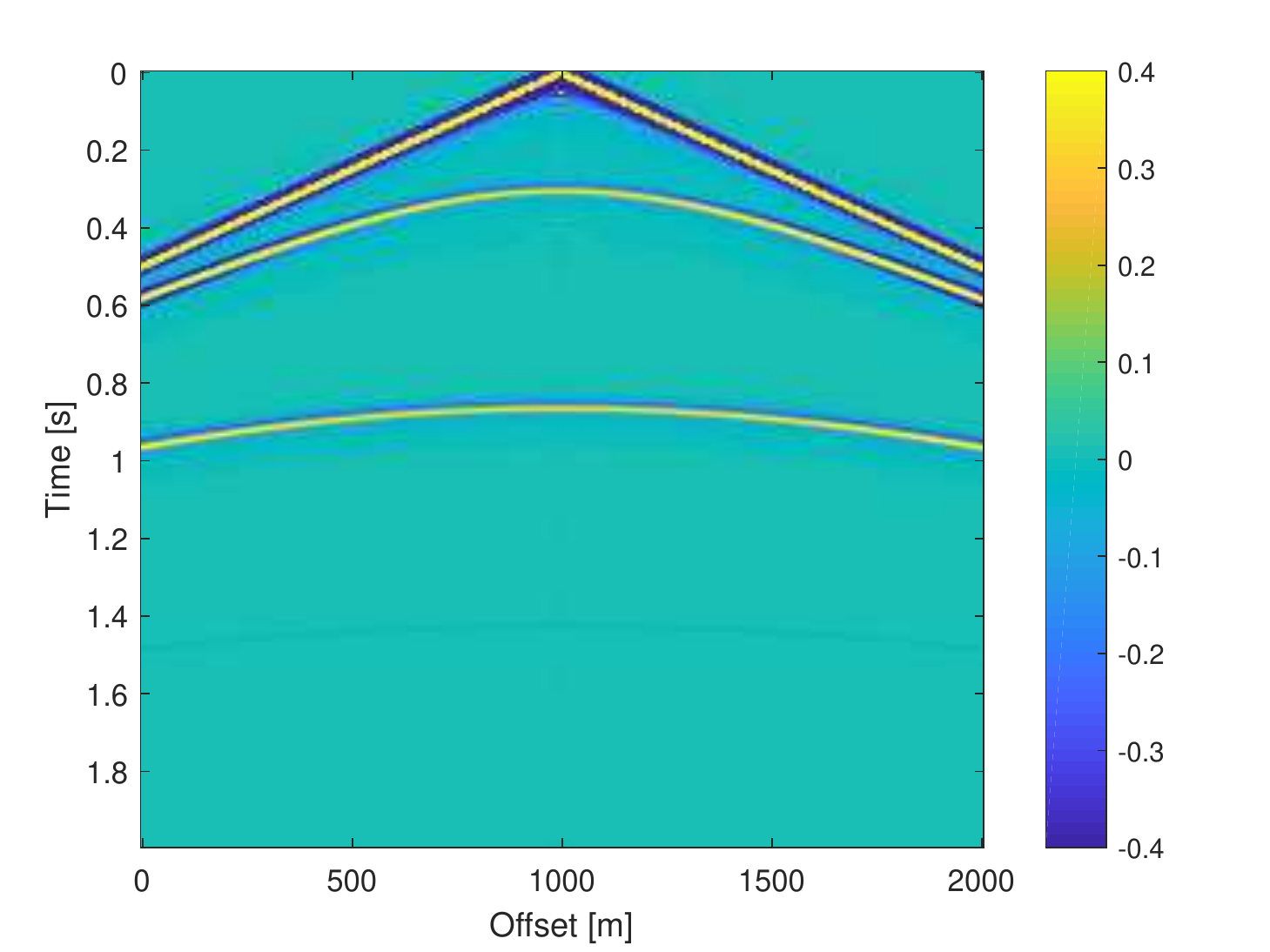}
&\includegraphics[width=2in]{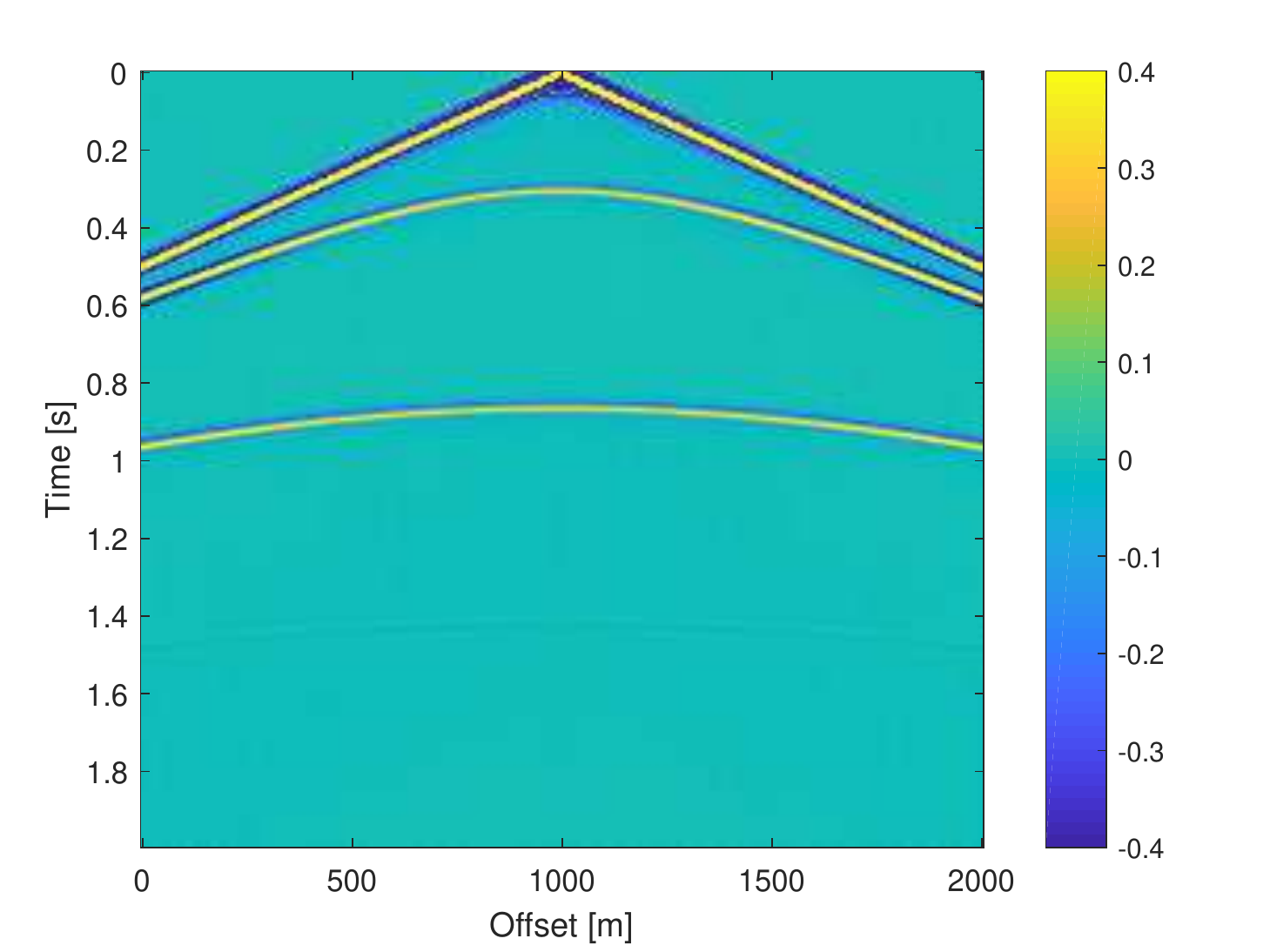}\\
(a)&(b)&(c)
\end{tabular}
\caption{ The common-shot-point records via different methods
with the frequency samples taken from $[1, 42]$:~(a) IDFT; (b)
L1M; (c) EL0M.} \label{figure:layered_model_42HZ}
\end{figure}

\begin{figure}[h]
\centering
\begin{tabular}{ccc}
\includegraphics[width=2in]{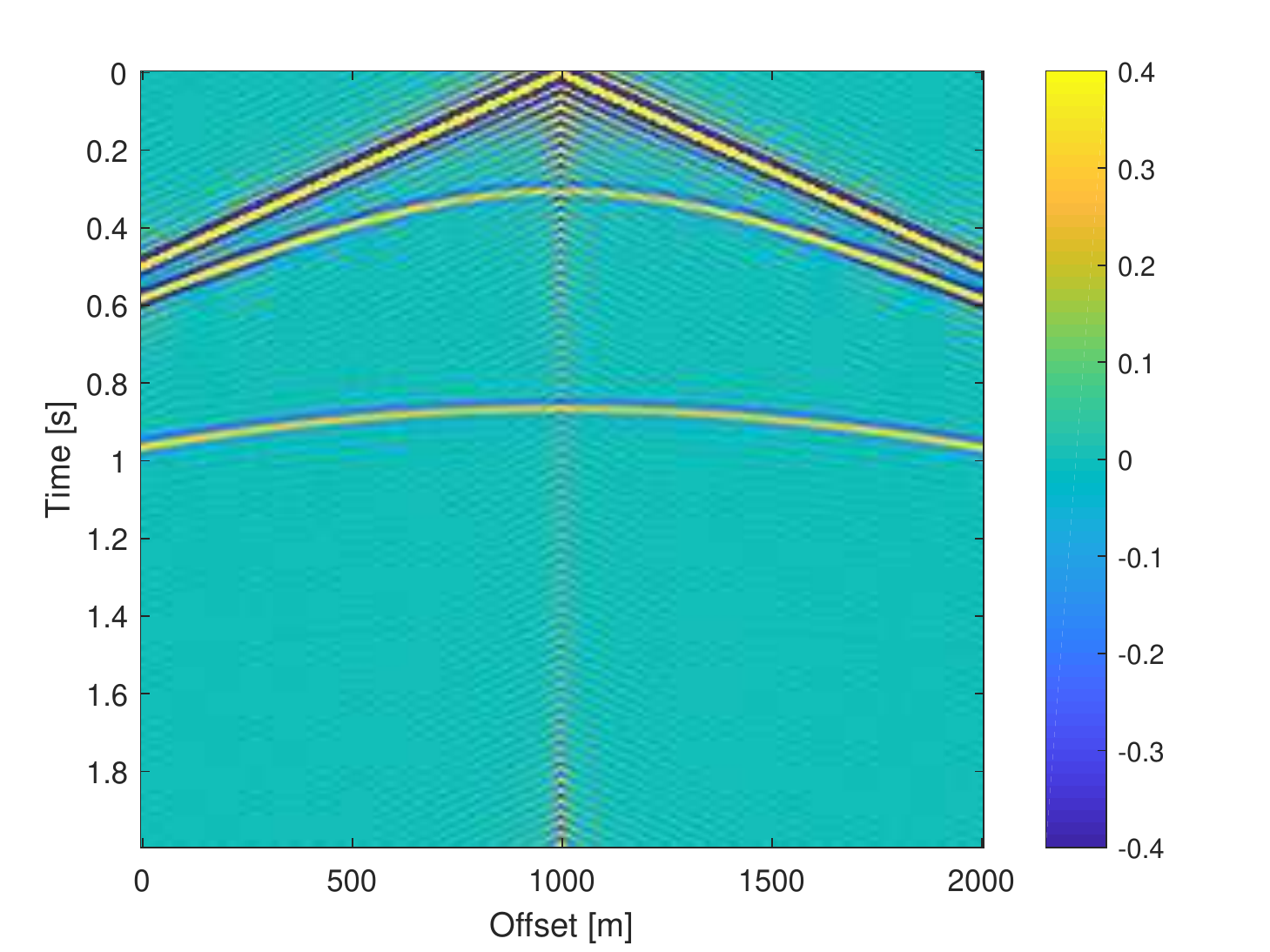}
&\includegraphics[width=2in]{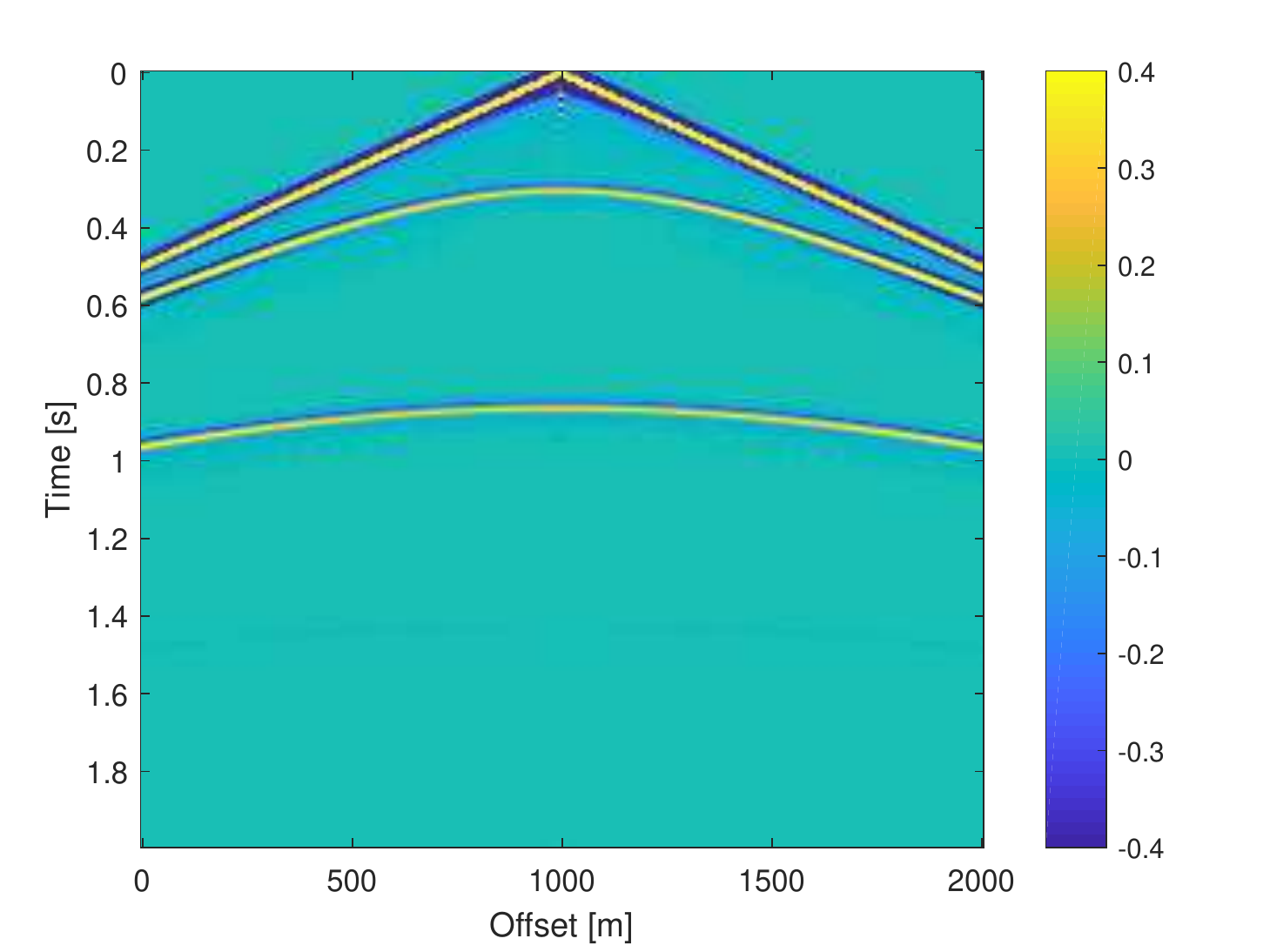}
&\includegraphics[width=2in]{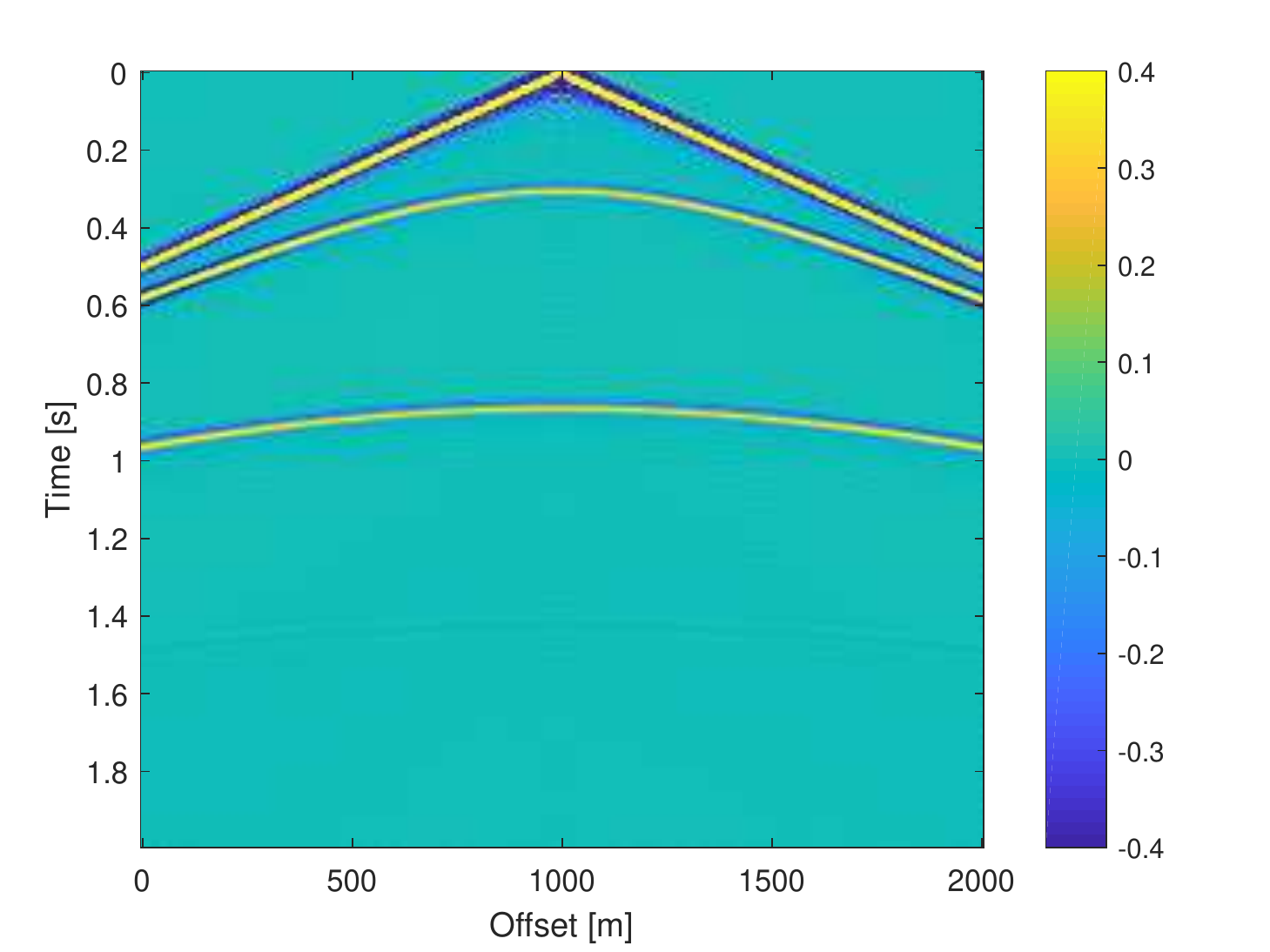}\\
(a)&(b)&(c)
\end{tabular}
\caption{ The common-shot-point records via different methods
with the frequency samples taken from $[1, 36]$:~(a) IDFT; (b)
L1M; (c) EL0M.} \label{figure:layered_model_36HZ}
\end{figure}

\begin{figure}[h]
\centering
\begin{tabular}{ccc}
\includegraphics[width=2in]{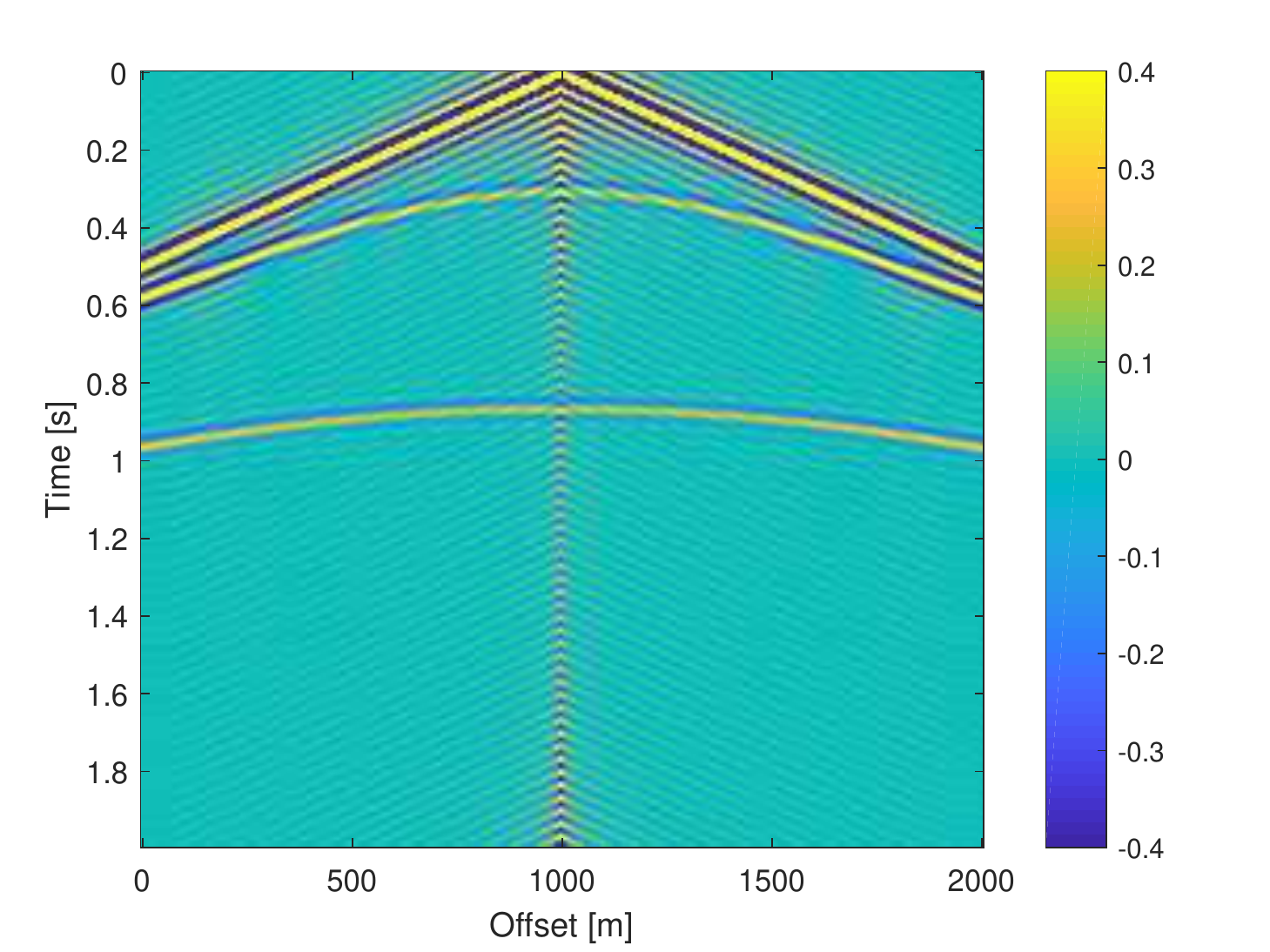}
&\includegraphics[width=2in]{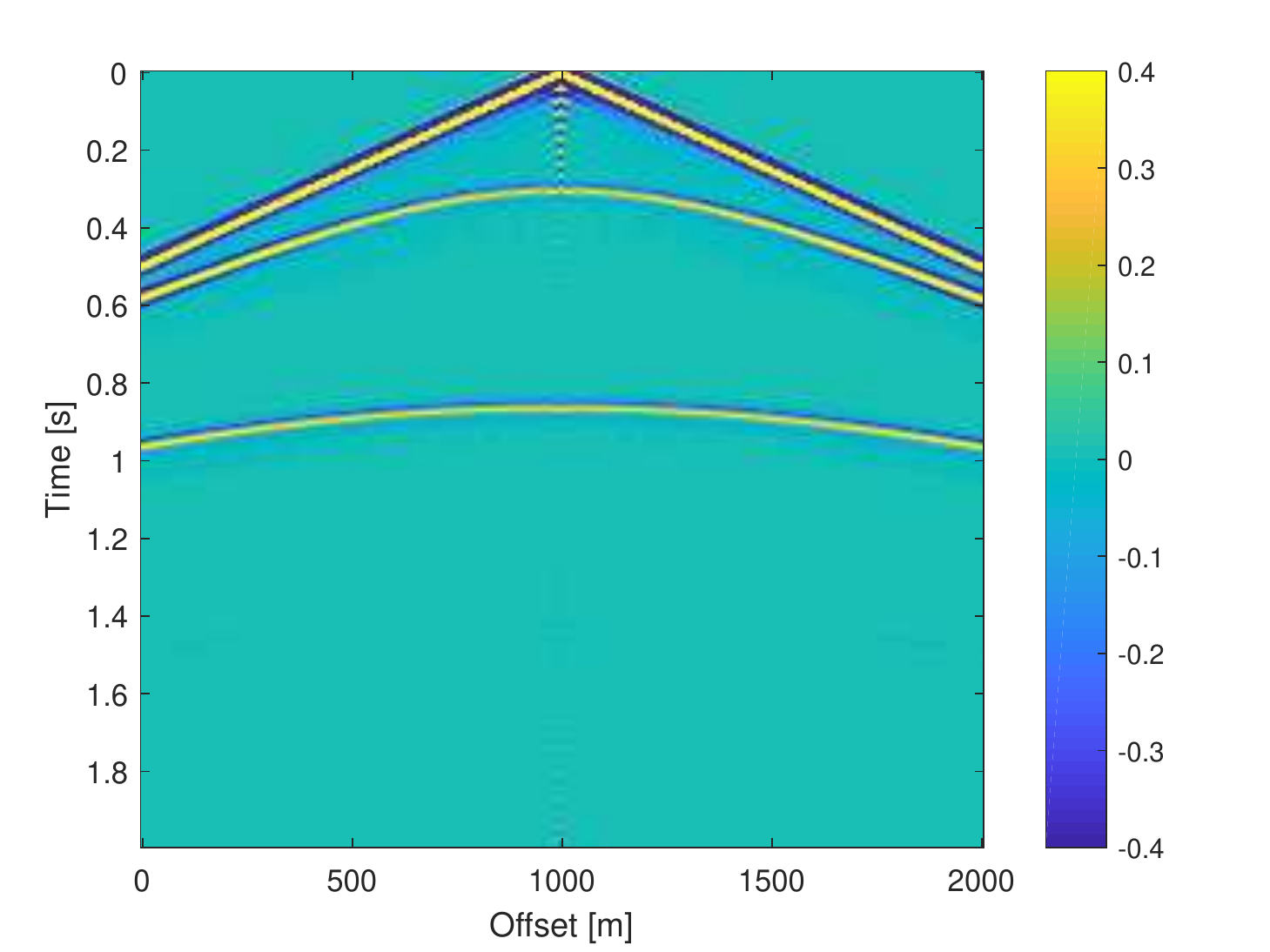}
&\includegraphics[width=2in]{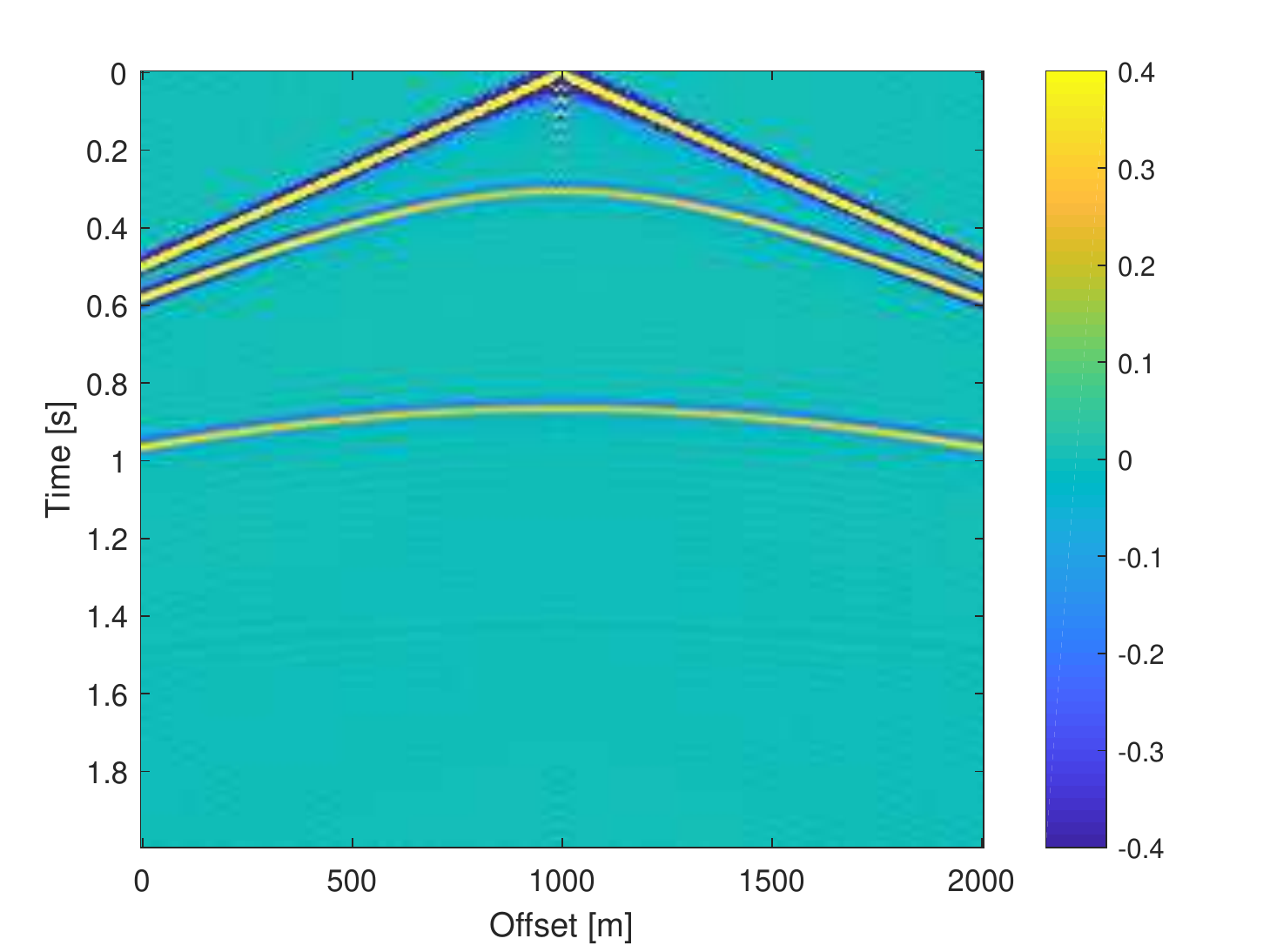}\\
(a)&(b)&(c)
\end{tabular}
\caption{ The common-shot-point records via different methods
with the frequency samples taken from $[1, 30]$:~(a) IDFT; (b)
L1M; (c) EL0M.} \label{figure:layered_model_30HZ}
\end{figure}

The common-shot-point records obtained by EL0M are compared to those by IDFT and L1M.  Figures \ref{figure:layered_model_60HZ},  \ref{figure:layered_model_42HZ}, \ref{figure:layered_model_36HZ} and  \ref{figure:layered_model_30HZ} present the common-shot-point records with frequencies sampled from intervals $[1, 60]$, $[1, 42]$, $[1, 36]$ and $[1, 30]$, respectively.
Note that all of the three numbers $42$, $36$, $30$ are smaller than $f_{nmax}:=60$ required by the Nyquist sampling theorem.
From part (a) of Figures \ref{figure:layered_model_42HZ},  \ref{figure:layered_model_36HZ}
and \ref{figure:layered_model_30HZ}, we find that nonphysical oscillations appear in the seismic wavefields obtained by the IDFT, as the Nyquist-Shannon criterion are not satisfied, and as $f_{max}$ reduces, the oscillations become stronger. As shown in parts (b) and (c) of Figures \ref{figure:layered_model_42HZ},  \ref{figure:layered_model_36HZ}
and \ref{figure:layered_model_30HZ}, the direct waves of the source, the reflected waves of the top side of the second layer and the reflected waves of the bottom side of the second layer are displayed clearly in the seismic wavefields obtained by both L1M and EL0M. Furthermore, waves obtained by EL0M are clearer than those by L1M, as less nonphysical oscillations appear in part (c) of these figures. These demonstrate that EL0M outperforms  L1M and frequencies sampled from $[1, 30]$ with $\Delta f=0.4464$ are enough to restore the seismic wavefield by EL0M, which confirms the effectiveness of the proposed method.

\section{Conclusions}\label{sec:Conclusion}\setcounter{equation}{0}
We have developed a sparse regularization model based on the Moreau envelope of the $\ell_0$ norm under a tight framelet system for inversion of incomplete Fourier transforms and a fixed-point iteration algorithm to solve the model. We have also applied this proposed method to seismic wavefield modeling. We have established that the proposed fixed-point algorithm converges to a local minimizer of the non-convex, non-smooth model. Numerical results have verified that the proposed model outperforms significantly the model based on the $\ell_1$ norm. In the context of the seismic wavefield modeling, substantial numerical studies that we have conducted show that the proposed method, which requires data of only a few low frequencies and avoids solving the Helmholtz equations with large wave numbers, performs better than the method based on the $\ell_1$ norm, in terms of the SNR values and visual quality of the restored synthetic seismograms. They confirm that the proposed model is particularly suitable for the seismic wavefield modeling.

The proposed inverting incomplete Fourier transform method may be applicable to other applications such as MRI and seismic data restoration. Some MRI images such as angiograms are already sparse in the pixel representation, and more complicated images may have a sparse representation in some transform domain, for example, in terms of their wavelet coefficients. The paper
\cite{LDP} performed the reconstruction of sparse MRI by minimizing the $\ell_1$ norm of a transformed image, subject to data
fidelity constraints. The proposed model \eqref{model:general} with $\mathrm{env}_{\beta\|\cdot\|_0}$ as a measure of sparsity may also work well in this application. In addition, seismic data restoration is a useful tool in seismic exploration, and it is an ill-posed inverse problem. Due to the sparsity of seismic data in some transform domain, this problem can be transformed into a sparse optimization problem. Thus, our proposed method is also expected to work efficiently for this problem.

Finally, we comment that machine learning methods may be employed to train data driven filters for regularization when training data are available.  We will consider it as our future projects for applications in which training data are available.


\end{document}